\documentclass[11pt]{article}
\usepackage[utf8]{inputenc}

\usepackage{amsmath,amssymb,amsthm,amssymb} 
\usepackage{xcolor}
\usepackage[sort&compress,numbers]{natbib}
\usepackage[T1]{fontenc}
\usepackage{geometry}
\usepackage{array}
\usepackage{url}
\usepackage{float}
\usepackage{enumitem}
\usepackage{booktabs}
\usepackage{tablefootnote}
\usepackage{titlesec}
\usepackage{caption}
\usepackage[title]{appendix}
\usepackage{xspace}

\newcolumntype{P}[1]{>{\centering\arraybackslash}p{#1}}

\global\long\def\RR{{\mathbb R}}
\global\long\def\CC{{\mathbb C}}
\global\long\def\FF{{\mathbb F}}

\renewcommand{\subset}{\subseteq}

\DeclareMathOperator{\tr}{tr}
\DeclareMathOperator*{\argmin}{arg\,min}

\definecolor{green}{rgb}{0,0.6,0.0}

\usepackage{amsmath}
\usepackage{amsfonts}
\usepackage{latexsym}
\usepackage{amssymb}

\newtheorem{thm}{Theorem}[section]

\newtheorem{lem}[thm]{Lemma}
\newtheorem{lemma}[thm]{Lemma}
\newtheorem{prop}[thm]{Proposition}
\newtheorem{proposition}[thm]{Proposition}

\newtheorem{corollary}[thm]{Corollary}
\newtheorem{definition}[thm]{Definition}

\newtheorem{rem}[thm]{Remark}
\newtheorem{remark}[thm]{Remark}

\newcommand{\beq}{\begin{equation}}
\newcommand{\eeq}{\end{equation}}
\newcommand{\beqa}{\begin{eqnarray}}
\newcommand{\eeqa}{\end{eqnarray}}
\newcommand{\beqas}{\begin{eqnarray*}}
\newcommand{\eeqas}{\end{eqnarray*}}
\newcommand{\ei}{\end{itemize}}

\newcommand{\vgap}{\vspace{.1in}}

\newcommand{\R}{\mathbb{R}}

\newcommand{\lam}{{\lambda}}

\newcommand{\inner}[2]{\langle #1,#2\rangle}

\newcommand{\dom}{\mathrm{dom}\,}

\newcommand{\tx}{\tilde x}

\newcommand{\ourmethod}{HALLaR\xspace}

\titleformat{\section}
  {\normalfont\fontsize{14}{15}\bfseries}{\thesection}{1em}{}
  \titleformat{\subsection}
  {\normalfont\fontsize{12}{15}\bfseries}{\thesubsection}{1em}{}

\newtheorem{assumption}{Assumption}[section]

\newtheorem{innercustomgeneric}{\customgenericname}
\providecommand{\customgenericname}{}
\newcommand{\newcustomtheorem}[2]{%
  \newenvironment{#1}[1]
  {%
   \renewcommand\customgenericname{#2}%
   \renewcommand\theinnercustomgeneric{##1}%
   \innercustomgeneric
  }
  {\endinnercustomgeneric}
}
\newcustomtheorem{blackbox}{Blackbox}

\title{A Low-Rank Augmented Lagrangian Method \\ for Large-Scale Semidefinite Programming \\
Based on a Hybrid Convex-Nonconvex
Approach}
\author{Renato D.C. Monteiro \thanks{Stewart School of Industrial and Systems Engineering, Georgia Institute of Technology, Atlanta, GA, 30332-0205. (Email: {\tt monteiro@isye.gatech.edu} \& {\tt asujanani6@gatech.edu}). These authors were partially supported by AFORS Grant FA9550-22-1-0088.} \and 
Arnesh Sujanani \footnotemark[1] \and Diego Cifuentes \thanks{Stewart School of Industrial and Systems Engineering, Georgia Institute of Technology, Atlanta, GA, 30332-0205. (Email: {\tt diego.cifuentes@isye.gatech.edu}). This author was supported partially supported by the Office of Naval Research, N00014-23-1-2631.} }

\date{January 22, 2024 (second version: March 15, 2024)}

\begin{document}
\maketitle

\begin{abstract}
   This paper introduces \ourmethod,
   a new first-order method for solving large-scale semidefinite programs (SDPs) with bounded domain.
   \ourmethod is an inexact augmented Lagrangian (AL) method where the AL subproblems
   are solved by a novel hybrid low-rank (HLR) method.
   The recipe behind HLR is based on two key ingredients:
   1) an adaptive inexact proximal point method with inner acceleration;
   2) Frank-Wolfe steps to escape from spurious local stationary points.
   In contrast to the low-rank method of Burer and Monteiro,
   \ourmethod finds a near-optimal solution (with provable complexity bounds)
   of SDP instances satisfying strong duality.
   Computational results comparing \ourmethod to state-of-the-art solvers on several large SDP instances
   arising from 
   maximum stable set, phase retrieval, and matrix completion,
   show that the former finds highly accurate solutions in substantially less CPU time than the latter ones.
   For example, in less than 20 minutes, \ourmethod can solve a maximum stable set SDP instance with dimension pair $(n,m)\approx (10^6,10^7)$ within 
   $10^{-5}$ relative precision.

   \vgap


\end{abstract}












\section{Introduction}\label{Introduction}

\emph{Semidefinite programming} (SDP) has many applications in engineering, machine learning, sciences, finance, among other areas.
However, solving large-scale SDPs is very computationally challenging.
In particular, interior point methods usually get stalled in large-scale instances due to lack of memory.
This has motivated a recent surge of first-order methods for solving SDPs
that scale to larger instances
\cite{zhao2010newton, yang2015sdpnal+, odonoghue2016conic, garstka2021cosmo, deng2022new, zheng2017fast, shinde2021memory, yurtsever2019conditional, yurtsever2021scalable, renegar2019accelerated}.

This paper introduces \ourmethod,
a new first-order method for solving SDPs with bounded trace.
Let $\mathbb S^n$ be the space of symmetric $n\times n$ matrices
with Frobenius inner product $\bullet$
and with positive semidefinite partial order $\succeq$.
\ourmethod solves the primal/dual pair of SDPs:
\begin{gather*}
\label{eq:sdp-primal}\tag{P}
\min_{X} \quad \{C \bullet X
\quad : \quad
\mathcal A X=b ,\quad
X  \in \Delta^n \}
\end{gather*}
\begin{gather*}
\label{eq:sdp-dual}\tag{D}
\max_{p \in \RR^m, \theta \in \RR}   \quad \{ -b^{T}p-\theta
\quad : \quad
S := C+\mathcal A^{*}p+\theta I \succeq 0, \quad
\theta \geq 0 \}
\end{gather*}
where $b \in \RR^m$, $C\in \mathbb S^n$, $\mathcal A: \mathbb S^n \to \RR^m$ is a linear map,
$\mathcal A^*: \RR^m \to \mathbb S^n$ is its adjoint,
and $\Delta^n$ is the spectraplex
\begin{equation}\label{Delta Definition}
\Delta^n := \{ X \in \mathbb S^n : \tr X \leq 1, X \succeq 0 \}.
\end{equation}
\ourmethod is based on Burer and Monteiro's \emph{low-rank} (LR) approach \cite{burer2003nonlinear, burer2005local} which is described in the next paragraph.

\paragraph{Low-rank approach.}
The LR approach is motivated by SDPs often having optimal solutions with small ranks.
More specifically, it is known
(see \cite{pataki1998rank, barvinok1995problems, shapiro1982rank})
that $r_* \leq \sqrt{2m}$,
where $r_*$ is the smallest among the ranks of all optimal solutions of \eqref{eq:sdp-primal}.
The LR approach consists of solving subproblems obtained by restricting \eqref{eq:sdp-primal} to matrices of rank at most~$r$, for some integer $r$, or equivalently,
the nonconvex smooth reformulation
\begin{align}\label{eq:sdp-lr}
   \tag{$P_r$}
   \min_U\quad
    \left \{ 
    C \bullet U U^T \quad:\quad \mathcal{A}(U U^T) = b, \quad
    \|U\|_F \leq 1, \quad
    U \in \RR^{n\times r} \right\}.
\end{align}

Problems \eqref{eq:sdp-primal} and \eqref{eq:sdp-lr} are equivalent when $r \geq r_*$,
in the sense that if $U_*$ is optimal for \eqref{eq:sdp-lr} then $X_* = U_* U_*^T$ is optimal for \eqref{eq:sdp-primal}.
The advantage of \eqref{eq:sdp-lr} 
compared to \eqref{eq:sdp-primal} is that its matrix variable $U$
has significantly less entries than that of \eqref{eq:sdp-primal} when $r \ll n$, namely,
$n r$ instead of $n(n+1)/2$.
However, as \eqref{eq:sdp-lr} is nonconvex,
it may have stationary points which are not globally optimal.
For a generic instance,
the following results are known:
i)
if $r \geq \sqrt{2m}$ then all local minima of \eqref{eq:sdp-lr} are globally optimal
(see  \cite{cifuentes2021burer, cifuentes2022polynomial, boumal2016non, boumal2020deterministic, bhojanapalli2018smoothed, pumir2018smoothed});
and ii)
if $r < \sqrt{2m}$
then \eqref{eq:sdp-lr} may have local minima which are not globally optimal (see~\cite{waldspurger2020rank}).

\paragraph{Outline of \ourmethod.} 
\ourmethod is an inexact augmented Lagrangian (AL) method
that generates sequences $\{X_t\}$ and $\{p_t\}$ according to the recursions
\begin{subequations}
\begin{align}
    \label{eq:AL}
    X_{t} \quad &\approx \quad
    \argmin_X\quad
    \{\mathcal{L}_{\beta}(X;p_{t-1}) \;:\; X \in \Delta^n\},
    \\
    p_{t} \quad &= \quad
    p_{t-1} + \beta(\mathcal{A}X_{t} - b)\label{lm update}
\end{align}
\end{subequations}
where
\begin{align}\label{AL function}
\mathcal{L}_\beta(X;p)
    \quad:=\quad
    C \bullet X + p^T (\mathcal{A}X - b) + \frac{\beta}{2} \| \mathcal{A}X - b\|^2.
\end{align}
The key part of \ourmethod is an efficient method, called Hybrid Low-Rank (HLR), for finding an approximate global
solution $X_{t}$ of the AL subproblem~\eqref{eq:AL}. 
The HLR method
solves subproblems of the form
\begin{align}\label{eq:Lr}
  \tag{$\mathrm{L}_r$}
  \min_Y\quad
  \left \{ 
    \mathcal{L}_{\beta}(YY^T;p_{t-1}) \quad:\quad
    \|Y\|_F \le 1, \quad
  Y \in \RR^{n\times r} \right\}
\end{align}
for some integer $r \ge 1$.
Subproblem \eqref{eq:Lr} is 
equivalent to the subproblem obtained by restricting $X$ in \eqref{eq:AL} to matrices with rank at most~$r$.
Since \eqref{eq:Lr} is nonconvex, it may have a \emph{spurious} (near) stationary point,
i.e., a (near) stationary point $Y$ such that $YY^T$ is not (nearly) optimal for \eqref{eq:AL}.

More specifically, HLR finds an approximate global solution $X_t$ of \eqref{eq:AL}
by solving a sequence of nonconvex subproblems $(\mathrm{L}_{r_k})_{k\ge 1}$ such that $r_{k+1} \le r_k +1$,
according to following steps:
i) find a near stationary point $Y=Y_k \in \mathbb R^{n \times r_k}$ of $(\mathrm{L}_{r_k})$ using an 
adaptive accelerated inexact proximal point (ADAP-AIPP) method that is based on a combination of ideas developed in \cite{SujananiMonteiro, Aaronetal2017, CatalystNC, WJRComputQPAIPP, WJRproxmet1};
ii)~check if $Y_kY_k^T$ is nearly optimal for \eqref{eq:AL} through a minimum eigenvalue computation and
terminate
the method if so;
else
iii) 
use the following escaping strategy to move away from
the current spurious near stationary point $Y_k$:
perform a Frank-Wolfe (FW) step from
$Y_k$ to obtain
a point $\tilde Y_k$ with either one column
in which case (the unlikely one) $r_{k+1}$ is set to one, or with $r_k + 1$ columns in which case $r_{k+1}$ is set to $r_k+1$,
and use 
$\tilde Y_k$ as the initial iterate for solving $\mathrm{L}_{r_{k+1}}$.
The initial pair $(r_1, \tilde Y_0)$ for HLR is chosen by using a warm start strategy, namely,
as the pair obtained at the
end of the HLR call for solving the
previous subproblem \eqref{eq:AL}.
 It is worth noting that \ourmethod
only stores the
current iterate $Y$  and never
computes the (implicit) iterate $YY^T$ (lying in the $X$-space).





Under the strong duality assumption,
it is shown that \ourmethod obtains an approximate
primal-dual solution 
of \eqref{eq:sdp-primal} and \eqref{eq:sdp-dual} with provable computational complexity bounds expressed
in terms of 
parameters associated with the SDP instance and user-specified tolerances.

\paragraph{Computational impact.}

Our computational results show that \ourmethod performs very well on many large-scale SDPs such as phase retrieval, maximum-stable-set, and matrix completion.
In all these applications, \ourmethod efficiently obtains accurate solutions for large-scale instances,
largely outperforming other state-of-the-art solvers.
For example,
\ourmethod takes approximately $1.75$ hours (resp. $13$ hours) on a personal laptop to solve
within $10^{-5}$ relative precision 
maximum stable set SDP instance for a Hamming graph
with $n \approx 4,000,000$ and $m \approx 40,000,000$ (resp. $n \approx 16,000,000$ and $m \approx 200,000,000$).
Moreover,
\ourmethod
takes approximately $7.5$ hours on a personal laptop to solve
within $10^{-5}$ relative precision
a  phase retrieval SDP instance
with $n = 1,000,000$ and $m = 12,000,000$.
An important reason for the good computational performance of \ourmethod
is that the rank of the iterates $X_t$ remain relatively small throughout the whole algorithm.


\paragraph{Related works.}

This part describes other methods for solving large-scale SDPs.  SDPNAL+ \cite{zhao2010newton, yang2015sdpnal+} is an AL based method that solves each AL subproblem using a semismooth Newton-CG method. Algorithms based on spectral bundle methods (more generally bundle methods) have also been proposed and studied for solving large-scale SDPs (e.g., \cite{Helmberg, Helmberg2, Helmberg3, DingGrimmer, Oustry}).
The more recent works
(e.g., \cite{odonoghue2016conic, garstka2021cosmo, deng2022new, madani2015admm, zheng2017fast})
propose methods for solving large-scale SDPs based on the alternating direction method of multipliers (ADMM).
The remaining of this section discusses
in more detail works that rely on the nonconvex LR approach and the FW method,
since those works are more closely related to this paper.
The reader is referred to the survey paper \cite{majumdar2020recent}
for additional methods for solving large scale SDPs.

The nonconvex LR approach of \cite{burer2003nonlinear, burer2005local}  has been successful in solving many relevant classes of SDPs.
The SDPLR method developed in  these works solves \eqref{eq:sdp-lr} with an AL method
whose AL subproblems are solved by a limited memory BFGS method.
Although SDPLR only handles equality constraints,
it is possible to modify it to handle inequalities
(e.g., \cite{kulis2007fast}).
\ourmethod is also based on the AL method but it applies it directly to \eqref{eq:sdp-primal} instead of \eqref{eq:sdp-lr}.
Moreover, in contrast to SDPLR, \ourmethod solves the AL subproblems
using the HLR method outlined above.
%


This paragraph describes works that solve (possibly a sequence of) \eqref{eq:sdp-lr} without using~ the AL method.
Approaches that use interior point methods for solving  \eqref{eq:sdp-lr} have been pursued for example in 
\cite{rosen2021scalable}.
In the context of MaxCut SDPs, several specialized methods have been proposed which solve \eqref{eq:sdp-lr}
using optimization techniques
which preserves feasibility
(e.g., \cite{burer2001projected, homer1997design,
erdogdu2022convergence,mei2017solving, journee2010low}).
Finally, 
Riemannian optimization methods have been used to solve special classes of SDPs where the feasible sets for \eqref{eq:sdp-lr} are smooth manifolds
(e.g., \cite{rosen2019se, journee2010low, huang2017solving, mei2017solving}).
The FW method minimizes a convex function $g(X)$ over a compact convex domain (e.g., the spectraplex $\Delta^n$).
It is appealing when a sparse solution is desired,
where the notion of sparsity is broad (e.g., small cardinality and/or rank).
The FW method has been used
(e.g., \cite{hazan2008sparse, shinde2021memory, jaggi2010simple})
for solving SDP feasibility problems by minimizing $g(X) = \phi(\mathcal A X - b)$
where $\phi$ is either the squared norm function $\| \cdot \|^2$
or the function $\mathrm{LSE}(y) = \log(\sum_i \exp y_i)$.
Several papers
(e.g., \cite{freund2017extended, nemiro2014, rao2013conditional})
introduce variants of the FW method for general convex optimization problems.

Another interesting method for solving \eqref{eq:sdp-primal} is CGAL of \cite{yurtsever2019conditional, yurtsever2015scalable} which
generates 
its iterates by
performing only FW steps with respect to AL subproblems in the same format as \eqref{eq:AL}.
As \ourmethod, the method of \cite{yurtsever2015scalable} only generates iterates in the $Y$-space.
Its Lagrange multiplier update policy though differs from \eqref{lm update} in that
it updates the Lagrange multiplier
in a more conservative way,
i.e., with $\beta$ in \eqref{lm update} replaced by a usually much smaller $\gamma_t>0$, and does
so only when the size of the new tentative multiplier is not too large.
Moreover, 
instead of using a pure FW method to
solve \eqref{eq:AL},
an iteration of the subroutine HLR invoked by
\ourmethod to solve 
\eqref{eq:AL} consists of an
ADAP-AIPP call applied to \eqref{eq:Lr}
and, if HLR does not terminate,
also a FW step (which generally increases the rank of the iterate by one).
As demonstrated by our computational results,
the use of ADAP-AIPP calls significantly reduces the number of FW steps
performed by \ourmethod,
and, as a by-product,
keeps the ranks of its iterates considerably smaller than those of the CGAL iterates.


The CGAL method was enhanced in \cite{yurtsever2021scalable} to derive a low-storage variant, namely, Sketchy-CGAL.
Instead of explicitly storing its most recent $Y$-iterate
as CGAL does,
this variant computes a certain approximation of the above iterates lying in $\RR^{n \times r}$ where  $r \in \{1,\ldots,n-1\}$ is a specified threshold value whose purpose is to limit
the rank of the stored approximation.
It is shown in \cite{yurtsever2021scalable} that Sketchy-CGAL has $O(m + n r)$ memory storage,
and that it outputs an $O(r^*/(r - r^*-1))$-approximate solution of \eqref{eq:sdp-primal} (constructed using the sketch)
under the assumption that
$r > r^* + 1$,
where $r^*$ is the largest among the ranks of all optimal solutions of \eqref{eq:sdp-primal}.
In contrast to either CGAL or \ourmethod,
a disadvantage of Sketchy-CGAL is that the accuracy of its output primal approximate solution is often low
and degrades further as $r$ decreases,
and can even be undetermined if
$r \le r^* + 1$.
Finally, alternative methods for solving SDPs with $O(m + n r^*)$ memory storage are presented in \cite{ding2021optimal,shinde2021memory,wang2022accelerated}.

\paragraph{Structure of the paper.}

This paper is organized into four sections.
Section~2 discusses the HLR method for solving the AL subproblem~\eqref{eq:AL}
and, more generally, smooth convex optimization problems over the spectraplex $\Delta^n$.
It also presents complexity bounds for HLR, given in Theorem~\ref{thm:complexityLRHFW}.
Section~3 presents \ourmethod for solving the pair of SDPs \eqref{eq:sdp-primal} and \eqref{eq:sdp-dual}
and presents
the main complexity result  of this paper, namely Theorem~\ref{thm:OuterComplexitySDP},
which provides complexity bounds for \ourmethod.
Finally, Section~4 presents computational experiments comparing \ourmethod with various solvers in a large collection of SDPs arising from 
stable set, phase retrieval, and matrix completion problems.

\subsection{Basic Definitions and Notations}

Let $\RR^n$ be the space of $n$ dimensional vectors,
$\RR^{n\times r}$ the space of $n \times r$ matrices,
and $\mathbb S^{n}$ the space of $n \times n$ symmetric matrices.
Let $\RR_{++}^n$ ($\RR^n_{+}$) be the convex cone in $\RR^n$ of vectors with positive (nonnegative) entries,
and let $\mathbb S_{++}^n$ ($\mathbb S^n_{+}$) be the convex cone in $\mathbb S^n$ of positive (semi)definite matrices.
Let $\inner{\cdot}{\cdot}$ and $\|\cdot\|$ be the Euclidean inner product and norm on~$\RR^n$,
and let $\bullet$ and $\|\cdot\|_{F}$ be the Frobenius inner product and norm on~$\mathbb S^n$.
The minimum eigenvalue of a matrix $Q \in \mathbb S^n$ is denoted by $\lambda_{\min}(Q)$, and $v_{\min}(Q)$ denotes a corresponding eigenvector of unit norm.
For any $t>0$ and $a\geq 0$, let $\log_a^+(t):=\max\{\log t, a\}$.

For a given closed convex set $C \subset \RR^n$, its boundary is denoted by $\partial C$ and the distance of a point $z \in \RR^n$ to $C$ is denoted by ${\rm dist}(z,C)$. The diameter of $C$, denoted $D_{C}$, is defined as
\begin{equation}\label{Diameter}
D_{C}:=\sup \{ \|Z-Z'\| : Z, Z' \in C\}.
\end{equation}
The indicator function of $C$, denoted by $\delta_C$, is defined by $\delta_C(z)=0$ if $z\in C$, and $\delta_C(z)=+\infty$ otherwise. The domain of a function $h :\RR^n\to (-\infty,\infty]$ is the set $\dom h := \{x\in \RR^n : h(x) < +\infty\}$.
Moreover, $h$ is said to be proper if
$\dom h \ne \emptyset$. The $\epsilon$-subdifferential of a proper convex function $h :\RR^n\to (-\infty,\infty]$ is defined by 
\begin{equation}\label{def:epsSubdiff}
\partial_\epsilon h(z):=\{u\in \RR^n: h(z')\geq h(z)+\inner{u}{z'-z}-\epsilon, \quad \forall z' \in \RR^n\}
\end{equation}
for every $z\in \RR^n$.	
The classical subdifferential, denoted by $\partial h(\cdot)$,  corresponds to $\partial_0 h(\cdot)$.  
Recall that, for a given $\epsilon\geq 0$, the $\epsilon$-normal cone of a closed convex set $C$ at $z\in C$, denoted by  $N^{\epsilon}_C(z)$, is 
$$N^{\epsilon}_C(z):=\{\xi \in \RR^n: \inner{\xi}{u-z}\leq \epsilon,\quad \forall u\in C\}.$$
The normal cone of a closed convex set $C$ at $z\in C$ is denoted by  $N_C(z)=N^0_C(z)$.

Given a differentiable function $\psi : \RR^n \to \RR$,
its affine approximation at a point $\bar z \in \RR^n$ is
\begin{equation}\label{eq:defell}
\ell_\psi(z;\bar z) :=  \psi(\bar z) + \inner{\nabla \psi(\bar z)}{z-\bar z} \quad \forall  z \in \RR^n.
\end{equation}
The function $\psi$ is $L$-smooth on a set $\Omega \subset \RR^n$ if its gradient is $L$-Lipschitz continuous on $\Omega$, i.e.,
\begin{equation}\label{upper curvature f}
\|\nabla \psi(x') -  \nabla \psi(x)\|\le L \|x'-x\| \quad \forall x,x' \in \Omega.
\end{equation}
The set of $L$-smooth functions on $\Omega$ is denoted by $\mathcal{C}^{1}(\Omega;L)$.

\section{Hybrid Low-Rank Method}\label{s:fw}
This section introduces a Hybrid Low-Rank (HLR) method which, as outlined in the introduction, uses a combination of the ADAP-AIPP method and Frank-Wolfe steps for approximately solving convex problems of the form as in \eqref{eq:AL}.
This section consists of three subsections. The first subsection introduces the main problem that the HLR method considers and introduces a notion of the type of approximate solution that it aims to find.
The second subsection presents the ADAP-AIPP method and its complexity results. The third subsection states the complete HLR method and establishes its total complexity.

\subsection{Problem of Interest and Solution Type}
Let $g : \mathbb S^n \to \RR$ be a convex and differentiable function.
The HLR method is developed in the context of solving the problem
\begin{equation}\label{AL subproblem}
    g_* := \min \left\{g\left(Z\right) : Z \in \Delta^n \right\}
\end{equation}
where $\Delta^n$ is the spectraplex as in \eqref{Delta Definition}
and $g$ is $L_g$-smooth on $\Delta^n$, i.e.,
there exists $L_g \geq 0$ such that
\begin{equation}\label{upper curvature g}
\|\nabla g(Z') -  \nabla g(Z)\|_{F}\le L_{g} \|Z'-Z\|_{F} \quad \forall Z,Z' \in \Delta^{n}.
\end{equation}

The goal of the HLR method is to find a near-optimal solution of \eqref{AL subproblem} whose definition is given immediately after the next result.

\begin{lem}\label{Normal Cone Delta Result}
Let $Z \in \Delta^n$ be given and define
\begin{equation}\label{theta bar x}
\theta (Z):=\max\{-\lambda_{\min}(\nabla g(Z)), 0\}.
\end{equation}
Then:
\begin{itemize}
    \item[a)] there hold
   \begin{equation}\label{PSD dual}
 \theta(Z) \ge 0, \quad
       \nabla g(Z)+\theta(Z) I \succeq 0;
\end{equation}
\item[b)]
for any $\epsilon>0$,
the inclusion holds
\begin{equation}\label{epsilon approximate optimality}
0 \in \nabla g(Z)+\partial_{\epsilon} \delta_{\Delta^n}(Z)
\end{equation}
if and only if
\begin{equation}\label{normal cone result}
   \nabla g(Z)\bullet Z +\theta(Z) \leq \epsilon.
   \end{equation}
\end{itemize}
   
\end{lem}
\begin{proof}
(a) The result is immediate from the definition of $\theta(Z)$ in \eqref{theta bar x}.

(b) It is easy to see that $\partial_{\epsilon} \delta_{\Delta^n}(Z)=N^{\epsilon}_{\Delta^n}(Z)$. Statement (b) then follows immediately from this observation, the definition of $\theta(Z)$ in \eqref{theta bar x}, and Proposition~\ref{Normal Cone Characterization Appendix}(b) with $G=\nabla g(Z)$.
\end{proof}




Relation \eqref{normal cone result} provides an easily verifiable condition for checking whether $Z$ satisfies inclusion \eqref{epsilon approximate optimality}. Moreover, \eqref{normal cone result} is equivalent to the ``complementary slackness'' condition
\begin{align}\label{Comp Slackness}
   \left(1 - \tr Z\right)\theta(Z) + [\nabla g(Z)+\theta(Z) I]\bullet Z\leq \epsilon.
\end{align}

\begin{definition}\label{def:sdp-2}
An $\epsilon$-optimal solution of \eqref{AL subproblem} is a matrix $Z \in \Delta^n$ satisfying relation \eqref{epsilon approximate optimality} or \eqref{normal cone result}. 
\end{definition}



The next lemma shows that the objective value of an $\epsilon$-optimal solution of \eqref{AL subproblem} is within $\epsilon$ of the optimal value of \eqref{AL subproblem}.

\begin{lemma}
  An $\epsilon$-optimal solution $Z$ of \eqref{AL subproblem} satisfies that $g(Z)-g_{*}\leq \epsilon$.
\end{lemma}
\begin{proof}
  Let $Z_*$ be an optimal solution of \eqref{AL subproblem}.
  Relation \eqref{epsilon approximate optimality} implies that $-\nabla g(Z) \in N^{\epsilon}_{\Delta^n}(Z)$
  and hence that $\inner{-\nabla g(Z)}{Z_*-Z} \leq \epsilon$.
It then follows from this relation and the fact that $g$ is convex that
  \begin{equation*}
    g(Z_*) - g(Z) \geq \inner{\nabla g(Z)}{Z_* - Z} \geq - \epsilon,
  \end{equation*}
which immediately implies the result.
\end{proof}

\subsection{The ADAP-AIPP Method}\label{ADAP-AIPP Method}

As already mentioned in the introduction, one iteration of the HLR method consists of a call to the ADAP-AIPP method followed by a FW step.
The purpose of this subsection is to describe the details of the ADAP-AIPP method.

For a given integer $s$,
consider the
the subproblem obtained by restricting \eqref{AL subproblem} to matrices $Z$ of rank at most~$s$, or equivalently, the reformulation
\begin{equation}\label{factorized problem}
\min \{\tilde g(U):=g(UU^{T}) \;:\: U \in \bar B_1^s\},
\end{equation}
where 
\begin{equation}\label{dimensional ball}
\bar B_r^s:=\{U \in \RR^{n \times s}:\|U\|_{F}\leq r\}
\end{equation}
denotes the Frobenius ball of radius $r$ in $\RR^{n \times s}$.
In this subsection, the above set will be denoted by $\bar B_r$
since the column dimension $s$ remains constant throughout its presentation.

The goal of the ADAP-AIPP method is to find an approximate stationary solution of \eqref{factorized problem} as described in
Proposition~\ref{New Main Result ADAP-AIPP}(a) below.
Briefly, ADAP-AIPP is an inexact proximal point method which attempts to solve its (potentially nonconvex) prox subproblems using an accelerated composite gradient method, namely,
ADAP-FISTA, whose description is given in Appendix~\ref{ADAP-FISTA}.
A rough description of the $j$-th iteration of ADAP-AIPP is as follows: given $W_{j-1} \in \bar B_1$ and a positive scalar $\lambda_{j-1}$, ADAP-AIPP calls 
the ADAP-FISTA method 
to attempt to find a suitable approximate solution of the possibly nonconvex proximal subproblem
\begin{align}\label{ADAP-AIPP subproblem}
\min_{U\in \bar B_{1}} \left\{\lambda \tilde g(U)+ \frac{1}{2}\|U- W_{j-1}\|^2_{F} \right\},
\end{align}
where the first call made is always performed with $\lambda=\lambda_{j-1}$.  If ADAP-FISTA successfully finds such a solution, ADAP-AIPP sets this solution as its next iterate $W_j$ and sets $\lam$ as its next prox stepsize $\lambda_j$. If ADAP-FISTA is unsuccessful, ADAP-AIPP invokes it again to attempt to solve \eqref{ADAP-AIPP subproblem} with $\lambda=\lambda/2$.  This loop always terminates since ADAP-FISTA is guaranteed to terminate with success when the objective in \eqref{ADAP-AIPP subproblem} becomes strongly convex,
which occurs when $\lambda$ is sufficiently small.

The formal description of the ADAP-AIPP method is presented below.
For the sake of simplifying the input lists of the algorithms stated throughout this paper,
the parameters $\sigma$ and $\chi$ are
considered universal ones
(and hence not input parameters).


\noindent\begin{minipage}[t]{1\columnwidth}%
\rule[0.5ex]{1\columnwidth}{1pt}

\noindent \textbf{ADAP-AIPP Method}

\noindent \rule[0.5ex]{1\columnwidth}{1pt}%
\end{minipage}
\noindent \textbf{Universal Parameters}: $\sigma\in (0,1/2)$ and $\chi \in (0,1)$.

\noindent \textbf{Input}: quadruple $(\tilde g, \lambda_0, \underline W, \bar \rho) \in (\mathcal{C}^{1}\left(\Delta^{n};L_g\right), \RR_{++}, \bar B_1, \RR_{++})$.




\begin{itemize}
\item[{\bf 0.}] set $W_0=\underline W$, $j=1$, and
\begin{equation}\label{def:lamb-C1}
\lambda=\lambda_{0}, \quad \bar M_0 = 1;
\end{equation}

\item[{\bf 1.}]  choose
$\underbar M_j \in [1, \bar M_{j-1}]$ and call the ADAP-FISTA method in Appendix~\ref{ADAP-FISTA}
with universal input $(\sigma, \chi)$ and inputs
\begin{align}
x_0&=W_{j-1}, \quad
(\mu,L_0)= (1/2,\underbar M_j)\label{eq:Ms-mu}, \\
\psi_s &=\lam \tilde g+\frac{1}{2}\|\cdot-W_{j-1}\|_{F}^2 , \quad \psi_n = \lam \delta_{\bar B_1} \label{eq:psiS-psimu};
\end{align}

\item[{\bf 2.}]
if ADAP-FISTA fails or its output 
$(W,V,L)$ (if it succeeds) does not satisfy the inequality \begin{equation}\label{subdiff ineq check}
	   \lambda \tilde g(W_{j-1}) - \left [\lambda \tilde g(W) + \frac{1}{2} \|W-W_{j-1}\|_{F}^2 \right ] \ge V \bullet (W_{j-1}-W),
	   \end{equation}
then set $\lam=\lam/2$ and go to step $1$; else, set
	$(\lam_j,\bar M_j)=(\lam,L)$, $(W_{j},V_{j})=(W,V)$,
	and 
	\begin{align}
    &R_j:=\frac{V_j+W_{j-1}-W_j}{\lambda_j}\label{Rj def}
	\end{align}
	and
	go to step~3;
\item[{\bf 3.}]	
if $\|R_j\|_{F}\leq \bar \rho$,
then stop with success and output $(\overline W,\overline R)=(W_j,R_j)$; else, go to step~4;

\item[{\bf 4.}] 
set $j\gets j+1$ and go to step~1. 
\end{itemize}
\noindent \rule[0.5ex]{1\columnwidth}{1pt}

Several remarks about ADAP-AIPP are now given.
First, at each iteration, steps 1 and 2 successively call the ADAP-FISTA method with inputs given by \eqref{eq:Ms-mu} and \eqref{eq:psiS-psimu} to obtain a prox stepsize $\lambda_j\leq \lambda_{j-1}$ and a pair $(W_j,V_j)$ satisfying \eqref{subdiff ineq check} and 
\begin{equation}\label{relative error}
\|V_j\|_{F}\leq \sigma\|W_j-W_{j-1}\|_{F}, \quad V_j \in \lambda_j \left[\nabla \tilde g(W_j)+\partial \delta_{\bar B_1}(W_j)\right]+(W_j-W_{j-1})
\end{equation}
where $\sigma$ is part of the input of ADAP-AIPP. Such a pair $(W_j,V_j)$ can be viewed as an approximate stationary solution of prox subproblem \eqref{ADAP-AIPP subproblem} with $\lambda=\lambda_j$, where the residual
$V_j$ is relaxed from being
zero to a quantity that is
now relatively bounded as in \eqref{relative error}.
Second, it follows immediately from the inclusion in relation \eqref{relative error} and the definition of $R_j$ in \eqref{Rj def} that the pair $(W_j,R_j)$ computed in step 2 of ADAP-AIPP satisfies the inclusion $R_j \in \nabla \tilde g(W_j)+\partial \delta_{\bar B_1}(W_j)$ for every iteration $j \geq 1$. As a consequence, if ADAP-AIPP terminates in step 3, then the pair $(\overline W,\overline R)=(W_j,R_j)$ output by this step is a $\bar \rho$-approximate stationary solution of \eqref{factorized problem}, i.e., it satisfies
\begin{equation}\label{stationary solution}
    \overline R\in \nabla \tilde g(\overline W)+\partial \delta_{\bar B_1}(\overline W), \quad \|\overline R\|_{F}\leq \bar \rho.
\end{equation} Finally, it is interesting to note that ADAP-AIPP is a universal method in that it requires no knowledge of any parameters (such as objective function curvatures) underlying problem \eqref{factorized problem}. 

Before stating the main complexity result of the ADAP-AIPP method, the following quantities are introduced
\begin{equation}
\bar G:=\sup \{\|\nabla g(UU^{T})\|_{F} : U\in \bar B_3\}, \quad L_{\tilde g}:=2\bar G+36L_g, \label{phi* alpha0 quantities}
\end{equation}
\begin{equation}
\underline \lambda:=\min\{\lambda_0, 1/(4L_{\tilde g})\}, \quad C_{\sigma}=\frac{2(1-\sigma)^2}{1-2\sigma},\label{phi* alpha0 quantities-2}
\end{equation}
where $\lambda_0$ is the initial prox stepsize of ADAP-AIPP and $L_{g}$ and $\bar B_3$ are as in \eqref{upper curvature g} and \eqref{dimensional ball}, respectively.
Observe that $C_\sigma$ is well-defined and positive due to the fact that $\sigma \in (0,1/2)$.

The main complexity result of ADAP-AIPP is stated in the proposition below.
Its proof is in Appendix~\ref{ADAP-AIPP Appendix Proof}.  
\begin{proposition}\label{New Main Result ADAP-AIPP}
The following statements about ADAP-AIPP hold:
\begin{itemize}
\item[(a)]
ADAP-AIPP terminates with a pair $(\overline W,\overline R)$ that is a $\bar \rho$-approximate stationary solution of \eqref{factorized problem}
and its last iteration index $l$ satisfies
\begin{equation}\label{iteration bound}
1 \le l \leq \mathcal T:=1+\frac{C_{\sigma}}{\underline \lambda \bar \rho^2} \left[\tilde g(\underline W)-\tilde g(\overline W)\right],
\end{equation}
where $\bar \rho>0$ is an input tolerance, $\underline W$ is the initial point, and $C_{\sigma}$ and $\underline \lambda$ are as in \eqref{phi* alpha0 quantities-2};
\item[(b)] the total number of ADAP-FISTA calls performed by ADAP-AIPP is no more than
\begin{equation}\label{ADAP FISTA Calls Bound}
\mathcal T+\lceil \log_0^+(\lambda_0/{\underline \lam})/\log 2 \rceil
\end{equation}
where $\lambda_0$ is the initial prox stepsize and $\mathcal T$ is as in \eqref{iteration bound}.
\end{itemize}
\end{proposition}

Some remarks about Proposition~\ref{New Main Result ADAP-AIPP} are now in order.
First, it follows from statement (a) that
$\tilde g(\overline W)\leq\tilde g(\underline W)$. Second, recall that each ADAP-AIPP iteration may perform more than a single ADAP-FISTA call. Statement (b) implies that 
the total number of ADAP-FISTA calls performed is at most the total number of ADAP-AIPP iterations performed plus a logarithmic term.


\subsection{HLR Method}\label{HLR Method}
The goal of this subsection is to describe the HLR method for solving problem \eqref{AL subproblem}. 
The formal description of the HLR method is presented below.
\noindent\begin{minipage}[t]{1\columnwidth}%
\rule[0.5ex]{1\columnwidth}{1pt}

\noindent \textbf{HLR Method}

\noindent \rule[0.5ex]{1\columnwidth}{1pt}%
\end{minipage}

\noindent \textbf{Input}:  A quintuple $(\bar Y_0, g, \bar \epsilon, \bar \rho, \lambda_0) \in \bar B^{s_0}_1 \times \mathcal{C}^{1}\left(\Delta^{n};L_g\right) \times \RR^{3}_{++}$ for some $s_0 \ge 1$.


\noindent \textbf{Output}: $\bar Y\in \bar B^{s}_1$ for some $s \ge 1$ such that $\bar Y\bar Y^{T}$ is an $\bar \epsilon$-optimal solution of \eqref{AL subproblem}.

\begin{itemize}
\item[{\bf 0.}] set  $\tilde Y_0 = \bar Y_0$,
$s=s_0$, and $k=1$.

\item[{\bf 1.}]  call ADAP-AIPP with quadruple $(g,\lambda_0, \bar \rho, \underline W)=(g,\lambda_0, \bar \rho,\tilde Y_{k-1})$ and let $(Y_{k},\mathcal R_k)\in \bar B^s_1 \times \RR^{n \times s}$ denote its output pair $(\overline W,\overline R)$;

\item[{\bf 2.}]
compute 
\begin{equation}\label{FW subproblem eq}
\theta_k = \max\{-\lambda_{\min}(G_k) , 0\}, \qquad
y_k= \begin{cases}
     v_{\min}(G_k) & \text{ if $\theta_k>0$},\\
        0 & \text{ otherwise},
        \end{cases}
\end{equation}
where
\begin{equation}\label{gradient Def}
G_k := \nabla g(Y_kY_k^T) \in \mathbb S^n
\end{equation}
and $(\lam_{\min}(G_k), v_{\min}(G_k)) \in \R \times \R^n$ is a minimum eigenpair of $G_k$;
\item[{\bf 3.}]
set
\begin{equation}\label{FW subproblem eq-2}
\epsilon_k:=(G_kY_k)\bullet Y_k+\theta_k.
\end{equation}
If 
\begin{equation}\label{normal cone delta}
\epsilon_{k} \leq \bar \epsilon
\end{equation}
then {\bf stop} and output pair $(\bar Y, \bar \theta)= (Y_{k}, \theta_k)$; else go to step 4; 
\item[\bf 4.] compute \begin{equation}\label{stepsize Yk}
    \alpha_k=\argmin_{\alpha}\left\{g\left(\alpha y_{k}(y_{k})^{T}+(1-\alpha)Y_kY^{T}_k\right): \alpha \in [0,1] \right\}\}
\end{equation}
and set 
\begin{equation}\label{tilde Y}
(\tilde Y_{k},s)=\begin{cases}
(y_k,1) & \text{ if $\alpha_{k}=1$}\\
\left(\left[\sqrt{1-\alpha_{k}} Y_k, \sqrt{\alpha_k} {y_k}\right],s+1\right)& \text{ otherwise};
\end{cases}
\end{equation}
\item[{\bf 5.}] 
set $k\gets k+1$ and go to step~1. 
\end{itemize}
\noindent \rule[0.5ex]{1\columnwidth}{1pt}

Remarks about each of the steps in HLR are now given.
First, step 1 of the $k$-th iteration of HLR calls ADAP-AIPP with initial point $\tilde Y_{k-1} \in \bar B_1^s$ to produce an iterate $Y_k \in \bar B_1^s$ that is an $\bar \rho$-approximate stationary solution of nonconvex problem \eqref{factorized problem}.
Second, step~2 performs a minimum eigenvector (MEV) computation to compute the quantity $\theta_k$ in \eqref{FW subproblem eq}, which is needed for the termination check performed in step~3. Third, the definition of $\theta(\cdot)$ in \eqref{theta bar x}, and relations
\eqref{gradient Def}, \eqref{FW subproblem eq}, and \eqref{FW subproblem eq-2}, imply that
\[
\theta_k=\theta(Y_kY_k^{T}), \quad
\epsilon_k = \nabla g(Y_kY_k^{T}) \bullet (Y_kY_k^{T})+\theta(Y_kY_k^{T}).
\]
Hence, it follows from Lemma~\ref{Normal Cone Delta Result}(b) that termination criterion \eqref{normal cone delta} in step 3 is equivalent to checking if $Y_kY_k^{T}$ is a $\bar \epsilon$-optimal solution of \eqref{AL subproblem}.
Fourth, if the termination criterion in step 3 is not satisfied,
a FW step at $Y_kY_k^{T}$ for \eqref{AL subproblem} is taken in step 4
to produce an iterate $\tilde Y_k\tilde Y_k^{T}$ as in \eqref{stepsize Yk} and \eqref{tilde Y}. 
The computation of $\tilde Y_k$
is entirely performed in the $Y$-space to avoid forming
matrices of the form $YY^T$, and
hence to save storage space.
The reason for performing this FW step is to make HLR escape from the spurious near stationary point $Y_k$ of \eqref{factorized problem} (see the end of the paragraph containing \eqref{eq:Lr} in the Introduction).
Finally, the quantity $s$ as in \eqref{tilde Y} keeps track of the column dimension of the most recently generated $\tilde Y_k$.
It can either increase by one or be set to one after the update \eqref{tilde Y}
is performed.


The complexity
of the HLR method is described in the result below
whose
proof is given in the next subsection. 

\begin{thm}\label{thm:complexityLRHFW}
The following statements about the HLR method hold:
\begin{itemize}
\item[(a)] the HLR method outputs a point $\bar Y$ such that $\bar Z = \bar Y \bar Y^T$ is a $\bar \epsilon$-optimal solution of \eqref{AL subproblem}
in at most
    \begin{equation}\label{Hybrid FW Method Exact Complexity-Body}
    {\cal S}(\bar \epsilon) := \left \lceil 1+\frac{4\max\left\{g(\bar Z_0)-g_{*},\sqrt{4L_g(g(\bar Z_0)-g_{*})}, 4L_{g}\right\}}{\bar \epsilon} \right\rceil
    \end{equation}
iterations (and hence MEV computations) where $\bar Z_0=\bar Y_0\bar Y^{T}_0$, $L_g$ is as in \eqref{upper curvature g}, and $g_{*}$ is the optimal value of \eqref{AL subproblem};
\item[(b)] the total number of ADAP-FISTA calls performed by the HLR method is no more than
\begin{align}\label{Total Complexity HLR}
(1+\mathcal Q){\cal S}(\bar \epsilon)
+\frac{C_{\sigma}\max\{8\bar G+144 L_g,1/\lambda_0\}}{\bar \rho^2}\left[g(\bar Z_0)-g(\bar Z)\right]
\end{align}
where $\lambda_0$ is given as input to the HLR method, and $\bar G$, $C_{\sigma}$, and $\mathcal Q$ are as in \eqref{phi* alpha0 quantities}, \eqref{phi* alpha0 quantities-2}, and \eqref{iteration bound-3}, respectively.
\end{itemize}
\end{thm}

Two remarks about Theorem~\ref{thm:complexityLRHFW} are now given.
First, it follows from statement (a) that
HLR performs $\mathcal O\left(1/\bar \epsilon\right)$ iterations (and hence MEV computations). Second, statement (b) implies that 
the total number of ADAP-FISTA calls performed by HLR is $\mathcal O\left(1/\bar \epsilon+1/\bar \rho^2\right)$ where $\bar \rho$ is the input tolerance for each ADAP-AIPP call in step 1 of HLR.

Recall from the Introduction that
an iteration of the \ourmethod method described in \eqref{eq:AL} and \eqref{lm update}
has to approximately solve subproblem \eqref{eq:AL} and that the HLR method specialized to the case where $g(\cdot)=\mathcal L_{\beta}(\cdot;p)$ is the novel proposed tool towards solving it.
In the following,
we discuss how to specialize some of the steps of the HLR to this case.
Step 1 of HLR calls the ADAP-AIPP method whose steps can be easily implemented if it is known how to project a matrix onto the unit ball $\bar B_1$ and how to compute $\nabla \tilde g(Y)$ where $\tilde g(Y)$ is as in \eqref{factorized problem}.
Projecting a matrix onto the unit ball is easy as it just involves dividing the matrix by its norm. Define the quantities
\begin{equation}\label{Theta Lagrangian}
q(Y;p) := p + \beta (\mathcal{A}(YY^T) - b), \quad
\tilde \theta (Y;p) :=
\max\{-\lambda_{\min}[C+\mathcal A^{*}\left(q(Y;p)\right)], 0\}.
\end{equation}
When $g(\cdot)=\mathcal L_{\beta}(\cdot;p)$,
the matrix $\nabla \tilde g(Y)$ can be explicitly computed as
\begin{equation}\label{Computation of Gradient}
\nabla \tilde g(Y) = 2\nabla g(YY^T)Y, \quad \nabla g(YY^T) =
C+\mathcal A^{*}(q(Y;p)).
\end{equation}
Also, the stepsize $\alpha_k$ in step 4 of HLR has a closed form expression given by
\begin{align*}
  \alpha_k = \min \left \{\frac{C Y_k \bullet Y_k + (\mathcal{A}^*q_k) Y_k \bullet Y_k + \tilde \theta(Y_k;p)}{\beta \|\mathcal{A}(Y_k Y_k^T) - \mathcal{A}(y_k (y_k)^T)\|^2}, \; 1\right\}
\end{align*}
where
\begin{align}
&q_k:=q(Y_k;p), \quad y_k:= \begin{cases}
     v_{\min}\left(\nabla g(Y_kY_k^{T})\right) & \text{ if $\tilde \theta(Y_k;p)>0$,}\\
        0 & \text{ otherwise}.
        \end{cases}
\end{align}

\subsection{Proof of Theorem~\ref{thm:complexityLRHFW}}\label{LRHFW Proof}
This subsection provides the proof of Theorem~\ref{thm:complexityLRHFW}. The following proposition establishes important properties of the iterates generated by HLR.
\begin{prop}\label{FW Equivalences}
    For every $k \ge 1$, define
    \begin{equation}\label{Xk definitions}
Z_k=Y_kY^{T}_{k}, \quad Z^{F}_k=y_ky_k^{T}, \quad D_k:=Z_k-Z_k^{F}, \quad \tilde Z_{k}=\tilde Y_k\tilde Y^{T}_k,
\end{equation}
where $y_k$ is as in \eqref{FW subproblem eq-2}.
Then, for every $k \ge 1$, the following relations hold:
\begin{align}
    &Z^F_k \in \argmin_{U} \{ \ell_g(U;Z_k) : U \in \Delta^{n}\}; \label{Xkf FW transl}\\
    &\tilde Z_{k}=Z_k-\alpha_kD_k; \label{Z tilde transl}\\
    &\epsilon_k=G_k \bullet D_k;\label{epsilon FW transl}\\
&\alpha_k=\argmin_{\alpha \in [0,1]}g(Z_k-\alpha D_k)\} \label{alpha translated}\\
&g(Z_{k+1})\leq g(\tilde Z_k) \leq g(Z_k) \label{descent HLR}
\end{align}
where $\epsilon_k$ and $\alpha_{k}$ are as in \eqref{FW subproblem eq} and \eqref{stepsize Yk}, respectively. Moreover, for every $k \ge 1$, it holds that
\begin{align}\label{theta F epsilon}
    \theta_k \geq 0, \quad G_k+\theta_kI \succeq 0, \quad 
    G_k\bullet Z_k+\theta_k=\epsilon_k.
\end{align}
\end{prop}
\begin{proof}
Relation \eqref{Xkf FW transl} follows immediately from the way $y_k$ is computed in \eqref{FW subproblem eq-2}, the definitions of $Z_k$ and $Z^{F}_k$ in \eqref{Xk definitions}, and Proposition~\ref{Pair Optimal FW subproblem} with $G=G_k$ and $(Z^{F},\theta^{F})=(Z_k^{F},\theta_k)$. 

To show relation \eqref{epsilon FW transl}, it is first necessary to show that
\begin{equation}\label{Theta K Bullet}
\theta_k=-G_k \bullet Z_k^{F}.
\end{equation}
Consider two possible cases. For the first case, suppose that $\theta_k=0$, which in view of the definitions of $y_k$ and $Z^{F}_k$ in \eqref{FW subproblem eq} and \eqref{Xk definitions}, respectively implies that $Z_k^{F}=0$. Hence, relation \eqref{Theta K Bullet} immediately follows. For the other case suppose that $\theta_k>0$ and hence $(\theta_k,y_k)=(-\lambda_{\min}(G_k),v_{\min}(G_k))$. This observation and the fact that $(\lambda_{\min}(G_k), v_{\min}(G_k))$ is an eigenpair of $G_k$ imply that $G_ky_k \bullet y_k =-\theta_k$, which in view of the definition of $Z^{F}_k$ in \eqref{Xk definitions} immediately implies relation \eqref{Theta K Bullet}.

It is now easy to see that the definition of $\epsilon_k$ in \eqref{FW subproblem eq}, relation \eqref{Theta K Bullet}, the update rule for $\tilde Y_k$ in \eqref{tilde Y}, and the definitions of $Z_k$, $Z^{F}_k$, and $D_k$ in \eqref{Xk definitions} imply that
\begin{align*}
\tilde Z_k&=Z_k-\alpha_kD_k\\
\epsilon_k&=G_k \bullet (Z_k-Z^F_k)=G_k \bullet D_k\\
\alpha_k&=\argmin_{\alpha}\left\{g(Z_k-\alpha D_k): \alpha \in [0,1] \right\}\}
\end{align*}
and hence relations \eqref{Z tilde transl}, \eqref{epsilon FW transl}, and \eqref{alpha translated} follow.


Relation \eqref{iteration bound} and the fact that HLR during its $k$-th iteration calls ADAP-AIPP with initial point $\underline W=\tilde Y_{k-1}$ and outputs point $\overline W=Y_k$ imply that $\tilde g(Y_k) \leq \tilde g(\tilde Y_{k-1})$. This fact together with the definition of $\tilde g$ in \eqref{factorized problem} and the definitions of $Z_k$ and $\tilde Z_k$ in \eqref{Xk definitions} imply the first inequality in \eqref{descent HLR}. Now the first inequality in \eqref{descent HLR} and relations \eqref{Xkf FW transl}, \eqref{Z tilde transl}, \eqref{epsilon FW transl}, and \eqref{alpha translated} imply that 
$Z_k=Y_kY_k^{T}$ and $\tilde Z_k=\tilde Y_k\tilde Y_k^{T}$ can be viewed as iterates of the $k$-th iteration of the RFW method of Appendix~\ref{FW Appendix}. The second inequality in \eqref{descent HLR} is then an immediate consequence of this observation and the second relation in \eqref{descent tildes}.


The first two relations in \eqref{theta F epsilon} follow directly from the definitions of $G_k$ and $\theta_k$ in \eqref{gradient Def} and \eqref{FW subproblem eq}, respectively. The definitions of $\epsilon_k$ and $Z_k$ in \eqref{FW subproblem eq-2} and \eqref{Xk definitions}, respectively, immediately imply the third relation in \eqref{theta F epsilon}.
\end{proof}

Since step 1 of the HLR method consists of a call to the ADAP-AIPP method developed in Subsection~\ref{ADAP-AIPP Method}, the conclusion of Proposition~\ref{New Main Result ADAP-AIPP} applies to this step.
The following result,
which will be useful in the analysis of the HLR method,
translates the conclusion
of Proposition~\ref{New Main Result ADAP-AIPP} to the current setting.

\begin{proposition}\label{ADAP-AIPP result translation}
The following statements about step 1 of the $k$-th iteration of the HLR method hold: 
\begin{itemize}
\item[(a)]
the ADAP-AIPP call terminates with a pair $(Y_k,\mathcal R_{k})$ satisfying \begin{equation}\label{stationary solution-2}
    \mathcal R_k\in \nabla \tilde g(Y_k)+\partial \delta_{\bar B_1}(Y_k), \quad \|\mathcal R_k\|_{F}\leq \bar \rho,
\end{equation}
and the number $l_k$ of ADAP-AIPP iterations performed by the ADAP-AIPP call in step 1 satisfies
\begin{equation}\label{iteration bound-2}
1\leq l_k\leq{\cal T}_k := 1+\frac{C_{\sigma} \max\{8\bar G+144 L_g,1/\lambda_0\}}{\bar \rho^2} \left[\tilde g(\tilde Y_{k-1})-\tilde g(Y_k)\right]
\end{equation}
where $\lambda_0$ and $\bar \rho$ are given as input to the HLR method and $L_{g}$, $\tilde g$, $\bar G$, and $C_{\sigma}$ are as in \eqref{upper curvature g}, \eqref{factorized problem}, \eqref{phi* alpha0 quantities}, and \eqref{phi* alpha0 quantities-2} respectively;
\item[(b)] the number of ADAP-FISTA calls performed by the ADAP-AIPP call in step 1 is no more than ${\cal T}_k + \mathcal Q$ where
\begin{equation}\label{iteration bound-3}
\mathcal Q := \left\lceil \log_0^+\left(\lambda_0\max\{8\bar G+144 L_g, 1/\lambda_0\}\right)/\log 2 \right\rceil.
\end{equation}
\end{itemize}
\end{proposition}

\begin{proof}
(a) The first statement is immediate from relation \eqref{stationary solution} and the fact that ADAP-AIPP outputs pair $(Y_k, \mathcal R_k)=(\overline W,\overline R)$. To prove the second statement, suppose that the ADAP-AIPP call made in step 1 terminates after performing $l_k$ iterations. It follows immediately from Proposition~\ref{New Main Result ADAP-AIPP}(a) and the fact that the HLR method during its $k$-th iteration calls ADAP-AIPP with initial point $\underline W=\tilde Y_{k-1}$ and outputs point $\overline W=Y_k$
that $l_k$ satisfies
\begin{equation}\label{ADAP-AIPP iteration index}
1\leq l_k \overset{\eqref{iteration bound}}{\leq} 1+\frac{C_{\sigma}}{\underline \lambda \bar \rho^2} \left[\tilde g(\tilde Y_{k-1})-\tilde g(Y_k)\right]. 
\end{equation}
The result then follows from the definitions of $\underline \lambda$ and $L_{\tilde g}$ in \eqref{phi* alpha0 quantities-2} and \eqref{phi* alpha0 quantities}, respectively.

(b) The result follows immediately from (a), the fact that
the number of times $\lam$ is divided by $2$
in step 2 of ADAP-AIPP is at most $\lceil \log_0^+(\lambda_0/{\underline \lam})/\log 2 \rceil$, and the definitions of $\underline \lambda$ and $L_{\tilde g}$ in \eqref{phi* alpha0 quantities-2} and \eqref{phi* alpha0 quantities}, respectively.
\end{proof}

We are now ready to give the proof of Theorem~\ref{thm:complexityLRHFW}.
\begin{proof}[Proof of Theorem~\ref{thm:complexityLRHFW}]
(a) Consider the matrix $\bar Z=\bar Y\bar Y^{T}$ where $\bar Y$ is the output of the HLR method. The definition of $G_k$ in \eqref{gradient Def}, the fact that the definitions of $\theta_{k}$ and $\theta(\cdot)$ in \eqref{FW subproblem eq} and \eqref{theta bar x}, respectively, imply that $\theta_{k}=\theta(Y_kY_k^{T})$, relation \eqref{theta F epsilon}, and the stopping criterion \eqref{normal cone delta} in step 3 of the HLR method immediately imply that the pair $(\bar Z, \theta(\bar Z))$ satisfies relation \eqref{normal cone result} in Lemma~\ref{Normal Cone Delta Result}(b) with $\epsilon=\bar \epsilon$ and hence $\bar Z$ is an $\bar \epsilon$-optimal solution of \eqref{AL subproblem}. 
To show that the number of iterations that the HLR method performs to find such an $\bar \epsilon$-optimal solution is at most the quantity in \eqref{Hybrid FW Method Exact Complexity-Body}, observe that Proposition~\ref{FW Equivalences} establishes that the HLR method generates iterates $Y_k$ and $\tilde Y_k$ during its $k$-th iteration such that $Z_k=Y_kY_k^{T}$ and $\tilde Z_k=\tilde Y_k\tilde Y_k^{T}$ can be viewed as iterates of the $k$-th iteration of the RFW method of Appendix~\ref{FW Appendix}. The result then immediately follows from this observation, Theorem~\ref{FW Theorem}, and the fact that the diameter of $\Delta^{n}$ is at most 2.
 
(b) Suppose that the HLR method terminates at an iteration index $\bar K$. It follows from relations \eqref{iteration bound-3} and \eqref{Hybrid FW Method Exact Complexity-Body} and the definition of $\mathcal T_{k}$ in \eqref{iteration bound-2} that the HLR method performs at most
\begin{align}\label{Summable Complexity}
(1+\mathcal Q){\cal S}(\bar \epsilon)
+\frac{C_{\sigma}\max\{8\bar G+144 L_g,1/\lambda_0\}}{\bar \rho^2}\sum_{k=1}^{\bar K}\left[\tilde g(\tilde Y_{k-1})-\tilde g(Y_k)\right]
\end{align}
ADAP-FISTA calls. 
Now using the fact that $\tilde g(\tilde Y_{k})\leq \tilde g(Y_k)$, it is easy to see that the last term in \eqref{Summable Complexity} is summable:
\begin{align}\label{complexity derivation}
\sum_{k=1}^{\bar K}\tilde g(\tilde Y_{k-1})-\tilde g(Y_k)&=\tilde g(\tilde Y_{\bar K-1})-\tilde g(Y_{\bar K})+\sum_{k=1}^{\bar K-1}\tilde g(\tilde Y_{k-1})-\tilde g(Y_k) \nonumber\\
&\leq \tilde g(\tilde Y_{\bar K-1})-\tilde g(Y_{\bar K})+\sum_{k=1}^{\bar K-1}\tilde g(\tilde Y_{k-1})-\tilde g(\tilde Y_k)=\tilde g(\tilde Y_0)-\tilde g(Y_{\bar K}).
\end{align}
The result then follows from relations \eqref{Summable Complexity} and \eqref{complexity derivation}, the facts that $Y_{\bar K}=\bar Y$, $\bar Z=\bar Y\bar Y^{T}$, $\tilde Y_0=\bar Y_0$, $\bar Z_{0}=\bar Y_{0}\bar Y^{T}_{0}$, and the definition of $\tilde g$ in \eqref{factorized problem}.
\end{proof}

\section{\ourmethod}\label{HLR-AL Method}
This section presents
an inexact AL method for solving
the pair of primal-dual SDPs \eqref{eq:sdp-primal} and \eqref{eq:sdp-dual}, namely,
\ourmethod, whose outline is given in the introduction. It contains two subsections.
Subsection~\ref{Desc Our} formally states \ourmethod and
presents its main complexity result, namely Theorem~\ref{thm:OuterComplexitySDP}. 
Subsection~\ref{Pf Main Theorem} is devoted to the proof of Theorem~\ref{thm:OuterComplexitySDP}.

Throughout this section, it is assumed that
\eqref{eq:sdp-primal} and \eqref{eq:sdp-dual} have optimal solutions $X_{*}$ and $(p_{*},\theta_*)$,
respectively, and that both \eqref{eq:sdp-primal} and \eqref{eq:sdp-dual} have
the same optimal value.
It is well-known that such an assumption is
equivalent to the existence of
a triple $(X_{*},p_{*},\theta_{*})$ satisfying the optimality conditions:
\begin{align}
\text{(primal feasibility)} &\qquad
\mathcal A(X_*)-b=0,
\quad \tr(X_*)\leq 1,
\quad X_*  \succeq 0, \nonumber\\
\text{(dual feasibility)} &\qquad
S_* := C+\mathcal A^{*}p_* + \theta_* I \succeq 0,
\quad \theta_* \geq 0, \label{KKT system2}\\
\text{(complementarity)} &\qquad
\inner{X_*}{S_*} = 0, \quad \theta_* (1 - \tr X_*) = 0.\nonumber
\end{align}

This section studies the complexity of \ourmethod for finding an $(\epsilon_{\mathrm p},\epsilon_{\mathrm c})$-solution of \eqref{KKT system2}, i.e., a triple $(\bar X,\bar p,\bar \theta)$ that satisfies

\begin{align}
\text{($\epsilon_{\mathrm{p}}$-primal feasibility)} &\qquad
\|\mathcal A(\bar X)-b\| \leq \epsilon_{\mathrm{p}},
\quad \tr(\bar X)\leq 1,
\quad \bar X  \succeq 0, \nonumber\\
\text{(dual feasibility)} &\qquad
\bar S := C+\mathcal A^{*}\bar{p}+\bar\theta I \succeq 0,
\quad \bar\theta \geq 0,\label{Approx KKT}\\
\text{($\epsilon_{\mathrm{c}}$-complementarity)} &\qquad
\inner{\bar X}{S} + \bar\theta (1 - \tr \bar X) \le \epsilon_{\mathrm{c}}. \nonumber
\end{align}

\subsection{Description of \ourmethod and Main Theorem}\label{Desc Our}

The formal description of \ourmethod is presented next.

\noindent\begin{minipage}[t]{1\columnwidth}%
\rule[0.5ex]{1\columnwidth}{1pt}

\noindent \textbf{\ourmethod Method}

\noindent \rule[0.5ex]{1\columnwidth}{1pt}%
\end{minipage}

\noindent \textbf{Input}: Initial points $(U_0, p_0) \in \bar B^{s_0}_1 \times \RR^m$,
tolerance pair $(\epsilon_{\mathrm{c}}, \epsilon_{\mathrm{p}}) \in \R^2_{++}$,
penalty parameter $\beta\in \RR_{++}$,
and ADAP-AIPP parameters $(\bar \rho, \lambda_0) \in \RR_{++}^2$.

\noindent \textbf{Output}: 
 $(\bar X, \bar p, \bar \theta) \in \Delta^{n} \times \RR^m \times \RR_{+}$,
 an $(\epsilon_{\mathrm{p}},\epsilon_{\mathrm{c}})$-solution of \eqref{KKT system2}

\begin{itemize}
\item[{\bf 0.}]
set $t=1$ and
\begin{equation}\label{hat epsilon body}
\bar \epsilon = 
\min\{\epsilon_{\mathrm{c}},\;
{\epsilon_{\mathrm{p}}^2\beta}/{6} \};
\end{equation}

\item[{\bf 1.}]
call HLR with input
$(g, \bar Y_0, \lambda_0, \bar \epsilon, \bar \rho)=(\mathcal L_{\beta}(\cdot;p_{t-1}), U_{t-1}, \lambda_0, \bar \epsilon, \bar \rho)$
where $U_{t-1} \in \bar B^{s_{t-1}}_{1}$,
and let $U_t \in \bar B^{s_t}_{1}$ denote its output $\bar Y$;
\item[{\bf 2.}]
set
\begin{equation}\label{dual update SDP}
p_t=p_{t-1}+\beta (\mathcal A(U_t U_t^T)-b);
\end{equation}
\item[{\bf 3.}]
if $\|\mathcal{A}(U_t U_t^T) - b\| \leq \epsilon_{\mathrm{p}}$,
then set $T=t$ and \textbf{return} $(\bar X, \bar p, \bar\theta) = (U_T U_T^T, p_T, \tilde \theta(U_TU_T^{T};p_{T-1}))$ where $\tilde \theta(\cdot;\cdot)$ is as in \eqref{Theta Lagrangian};
\item[{\bf 4.}]
set $t = t+1$ and \textbf{go to} step \textbf{1.}
\end{itemize}
\noindent \rule[0.5ex]{1\columnwidth}{1pt}

Some remarks about each of the steps in \ourmethod are now given.
First, step 1 invokes HLR to obtain
an $\bar \epsilon$-optimal solution
$U_{t}U^{T}_{t}$
of subproblem \eqref{eq:AL}
using the previous
$U_{t-1} \in \bar B_1^{s_{t-1}}$ as initial point.
Second, step 2 updates the multiplier $p_t$ according to a full Lagrange multiplier update. Third, it is shown in Lemma~\ref{dual feasibility and complementarity} below that the triple $(U_tU_t^{T},p_t,\tilde \theta(U_tU_t^{T};p_{t-1}))$ always satisfies the dual feasibility and $\epsilon_{\mathrm{c}}$-complementarity conditions in \eqref{Approx KKT} where $\tilde \theta(\cdot;\cdot)$ is as in \eqref{Theta Lagrangian}.
Finally, step 3 checks if $U_{t}U^{T}_{t}$ is an $\epsilon_{\mathrm p}$-primal feasible solution. It then follows from the above remarks that if this condition is satisfied, then the triple $(U_tU^{T}_t,p_t,\tilde \theta(U_tU_t^{T};p_{t-1}))$ is an $(\epsilon_{\mathrm{p}},\epsilon_{\mathrm{c}})$-solution of \eqref{Approx KKT}.

\begin{lemma}\label{dual feasibility and complementarity}
For every iteration index $t$, the triple $(U_tU^{T}_t,p_t,\tilde \theta(U_tU_t^{T};p_{t-1}))$ satisfies the dual feasibility and $\epsilon_{\mathrm{c}}$-complementarity conditions in \eqref{Approx KKT} where $\tilde\theta(\cdot;\cdot)$ is as in \eqref{Theta Lagrangian}.
\end{lemma}
\begin{proof}
The definitions of $q(\cdot;\cdot)$ and $\tilde \theta(\cdot;\cdot)$ in \eqref{Theta Lagrangian}, the second relation in \eqref{Computation of Gradient} with $Y=U_t$ and $p=p_{t-1}$, and the update rule for $p_t$ in \eqref{dual update SDP} imply that
\[
\nabla \mathcal L_{\beta}(U_tU^{T}_t;p_{t-1})=C+\mathcal A^{*}p_t, \quad \tilde \theta(U_tU_t^{T};p_{t-1})=\max\{-\lambda_{\min}(C+\mathcal A^{*}p_t), 0\}.
\]
It is then easy to see from the above relation that the triple $(U_tU^{T}_t,p_t,\tilde \theta(U_tU_t^{T};p_{t-1}))$ always satisfies the dual feasibility condition in \eqref{Approx KKT}.

The fact that the definition of $\bar \epsilon$ in \eqref{hat epsilon body} implies that $\bar \epsilon \leq \epsilon_{\mathrm c}$, the fact that $U_tU^{T}_t$ is an $\bar \epsilon$ solution of \eqref{eq:AL}, and the formula for $\nabla \mathcal L_{\beta}(U_tU^{T}_t;p_{t-1})$ above imply that
\begin{align*}
    0 \in C+\mathcal A^{*}p_t+ \partial_{\epsilon_\mathrm{c}} \delta_{\Delta^{n}}(U_tU^{T}_t).
\end{align*}
It then follows immediately from the above inclusion and relation \eqref{Comp Slackness} with $Z=U_tU_t^{T}$, $g=\mathcal L_{\beta}(\cdot;p_{t-1})$, and $\theta(\cdot)=\tilde \theta(\cdot;p_{t-1})$
that the triple $(U_tU^{T}_t,p_t,\tilde\theta(U_tU^{T}_t;p_{t-1}))$ always satisfies the $\epsilon_{\mathrm{c}}$-complementarity condition in \eqref{Approx KKT}.
\end{proof}

Before stating the complexity of \ourmethod,
the following quantities are first introduced:
\begin{align}
&\bar {\mathcal F}_{\beta}:=\frac{\beta}2 \|\mathcal AX_0-b\|^2+\frac{5\|p_{*}\|^2+\|p_{*}\|\sqrt{3\|p_{*}\|^2+2\beta\bar\epsilon}}{\beta}+3\bar \epsilon \label{g upper}\\
&\bar {\mathcal G}_{\beta}:=\|C\|_{F}+\|\mathcal A\|\left(9\beta\|\mathcal A\|+\sqrt{4\|p_{*}\|^2+2\beta\bar \epsilon} +\beta \|b\|\right)\label{grad upper}\\
&\bar \kappa_{\beta}:=\left\lceil \log_0^+\left(\lambda_0\max\{8\bar {\mathcal G}_{\beta}+144 \beta \|A\|^2, 1/\lambda_0\}\right)/\log 2 \right\rceil \label{K upper}
\end{align}
where $p_{*}$ is an optimal dual solution and $\bar \epsilon$ is as in \eqref{hat epsilon body}.

The main result of this paper is now stated.

\begin{thm}\label{thm:OuterComplexitySDP}
The following statements hold:
    \begin{itemize}
\item[a)]
 \ourmethod terminates with
 an $(\epsilon_{\mathrm{p}},\epsilon_{\mathrm{c}})$-solution $(\bar X,\bar p,\bar \theta)$ of the pair of SDPs \eqref{eq:sdp-primal} and \eqref{eq:sdp-dual}
 in at most
    \begin{equation}
\label{eq:Complexity-Result-Body} 
\mathcal J:=\left\lceil\frac{3\|p_{*}\|^2}{\beta^2\epsilon_{\mathrm{p}}^2}\right\rceil
\end{equation}
iterations where $p_{*}$ is an optimal solution to \eqref{eq:sdp-dual};
\item[(b)] \ourmethod performs at most 
\begin{equation}
    \mathcal J \cdot {\bar{\mathcal P}}_{\beta}(\epsilon_{\mathrm{p}},\epsilon_{\mathrm{c}})
\end{equation}
and 
\begin{equation}
\label{Total Complexity HLR-AL method}
\begin{aligned}
\mathcal J\left[(1+{\bar \kappa_{\beta}})\mathcal{\bar P}_{\beta}(\epsilon_{\mathrm{p}},\epsilon_{\mathrm{c}})\right]
&+\frac{C_{\sigma}\bar {\mathcal F}_{\beta}\max\{8\bar {\mathcal G}_{\beta}+144\beta \|A\|^2,1/\lambda_0\}}{\bar \rho^2}
\end{aligned}
\end{equation}
total HLR iterations (and hence MEV computations) and total ADAP-FISTA calls, respectively,
where 
\begin{equation}\label{Bar P}
\mathcal {\bar{P}}_{\beta}(\epsilon_{\mathrm{p}},\epsilon_{\mathrm{c}}):=\left\lceil 1+\frac{4\max\left\{\bar {\mathcal F}_{\beta},\sqrt{4\beta\|A\|^2\bar {\mathcal F}_{\beta}}, 4\beta \|A\|^2\right\}}{\min\{\epsilon_{\mathrm{c}},\;
{\epsilon_{\mathrm{p}}^2\beta}/{6} \}}\right\rceil,
\end{equation}
$\bar \rho$ is an input parameter to \ourmethod,
and $\bar {\mathcal F}_{\beta}$, $C_{\sigma}$, $\bar {\mathcal G}_{\beta}$, and ${\bar \kappa_{\beta}}$ are as in \eqref{g upper}, \eqref{phi* alpha0 quantities}, \eqref{grad upper}, and \eqref{K upper}, respectively.  
\end{itemize}
\end{thm}

Two remarks about Theorem~\ref{thm:OuterComplexitySDP} are now given. First, it follows from statement (a) that
\ourmethod performs $\mathcal O\left(1/(\beta^2\epsilon^2_p)\right)$ iterations. Second, statement (b) and the definitions of $\mathcal J$, $\mathcal{\bar P}_{\beta}(\epsilon_{\mathrm{p}},\epsilon_{\mathrm{c}})$, and $\bar {\mathcal F}_{\beta}$ in \eqref{eq:Complexity-Result-Body}, \eqref{Bar P}, and \eqref{g upper}, respectively imply that 
\ourmethod performs
\begin{equation}\label{Big O HLR iterations 1}
   \mathcal O\left(\frac{1}{\beta\epsilon^2_p\min\{\epsilon_{\mathrm{c}},\epsilon^2_p\beta\}}\right)
\end{equation}
and
\begin{equation}\label{Big O ADAP FISA 1}
   \mathcal O\left(\frac{1}{\beta\epsilon^2_p\min\{\epsilon_{\mathrm{c}},\epsilon^2_p\beta\}}+\frac{\beta^2}{\bar \rho^2}\right)
\end{equation}
total HLR iterations (and hence MEV computations) and ADAP-FISTA calls, respectively. In contrast to the case where $\beta=\mathcal O(1)$,
the result below shows that the bounds \eqref{Big O HLR iterations 1} and \eqref{Big O ADAP FISA 1} can be improved when
$\beta=\mathcal O(1/\epsilon_{\mathrm{p}})$.


\begin{corollary}
If $\beta=\mathcal O\left(1/\epsilon_{\mathrm{p}}\right)$ and $\bar \rho=\beta\min\{\epsilon_{\mathrm{c}},\epsilon^2_p\beta/6\}$, then 
\ourmethod performs at most
\begin{equation}\label{Big O FISTA}
\mathcal O\left(\frac{1}{\epsilon_{\mathrm p}^2}+\frac{1}{\epsilon_{\mathrm c}^2}\right)
\end{equation}
total HLR iterations (and hence MEV computations)
and
total ADAP-FISTA calls.
\end{corollary}

\begin{proof}
The conclusion of the corollary immediately follows from relations \eqref{Big O HLR iterations 1} and \eqref{Big O ADAP FISA 1} together with the assumptions that $\beta=\mathcal O(1/\epsilon_{\mathrm{p}})$ and $\bar \rho=\beta\min\{\epsilon_{\mathrm{c}},\epsilon^2_p\beta/6\}$. 
\end{proof}

It can be shown that each ADAP-FISTA call performs at most
\begin{equation}
\mathcal O_1\left(\sqrt{2\left[1+\lambda_0(2\bar {\mathcal G}_{\beta}+36\beta\|\mathcal A\|^2)\right]}\log^+_1\left(1+\lambda_0\left[2\bar {\mathcal G}_{\beta}+36\beta\|\mathcal A\|^2\right]\right)\right) 
\end{equation}
iterations/resolvent evaluations.\footnote{A resolvent evaluation of $h$ is an evaluation of $(I+\gamma \partial h)^{-1}(\cdot)$ for some $\gamma>0$.}
This is an immediate consequence of Lemma~\ref{ADAP FISTA translated}(a) in Appendix~\ref{ADAP-AIPP Appendix Proof}, the definition of $L_{\tilde g}$ in \eqref{phi* alpha0 quantities}, the fact that $\mathcal L_{\beta}(X;p_{t-1})$ is $(\beta\|\mathcal A\|^2)$-smooth, and Lemma~\ref{Lemma Grad Bound} which is developed in the next subsection.

\subsection{Proof of Theorem~\ref{thm:OuterComplexitySDP}}\label{Pf Main Theorem}
Since \ourmethod calls the HLR method at every iteration,
the next proposition specializes Theorem~\ref{thm:complexityLRHFW},
which states the complexity of HLR,
to the specific case of SDPs.
Its statement uses the following quantities associated
with an iteration $t$ of
\ourmethod:
\begin{align}
&\mathcal F^{(t)}_{\beta}:=\mathcal L_{\beta}(X_{t-1},p_{t-1})-\min_{X\in \Delta^{n}} \mathcal L_{\beta}(X,p_{t-1}),\label{Gap Upper}\\
&{\mathcal G}_{\beta}^{(t)}:=\sup \{\|\nabla \mathcal L_{\beta}(UU^{T},p_{t-1})\|: U\in \bar B^{s_{t-1}}_{3}\}, \label{Gradient Upper Index}\\
&\kappa_{\beta}^{(t)}=\left\lceil \log_0^+\left(\lambda_0\max\{8{\mathcal G}_{\beta}^{(t)}+144 \beta \|A\|^2, 1/\lambda_0\}\right)/\log 2 \right\rceil,\label{Log Iteration Index}
\end{align}

\begin{proposition}\label{HLR translation}
The following statements about the HLR call in step 1 of the $t$-th iteration of the \ourmethod hold: 
\begin{itemize}
\item[(a)]
it outputs $U_t$ such that $X_t=U_tU^{T}_t$ is an $\bar \epsilon$-optimal solution of 
\begin{align*}\label{AL for Prop}
\min_{X \in \Delta^n } \mathcal{L}_{\beta}(X;p_{t-1})
\end{align*}
by performing at most
    \begin{equation}\label{Hybrid FW Method Exact Complexity Translation}
    {\mathcal P^{(t)}_{\beta}}(\bar \epsilon) := \left \lceil 1+\frac{4\max\left\{\mathcal F^{(t)}_{\beta},\sqrt{4\beta\|\mathcal A\|^2\mathcal F^{(t)}_{\beta}}, 4\beta\|\mathcal A\|^2\right\}}{\bar \epsilon} \right\rceil
    \end{equation}
iterations (and hence MEV computations) where $\mathcal F^{(t)}_{\beta}$ and $\bar \epsilon$ are as in \eqref{Gap Upper} and \eqref{hat epsilon body}, respectively;
\item[(b)] the total number of ADAP-FISTA calls within such call is no more than
\begin{align}\label{Total Complexity HLR Translation}
(1+\kappa_{\beta}^{(t)}){\mathcal P^{(t)}_{\beta}}(\bar \epsilon)
+\frac{C_{\sigma}\max\{8{\mathcal G}_{\beta}^{(t)}+144\beta\|\mathcal A\|^2,1/\lambda_0\}}{\bar \rho^2}\left[\mathcal L_{\beta}(X_{t-1},p_{t-1})-\mathcal L_{\beta}(X_{t},p_{t-1})\right]
\end{align}
where 
$\lambda_0$ is given as input to HLR, and  $C_{\sigma}$, ${\mathcal G}_{\beta}^{(t)}$, and $\kappa_{\beta}^{(t)}$ are as in \eqref{phi* alpha0 quantities-2}, \eqref{Gradient Upper Index}, and \eqref{Log Iteration Index}, respectively.
\end{itemize}
\end{proposition}

\begin{proof}
(a) Recall that each HLR call during the $t$-th iteration of \ourmethod is made with
$(g, \bar Y_0, \lambda_0, \bar \epsilon, \bar \rho)=(\mathcal L_{\beta}(\cdot;p_{t-1}), U_{t-1}, \lambda_0, \bar \epsilon, \bar \rho)$.
The result then immediately follows from Theorem~\ref{thm:complexityLRHFW}(a), the definition of $\mathcal F^{(t)}_{\beta}$ in \eqref{Gap Upper}, and the fact that $\mathcal L_{\beta}(X;p_{t-1})$ is $(\beta\|\mathcal A\|^2)$-smooth.

(b) The proof follows directly from  Theorem~\ref{thm:complexityLRHFW}(b), statement(a), the fact that $\mathcal L_{\beta}(\cdot;p_{t-1})$ is $(\beta\|\mathcal A\|^2)$-smooth,
and the definitions of ${\mathcal G}_{\beta}^{(t)}$ and $\kappa^{(t)}_{\beta}$ in \eqref{Gradient Upper Index} and \eqref{Log Iteration Index}, respectively.
\end{proof}

The following Lemma establishes key bounds which will be used later to bound quantities $\mathcal F^{(t)}_{\beta}$, $\mathcal G^{(t)}_{\beta}$, and $\kappa^{(t)}_{\beta}$ that appear in Proposition~\ref{HLR translation}.

\begin{lemma}\label{Lagrange Multipliers bounded}
The following relations hold:
\begin{align}
\max_{t \in \{0, \ldots T\}}\|p_t\|&\leq \|p_{*}\|+\sqrt{3\|p_{*}\|^2+2\beta\bar \epsilon},\label{bound on Lagrange multipliers-SDP}\\
\beta\sum_{t=1}^{T}\|\mathcal{A}X_t-b\|^2&\leq \frac{3\|p_{*}\|^2}{\beta}+2\bar \epsilon, \label{bound on feasibility-SDP}
\end{align}
where $T$ is the last iteration index of \ourmethod, $p_{*}$ is an optimal Lagrange multiplier, and $\bar \epsilon$ is as in \eqref{hat epsilon body}.
\end{lemma}
\begin{proof}
It follows immediately from Proposition~\ref{HLR translation}(a) and the definition of $\bar \epsilon$-optimal solution that the HLR call made in step 1 of the $t$-th iteration of \ourmethod outputs $U_t$ such that $X_t=U_tU_t^{T}$ satisfies
\begin{equation}
    0 \in \nabla \mathcal L_{\beta}(X_t;p_{t-1})+\partial_{\bar \epsilon}\delta_{\Delta^{n}}.
\end{equation}
It is then easy to see that \ourmethod is an instance of the AL Framework in Appendix~\ref{s:ALgeneral} since the HLR method implements relation \eqref{residual} in the Blackbox AL with $(\hat \epsilon_{\mathrm{c}},\hat\epsilon_{\mathrm{d}})=(\bar \epsilon,0)$. The proof of relations \eqref{bound on Lagrange multipliers-SDP} and \eqref{bound on feasibility-SDP} now follows immediately from relations \eqref{bound on Lagrange multipliers} and \eqref{bound on feasibility} and the fact that $p_0=0$.
\end{proof}

\begin{lemma}
For every iteration index $t$ of \ourmethod, the following relations hold:
\begin{align}
&\mathcal F^{(t)}_{\beta}\leq \bar {\mathcal F}_{\beta} \label{Lagrange Gap 1st Relation}\\
&\sum_{l=1}^{t} \left[\mathcal L_{\beta}(X_{l-1},p_{l-1})-\mathcal L_{\beta}(X_{l},p_{l-1})\right] \leq \bar {\mathcal F}_{\beta}\label{Lagrange Summable}
\end{align}
where $X_{t}=U_{t}U^{T}_{t}$ and $\mathcal F^{(t)}_{\beta}$ and $\bar {\mathcal F}_{\beta}$ are as in \eqref{Gap Upper} and \eqref{g upper}, respectively.
\end{lemma}
\begin{proof}
Let $t$ be an iteration index. We first show that
\begin{equation}\label{Upper Bound Lagrangian-2}
\mathcal L_{\beta}(X_{t-1},p_{t-1})\leq \lambda_{\min}(C)+\frac{\beta}2 \|\mathcal A X_0-b\|^2+\frac{3\|p_{*}\|^2}{\beta}+2\bar \epsilon.
\end{equation}
holds.
It follows immediately from the definition of $\mathcal L_{\beta}(X;p)$ in \eqref{AL function}, the fact that $X_0=\argmin_{X\in \Delta^{n}}C\bullet X$, the Cauchy-Schwarz inequality, and the fact that $p_0=0$, that
\begin{align}\label{Lagrangian at 0 Unified}
\mathcal L_{\beta}(X_0,p_0)&\overset{\eqref{AL function}}{\leq}C\bullet X_0+\frac{\beta}{2}\|\mathcal A(X_0)-b\|^2=\lambda_{\min}(C)+\frac{\beta}{2}\|\mathcal A(X_0)-b\|^2.
\end{align}
Hence, relation \eqref{Upper Bound Lagrangian-2} holds with $t=1$.
Suppose now that $t\geq 2$ and let $l$ be an iteration index such that $l<t$. 
Relation \eqref{descent HLR}, the fact that \ourmethod during its $l$-th iteration calls the HLR method with input $g=\mathcal L_{\beta}(X,p_{l-1})$ and initial point $\bar Y_0=U_{l-1}$, and the fact that $X_l=U_lU^{T}_l$ imply that $\mathcal L_{\beta}(X_l,p_{l-1})\leq \mathcal L_{\beta}(X_{l-1},p_{l-1})$
and hence that
\begin{align}\label{successive lagrangian-2}
    \mathcal L_{\beta}(X_l,p_l)-\mathcal L_{\beta}(X_{l-1},p_{l-1})\leq \mathcal L_{\beta}(X_l,p_l)-\mathcal L_{\beta}(X_{l},p_{l-1})\overset{\eqref{dual update SDP}, \eqref{AL function}}{=} \beta \|\mathcal A(X_l)-b\|^2
\end{align}
in view of the update rule \eqref{dual update SDP} and the definition of $\mathcal L_{\beta}(\cdot;\cdot)$ in \eqref{AL function}.
Summing relation \eqref{successive lagrangian-2} from $l=1$ to $t-1$ and using relation \eqref{bound on feasibility-SDP} gives
\begin{align}
    \mathcal L_{\beta}(X_{t-1},p_{t-1})-\mathcal L_{\beta}(X_0,p_0) \overset{\eqref{successive lagrangian-2}}{\leq}\beta \sum_{l=1}^{t-1}\|\mathcal A(X_l)-b\|^2 \overset{\eqref{bound on feasibility-SDP}}{\leq} \frac{3\|p_{*}\|^2}{\beta}+2\bar \epsilon.\label{sum Lagrange Unification}
\end{align}
Relation \eqref{Upper Bound Lagrangian-2} now follows by combining relations \eqref{Lagrangian at 0 Unified} and \eqref{sum Lagrange Unification}.

Now, relation \eqref{bound on Lagrange multipliers-SDP}, the fact that $\min_{X \in \Delta^{n}}C\bullet X=\lambda_{\min}(C)$, and the definition of $\mathcal L_{\beta}(\cdot;\cdot)$ in \eqref{AL function},
imply that
for any $t=0,\ldots,T$,
\begin{align}
 \mathcal L_{\beta}(X,p_{t}) - \lambda_{\min}(C)
    &\geq \mathcal L_{\beta}(X,p_{t}) - C\bullet X
    = \frac{1}{2}\left\|\frac{p_{t}}{\sqrt{\beta}}+\sqrt{\beta}(AX-b)\right\|^2-\frac{\|p_{t}\|^2}{2\beta} \nonumber\\
    &\overset{\eqref{bound on Lagrange multipliers-SDP}}{\geq} -\frac{2\|p_{*}\|^2+\beta\bar\epsilon+\|p_{*}\|\sqrt{3\|p_{*}\|^2+2\beta\bar\epsilon}}{\beta} \qquad \forall X \in \Delta^n\label{lower b}
\end{align}
where $T$ is the last iteration index of \ourmethod. Relations \eqref{Upper Bound Lagrangian-2} and \eqref{lower b} together with the definition of $\bar {\mathcal F}_{\beta}$ in \eqref{g upper} then imply that
$\mathcal L_{\beta}(X_{t-1},p_{t-1})- \mathcal L_{\beta}(X,p_{t-1})\leq \bar {\mathcal F}_{\beta}$ for every
$X \in \Delta^n$ and iteration index $t$.
Relation \eqref{Lagrange Gap 1st Relation} then follows immediately from this conclusion and the definition of $\mathcal F^{(t)}_{\beta}$ in \eqref{Gap Upper}.

To show relation \eqref{Lagrange Summable}, observe that relations \eqref{successive lagrangian-2} and \eqref{bound on feasibility-SDP} imply that for any iteration index $t$ the following relations hold: 
\begin{align*}
&\sum_{l=1}^{t} \mathcal L_{\beta}(X_{l-1},p_{l-1})-\mathcal L_{\beta}(X_{l},p_{l-1})\\ 
&=\sum_{l=1}^{t} \left[\mathcal L_{\beta}(X_{l-1},p_{l-1})-\mathcal L_{\beta}(X_{l},p_{l})\right] + \sum_{l=1}^{t} \left[\mathcal L_{\beta}(X_{l},p_{l})-\mathcal L_{\beta}(X_{l},p_{l-1})\right]\\
&\overset{\eqref{successive lagrangian-2}}{=}\mathcal L_{\beta}(X_0,p_0)-\mathcal L_{\beta}(X_t,p_t)+\beta \sum_{l=1}^{t}\|\mathcal A(X_l)-b\|^2 \overset{\eqref{bound on feasibility-SDP}}{\leq} \mathcal L_{\beta}(X_0,p_0)-\mathcal L_{\beta}(X_t,p_t)+\frac{3\|p_{*}\|^2}{\beta}+2\bar \epsilon.
\end{align*}
Relation \eqref{Lagrange Summable} then follows immediately from the above relation, relation \eqref{lower b} with $X=X_t$, relation \eqref{Lagrangian at 0 Unified}, and the definition of  $\bar {\mathcal F}_{\beta}$ in \eqref{g upper}.
\end{proof}

\begin{lemma}\label{Lemma Grad Bound}
For every iteration $t$ of \ourmethod, we have
${\mathcal G}_{\beta}^{(t)}\leq \bar {\mathcal G}_{\beta}$
where ${\mathcal G}_{\beta}^{(t)}$ and $\bar {\mathcal G}_{\beta}$ are as in \eqref{Gradient Upper Index} and \eqref{grad upper}, respectively.
\end{lemma}
\begin{proof}
Let $t$ be an iteration index of \ourmethod and suppose that $U\in \bar B^{s_{t-1}}_{3}$. It is easy to see from the definition of $\mathcal L_{\beta}(\cdot;\cdot)$ in \eqref{AL function} that
\begin{align}\label{def of gradient Lagrangian}
    \nabla \mathcal L_{\beta}(UU^{T};p_{t-1})=C+\mathcal A^{*}p_{t-1}+\beta \mathcal A^{*}(\mathcal A(UU^{T})-b).
\end{align}
It then follows from the fact that $U \in \bar B^{s_{t-1}}_{3}$, Cauchy-Schwarz inequality, triangle inequality, relation \eqref{def of gradient Lagrangian}, and bound \eqref{bound on Lagrange multipliers-SDP} that
\begin{align*}
\|\nabla \mathcal L_{\beta}(UU^{T};p_{t-1})\|_{F} &\overset{\eqref{def of gradient Lagrangian}}{=}\|C+\mathcal A^{*}p_{t-1}+\beta \mathcal A^{*}(\mathcal A(UU^{T})-b)\|_{F}\\
&\leq \|C\|_{F}+\|\mathcal A\|\|p_{t-1}\|+\beta\|\mathcal A\|^2\|UU^{T}\|_{F}+\beta\|\mathcal A\|\,\|b\|\\
&\leq \|C\|_{F}+\|\mathcal A\|\|p_{t-1}\|+9\beta\|\mathcal A\|^2+\beta\|\mathcal A\|\,\|b\|\\
&\overset{\eqref{bound on Lagrange multipliers-SDP}}{\leq} \|C\|_{F}+\|\mathcal A\|\left(9\beta \|\mathcal A\|+\|p_{*}\|+\sqrt{3\|p_{*}\|^2+2\beta\bar \epsilon} +\beta \|b\|\right)
\end{align*}
which immediately implies the result of the lemma in view of the definitions of ${\mathcal G}_{\beta}^{(t)}$ and $\bar {\mathcal G}_{\beta}$ in \eqref{Gradient Upper Index} and \eqref{grad upper}, respectively.
\end{proof}

We are now ready to prove Theorem~\ref{thm:OuterComplexitySDP}.

\begin{proof}[Proof of Theorem~\ref{thm:OuterComplexitySDP}]
(a) It follows immediately from Lemma~\ref{dual feasibility and complementarity} that output $(\bar X, \bar p, \bar \theta)$ satisfies the the dual feasibility and $\epsilon_{\mathrm{c}}$-complementarity conditions in \eqref{Approx KKT}. It then remains to show that the triple $(\bar X, \bar p, \bar \theta)$ satisfies the $\epsilon_{\mathrm p}$-primal feasibility condition in \eqref{Approx KKT} in at most $\mathcal J$ iterations where $\mathcal J$ is as in \eqref{eq:Complexity-Result-Body}. To show this, it suffices to show that \ourmethod
is an instance of the AL framework analyzed in Appendix~\ref{s:ALgeneral}. Observe first that \eqref{eq:sdp-primal} is a special case of
\eqref{eq:general_probl} with $f(X) = C \bullet X$
and $h(X) = \delta_{\Delta^n}(X)$. It is also easy to see that
the call to the HLR  in step 1 of
\ourmethod is a special way of implementing step 1 of
the AL framework, i.e.,
the Blackbox~AL.  Indeed, Proposition~\ref{HLR translation}(a)
implies that the output $U_t$ of HLR satisfies
$0 \in \nabla \mathcal L_{\beta}(U_tU^{T}_t;p_{t-1})+\partial_{\bar \epsilon} \delta_{\Delta^{n}}(U_t U_t^T)$ and hence HLR implements relation \eqref{residual} in the Blackbox AL with $(\hat X,\hat R)=(U_tU_t^T,0)$, $(\hat \epsilon_{\mathrm{c}},\hat\epsilon_{\mathrm{d}})=(\bar \epsilon,0)$, $g=\mathcal L_{\beta}(\cdot;p_{t-1})$, and $h=\delta_{\Delta^{n}}(\cdot)$.


In view of the facts that \ourmethod
is an instance of the AL framework and $p_0=0$, it then follows immediately from
Theorem~\ref{thm:OuterComplexity} that
\ourmethod terminates within the number of iterations in~\eqref{eq:Complexity-Result-Body} and that
the output $(\bar X, \bar p, \bar \theta)$ satisfies the $\epsilon_{\mathrm p}$-primal feasibility condition in \eqref{Approx KKT}.

(b) Consider the quantity $\mathcal P^{(t)}_{\beta}(\bar \epsilon)$ as in \eqref{Hybrid FW Method Exact Complexity Translation} where $t$ is an iteration index of \ourmethod. It follows immediately from relations \eqref{hat epsilon body} and \eqref{Lagrange Gap 1st Relation} and the definition of $\mathcal{\bar{P}}_{\beta}(\epsilon_{\mathrm{p}},\epsilon_{\mathrm{c}})$ in \eqref{Bar P} that $\mathcal P^{(t)}_{\beta}(\bar \epsilon)\leq \mathcal{\bar{P}}_{\beta}(\epsilon_{\mathrm{p}},\epsilon_{\mathrm{c}})$. Hence, it follows from Proposition~\ref{HLR translation}(a) that each HLR call made in step 1 of \ourmethod performs at most $\mathcal{\bar{P}}_{\beta}(\epsilon_{\mathrm{p}},\epsilon_{\mathrm{c}})$ iterations/MEV computations. The result then follows from this conclusion and part (a).

(c) Let $t$ be an iteration index of \ourmethod. Lemma~\ref{Lemma Grad Bound} implies that ${\mathcal G}_{\beta}^{(t)}\leq \bar {\mathcal G}_{\beta}$ and $\kappa^{(t)}_{\beta} \leq {\bar \kappa_{\beta}}$ in view of the definitions of $\kappa^{(t)}_{\beta}$ and ${\bar \kappa_{\beta}}$ in \eqref{Log Iteration Index} and \eqref{K upper}, respectively. Hence, it follows from this conclusion, the fact that $\mathcal P^{(t)}_{\beta}(\bar \epsilon)\leq \mathcal{\bar{P}}_{\beta}(\epsilon_{\mathrm{p}},\epsilon_{\mathrm{c}})$, Proposition~\ref{HLR translation}(b), and part (a) that the total number of ADAP-FISTA calls performed by \ourmethod is at most
\begin{align*}\label{Total Complexity HLR Bounded}
\mathcal J \left[(1+{\bar \kappa_{\beta}})\right]\mathcal{\bar{P}}_{\beta}(\epsilon_{\mathrm{p}},\epsilon_{\mathrm{c}})
+\frac{C_{\sigma}\max\{8 \bar {\mathcal G}_{\beta}+144\beta\|\mathcal A\|^2,1/\lambda_0\}}{\bar \rho^2}\left[\sum_{t=1}^{T}\mathcal L_{\beta}(X_{t-1},p_{t-1})-\mathcal L_{\beta}(X_{t},p_{t-1})\right]
\end{align*}
where $\mathcal J$ is as in \eqref{eq:Complexity-Result-Body} and $T$ is the last iteration index of \ourmethod. The result in (c) then follows immediately from the above relation together with the fact that relation \eqref{Lagrange Summable} implies that $\sum_{t=1}^{T} \left[\mathcal L_{\beta}(X_{t-1},p_{t-1})-\mathcal L_{\beta}(X_{t},p_{t-1})\right] \leq \bar {\mathcal F}_{\beta}$.
\end{proof}

%

\section{Computational experiments}

In this section the performance of \ourmethod is tested against
state-of-the art SDP solvers.
The experiments are performed on a 2019 Macbook Pro with an 8-core CPU and 32 GB of memory.
The methods are tested in SDPs arising from the following applications:
maximum stable set,
phase retrieval,
and matrix completion.

This section is organized into five subsections.
The first subsection provides details on the implementation of \ourmethod.
The second subsection explains the SDP solvers considered in the experiments.
The remaining subsections describe the results of the computational experiments in each of the applications.

\subsection{Implementation details}
Our implementation of \ourmethod uses the Julia programming language.
The implementation applies to a class of SDPs slightly more general than \eqref{eq:sdp-primal}.
Let $\FF \in \{\RR, \CC\}$ be either the field of real or complex numbers.
Let $\mathbb S^n(\RR)$ (resp.\ $\mathbb S^n(\CC)$) be the space of $n\times n$ symmetric (resp.\ complex Hermitian) matrices,
with Frobenius inner product $\bullet$ and with positive semidefinite partial order~$\succeq$.
The implementation applies to SDPs of the form
\begin{gather*}
\min_{X} \quad \{C \bullet X
\quad : \quad
\mathcal A X=b ,\quad
\tr X  \leq \tau, \quad X \succeq 0, \quad X \in \mathbb S^n(\FF) \}
\end{gather*}
where $b \in \RR^m$, $C\in \mathbb S(\FF)^n$, $\mathcal A: \mathbb S(\FF)^n \to \RR^m$ is a linear map,
and $\mathcal A^*: \RR^m \to \mathbb S^n(\FF)$ is its adjoint.
In contrast to \eqref{eq:sdp-primal}, the trace of $X$ is bounded by~$\tau$ instead of one.
The inputs for our implementation are the initial points $U_0$, $p_0$, the tolerances $\epsilon_{\mathrm{c}}$, $\epsilon_{\mathrm{p}}$,
and the data describing the SDP instance, which is explained below.
In the experiments,
the primal initial point $U_0$ is a random $n \times 1$ matrix with entries generated independently from the Gaussian distribution over~$\FF$,
and the dual initial point $p_0$ is the zero vector.
Our computational results below are based 
on a variant of \ourmethod, also referred to as \ourmethod in this section,
which differs slightly from one described in Section~\ref{HLR-AL Method}
in that the penalty parameter $\beta$ and the tolerance $\bar \epsilon$ for the AL subproblem are chosen in an adaptive manner
based on
some of the ideas of the LANCELOT method \cite{conn1991globally}.
%

The data defining the  SDP instance involves matrices of size $n \times n$
which should not be stored as dense arrays in large scale settings.
Instead of storing a matrix $M\in \mathbb S^n(\FF)$,
it is assumed that a routine that evaluates the linear operator $\mathfrak{L}(M):\FF^n \to \FF^n$, $v \mapsto M v$ is given by the user.



Similar to the Sketchy-CGAL method of \cite{yurtsever2021scalable}, our implementation of \ourmethod requires the following user inputs to describe the SDP instance:
\begin{enumerate}[label=(\roman*)]
    \item The vector $b \in \RR^m$ and the scalar $\tau > 0$.
    \item\label{item:C} A routine for evaluating the linear operator $\mathfrak{L}(C)$.
    \item\label{item:A*} A routine for evaluating linear operators of the form
    $\mathfrak{L}(\mathcal{A}^*p)$ for any $p \in \RR^m$.
    \item\label{item:qA} A routine for evaluating the quadratic function
    \begin{align}\label{eq:qA}
    q_{\mathcal{A}}: \FF^n \to \RR^m,\quad
    y \mapsto \mathcal{A}(y y^T).
    \end{align}
\end{enumerate}

Note that the routine in \ref{item:qA} allows $\mathcal{A}$ to be evaluated on any matrix in factorized form
since $\mathcal{A}(Y Y^T) = \sum_i q_{\mathcal A}(y_i)$
where the sum is over the columns $y_i$ of~$Y$.
In addition, the routines in \ref{item:C} and \ref{item:A*} allow to multiply matrices of the form $C + \mathcal{A}^* p$ with a matrix $Y$
by multiplying by each of the columns $y_i$ separately.
It follows that all steps of \ourmethod (including the steps of HLR and ADAP-AIPP) can be performed by using only the above inputs.
For instance, the eigenvalue computation in Step~2 of HLR is performed using iterative Krylov methods, which only require matrix-vector multiplications.



\subsection{Competing methods}\label{Competing Methods}

We compare \ourmethod against the following SDP solvers:
\begin{itemize}
    \item CSDP : Open source Julia solver based on interior point methods;
    \item COSMO: Open source Julia solver based on ADMM/operator splitting;
    \item SDPLR : Open source MATLAB solver based on Burer and Monteiro's LR method;
    \item SDPNAL+ : Open MATLAB source solver based on AL with a semismooth Newton method;
    \item T-CGAL : Thin variant of the CGAL method~\cite{yurtsever2019conditional} that stores the iterates $X_t$ in factored form;
    \item $r$-Sketchy : Low-rank variant of CGAL that only stores a sketch of~$\pi(X_t) \in \FF^{n\times r}$ of the iterates~$X_t \in \mathbb S^n(\FF)$.
\end{itemize}
We use the default parameters in all methods. The $r$-Sketchy method is tested with two possible values of $r$, namely, $r=10$ and $r = 100$.

Given a tolerance $\epsilon > 0$,
all methods, except SDPLR, stop when a primal solution $X \in \Delta^n$ and a dual solution
$(p,\theta,S) \in \mathbb R^{m} \times \mathbb R_{+} \times \mathbb S^{+}_{n}$
satisfying
\begin{align}\label{eq:termination}
 \frac{\|\mathcal A X - b\|}{1 + \|b\|} \leq \epsilon, \qquad \frac{|\mathrm{pval} - \mathrm{dval}|}{1 + |\mathrm{pval}| + |\mathrm{dval}|} \leq \epsilon,\qquad
 \frac{\|C + \mathcal A^* p + \theta I - S\|}{1 + \|C\|}\leq \epsilon,
\end{align}
is generated, 
where $\mathrm{pval}$ and $\mathrm{dval}$ denote the primal and dual values of the solution, respectively. SDPLR terminates solely based off primal feasibility, i.e., it terminates when the first condition in \eqref{eq:termination} is satisfied.
The above conditions are standard in the SDP literature,
although some solvers use the $l_\infty$ norm instead of the Euclidean norm.
Given a vector $p \in \mathbb R^{m}$, \ourmethod, SDPLR, r-Sketchy, and T-CGAL,
set $\theta:= \max\{-\lambda_{\min}(C+\mathcal A^{*}p) , 0\}$ and
$S:=C+\mathcal A^{*}p+\theta I$. The definition of $\theta$ implies that
$S \succeq 0$ and that the left-hand side of the last inequality in \eqref{eq:termination} is zero.

Recall that a description of $r$-Sketchy is already given in the paragraph preceding the part titled "Structure of the paper" in the Introduction.
We now comment on
how this
method keeps track of its iterates and how it terminates.
First, it never stores
the iterate $U_t$ but only a
rank~$r$ approximation $\tilde U_t$ of it as already mentioned in the aforementioned paragraph of the Introduction.
(We refer to $U_t$ as the implicit iterate as is never computed and $\tilde U_t$ as the computed one.)
Second, it keeps track of
the quantities
$C \bullet (U_tU_t^{T})$ and $\mathcal A(U_t U^{T}_t)$ which, as can be easily verified,
allow the first two relative errors in  \eqref{eq:termination} with $X=U_tU_t^T$ to be easily evaluated.
Third, we kept the termination criterion for $r$-Sketchy code intact in
that it still stops when all three errors in \eqref{eq:termination} computed with the implicit solution $U_t$, i.e., with $X=U_tU_t^{T}$, are less than or equal to $\varepsilon$.
The fact that $r$-Sketchy terminates based on the implicit solution does not guarantee, and is even unlikely, that it would terminate based on the computed solution.
Fourth, if $r$-Sketchy does not terminate within the time limit specified for each problem class, the maximum of the three errors in \eqref{eq:termination} at its final computed solution
$\tilde U_t$, i.e., with $X=\tilde U_t\tilde U_t^{T}$,
is reported.

The solver COSMO includes an optional chordal decomposition pre-processing step (thoroughly discussed in \cite{vandenberghe2015chordal, kojima, Wolk}), which has not been invoked in the computational experiments reported below. 
This ensures that all solvers are compared over the same set of SDP instances.

The tables in the following three subsections present the results of the computational experiments performed on large collections of
maximum stable set, phase retrieval, and matrix completion, SDP instances.
A relative tolerance $\epsilon=10^{-5}$ is set and a time limit of either 10000 seconds ($\approx 3$ hours) or 14400 seconds ($=$ 4 hours) is given. An entry of a table marked with $*/N$
(resp., $**$) means that the corresponding method finds an approximate solution (resp., crashed) with relative accuracy 
strictly larger than the desired accuracy $10^{-5}$ in which case $N$
expresses the maximum of the three final relative accuracies
in \eqref{eq:termination}. For $r$-Sketchy, entries marked as $*/N$ mean that it did not terminate within the time limit and the maximum of the three final relative accuracies in \eqref{eq:termination} of its final computed solution $\tilde U_t$ was $N$. The bold numbers in the tables indicate the algorithm that had the best runtime for that instance.

\subsection{Maximum stable set}

\begin{table}[!tbh]
\captionsetup{font=scriptsize}
\begin{centering}
\begin{tabular}{>{\centering}p{2.9cm}|>{\centering}p{1.4cm}>{\centering}p{1.4cm}>{\centering}p{1.5cm}>{\centering}p{1.6cm}>{\centering}p{1.4cm}>{\centering}p{1.4cm}>{\centering}p{1.4cm}}
\multicolumn{1}{c}{\textbf{\scriptsize{Problem Instance}}} & \multicolumn{7}{c}{\textbf{\scriptsize{}Runtime (seconds)}} \tabularnewline
\toprule
\scriptsize{Graph($n$; $|E|$)} & {\scriptsize{}\ourmethod} & {\scriptsize{}T-CGAL} & {\scriptsize{}10-Sketchy} & {\scriptsize{}100-Sketchy} & {\scriptsize{}CSDP} & {\scriptsize{}COSMO} & {\scriptsize{}SDPNAL+}\tabularnewline
\midrule
\scriptsize{G1(800; 19,176)} & \scriptsize{218.23} & \scriptsize{*/.13e-02} & \scriptsize{*/.31e-01} & \scriptsize{*/.31e-01} & \scriptsize{226.01} & \scriptsize{322.91} & \scriptsize{\textbf{60.30}} \tabularnewline



\scriptsize{G10(800; 19,176)} & \scriptsize{241.51} &{\scriptsize{*/.20e-02}} & {\scriptsize{*/.78e-01}} & {\scriptsize{*/.31e-01}}& {\scriptsize{220.44}} &  {\scriptsize{229.22}} & {\scriptsize{\textbf{55.30}}} \tabularnewline

\hline

\scriptsize{G11(800; 1,600)} & {\scriptsize{\textbf{3.01}}} &{\scriptsize{1631.45}} & {\scriptsize{220.59}} & {{\scriptsize{234.76}}}& {{\scriptsize{3.74}}} &  {\scriptsize{118.66}} & {\scriptsize{73.50}} \tabularnewline

\scriptsize{G12(800; 1,600)} & {\scriptsize{\textbf{3.06}}} &{\scriptsize{414.27}} & {\scriptsize{72.56}} & {{\scriptsize{57.03}}}& {{\scriptsize{3.12}}} &  {\scriptsize{9531.99}} & {\scriptsize{70.20}} \tabularnewline


\hline
\scriptsize{G14(800; 4,694)} & {\scriptsize{66.46}} &{\scriptsize{*/.27e-03}} & {\scriptsize{6642.00}} & {{\scriptsize{7231.30}}}& {{\scriptsize{\textbf{9.31}}}} &  {\scriptsize{3755.36}} & {\scriptsize{115.30}} \tabularnewline


\scriptsize{G20(800; 4,672)} & {\scriptsize{439.37}} &{\scriptsize{*/.37e-03}} & {\scriptsize{*/.12e-01}} & {\scriptsize{*/.12e-01}}& {\scriptsize{\textbf{11.42}}} &  {\scriptsize{*/.69e-04}} & {\scriptsize{341.20}} \tabularnewline


\hline
\scriptsize{G43(1000; 9,990)} & \scriptsize{96.79} &{\scriptsize{*/.25e-04}} & {\scriptsize{*/.51e-01}} & {\scriptsize{*/.26e-01}}& {\scriptsize{\textbf{42.39}}} &  {\scriptsize{*/.10e-02}} & {\scriptsize{62.10}} \tabularnewline

\scriptsize{G51(1,000; 5,909)} & {\scriptsize{190.98}}&{\scriptsize{*/.23e-03}} & {\scriptsize{*/.98e-02}} & {\scriptsize{*/.99e-02}}& {\scriptsize{\textbf{16.97}}} &  {\scriptsize{*/.33e-04}} & {\scriptsize{284.40}} \tabularnewline

\hline


\scriptsize{G23(2,000; 19,990)} & {\scriptsize{390.34}} &{\scriptsize{*/.64e-03}} & {\scriptsize{*/.86e-01}} & {\scriptsize{*/.19e-01}}& {\scriptsize{\textbf{288.77}}} &  {\scriptsize{5739.78}} & {\scriptsize{503.80}} \tabularnewline


\scriptsize{G31(2,000; 19,990)} & {\scriptsize{357.75}} &{\scriptsize{*/.68e-03}} & {\scriptsize{*/.20e-01}} & {\scriptsize{*/.19e-01}}& {\scriptsize{\textbf{290.72}}} &  {\scriptsize{5946.84}} & {\scriptsize{498.30}} \tabularnewline

\hline 

\scriptsize{G32(2,000; 4,000)} & {\scriptsize{\textbf{3.62}}} &{\scriptsize{3732.70}} & {\scriptsize{329.17}} & {{\scriptsize{349.73}}}& {{\scriptsize{29.04}}} &  {\scriptsize{*/.58e-03}} & {\scriptsize{853.90}} \tabularnewline


\scriptsize{G34(2,000; 4,000)} & \scriptsize{\textbf{3.52}} &{\scriptsize{1705.60}} & {\scriptsize{162.85}} & {{\scriptsize{177.84}}}& {{\scriptsize{29.04}}} &  {\scriptsize{458.11}} & {\scriptsize{1101.80}} \tabularnewline

\hline 
\scriptsize{G35(2,000; 11,778)} & \scriptsize{730.54} &{\scriptsize{*/.78e-03}} & {\scriptsize{*/.62e-02}} & {\scriptsize{*/.62e-02}}& {\scriptsize{730.54}} &  {\scriptsize{\textbf{120.60}}} & {\scriptsize{2396.60}} \tabularnewline



\scriptsize{G41(2,000; 11,785)} & \scriptsize{555.02} &{\scriptsize{*/.17e-02}} & {\scriptsize{*/.59e-02}} & {\scriptsize{*/.59e-02}}& {\scriptsize{\textbf{114.73}}} &  {\scriptsize{*/.37e-03}} & {\scriptsize{2027.20}} \tabularnewline

\hline 

\scriptsize{G48(3,000; 6,000)} &\scriptsize{\textbf{3.49}} &{\scriptsize{4069.30}} & {\scriptsize{288.64}} & {{\scriptsize{306.91}}}& {{\scriptsize{81.97}}} &  {\scriptsize{1840.91}} & {\scriptsize{6347.50}} \tabularnewline

\hline

\scriptsize{G55(5,000; 12,498)} & \scriptsize{\textbf{253.22}} &{\scriptsize{*/.33e-02}} & {\scriptsize{*/.80e-02}} & {\scriptsize{*/.79e-02}}& {\scriptsize{535.57}} &  {\scriptsize{*/.11e-02}} & {\scriptsize{*/.17e-01}} \tabularnewline

\scriptsize{G56(5,000; 12,498)} & \scriptsize{\textbf{264.46}} &{\scriptsize{*/.38e-02}} & {\scriptsize{*/.80e-02}} & {\scriptsize{*/.79e-02}}& {\scriptsize{523.06}} &  {\scriptsize{*/.27e-02}} & {\scriptsize{*/.20e00}} \tabularnewline

\hline

\scriptsize{G57(5,000; 10,000)} & \scriptsize{\textbf{4.14}} &{\scriptsize{7348.50}} & {\scriptsize{791.75}} & {{\scriptsize{831.06}}}& {{\scriptsize{336.93}}} &  {\scriptsize{8951.40}} & {\scriptsize{*/.10e00}} \tabularnewline

\hline

\scriptsize{G58(5,000; 29,570)} & \scriptsize{2539.83} &{\scriptsize{*/.96e-02}} & {\scriptsize{*/.24e-01}} & {\scriptsize{*/.43e-02}}& {\scriptsize{\textbf{2177.03}}} &  {\scriptsize{*/.20e-03}} & {\scriptsize{*/.43e-01}} \tabularnewline

\scriptsize{G59(5,000; 29,570)} & \scriptsize{2625.47} &{\scriptsize{*/.54e-02}} & {\scriptsize{*/.24e-01}} & {\scriptsize{*/.43e-02}}& {\scriptsize{\textbf{2178.33}}} &  {\scriptsize{*/.37e+02}} & {\scriptsize{*/.65e-01}} \tabularnewline

\hline
\scriptsize{G60(7,000; 17,148)} & \scriptsize{\textbf{476.65}} & \scriptsize{*/.21e-02} & \scriptsize{*/.13e-01} & \scriptsize{*/.67e-02} & \scriptsize{2216.50} & \scriptsize{*/.39e+02} &\scriptsize{*/.10e+01} \tabularnewline

\scriptsize{G62(7,000; 14,000)} &\scriptsize{\textbf{5.00}}  & \scriptsize{*/.28e-03} & \scriptsize{1795.50} & {\scriptsize{1474.40}}  & {\scriptsize{1463.18}} & \scriptsize{*/.12e-03} & \scriptsize{*/.10e+01} \tabularnewline

\scriptsize{G64(7,000; 41,459)} & {\scriptsize{\textbf{3901.68}}} &{\scriptsize{*/.16e-01}} & {\scriptsize{*/.36e-02}}& {\scriptsize{*/.36e-02}} &  {\scriptsize{7127.72}} & {\scriptsize{*/.99e+01}} & {\scriptsize{*/.10e+01}}\tabularnewline
\hline

\scriptsize{G66(9,000; 18,000)} & {\scriptsize{\textbf{5.77}}} &{\scriptsize{*/.82e-01}} & {\scriptsize{2788.70}}& {{\scriptsize{3022.70}}} &  {{\scriptsize{2076.17}}} &  {\scriptsize{*/.77e-03}} &  {\scriptsize{*/.10e+01}}\tabularnewline

\scriptsize{G67(10,000; 20,000)} & {\scriptsize{\textbf{5.87}}} &{\scriptsize{*/.13e-01}} & {\scriptsize{3725.80}} & {{\scriptsize{3941.70}}} &  {{\scriptsize{7599.80}}} & {\scriptsize{**}} &  {\scriptsize{*/.47e+01}} \tabularnewline

\hline

\scriptsize{G72(10,000; 20,000)} & {\scriptsize{\textbf{5.92}}} &{\scriptsize{*/.11e00}} & {\scriptsize{3936.30}} & {{\scriptsize{3868.60}}} &  {{\scriptsize{7450.01}}}  & {\scriptsize{**}} &  {\scriptsize{*/.47e+01}} \tabularnewline

\scriptsize{G77(14,000; 28,000)} & {\scriptsize{\textbf{8.08}}} &{\scriptsize{*/.24e00}} & {\scriptsize{*/.60e-02}} & {\scriptsize{*/.60e-02}} &  {\scriptsize{**}} & {\scriptsize{**}} &  {\scriptsize{*/.99e00}}\tabularnewline

\scriptsize{G81(20,000; 40,000)} & {\scriptsize{\textbf{10.89}}} &{\scriptsize{*/.10e00}} & {\scriptsize{*/.91e-01}} & {\scriptsize{*/.71e-01}} &  {\scriptsize{**}} & {\scriptsize{**}} &  {\scriptsize{*/.10e+01}}\tabularnewline

\scriptsize{tor(69,192; 138,384)} & {\scriptsize{\textbf{40.64}}} &{\scriptsize{*/.38e00}} & {\scriptsize{*/.63e00}} & {\scriptsize{*/.29e00}} &  {\scriptsize{**}} & {\scriptsize{**}} &  {\scriptsize{**}}\tabularnewline

\bottomrule
\end{tabular}
\par\end{centering}
\caption{{Runtimes (in seconds) for the Maximum Stable Set problem. A relative tolerance of $\epsilon=10^{-5}$ is set and a time limit of 10000 seconds is given.  An entry marked with $*/N$
(resp., $**$) means that the corresponding method finds an approximate solution (resp., crashed) with relative accuracy 
strictly larger than the desired accuracy in which case $N$
expresses the maximum of the three relative accuracies
in \eqref{eq:termination}.}}\label{tab:StableSetSmall}
\end{table}

Given a graph $G = ([n],E)$,
the maximum stable set problem consists of finding
a subset of vertices of largest cardinality
such that no two vertices are connected by an edge.
Lov\'asz \cite{lovasz1979shannon} introduced a constant, the $\vartheta$-function, which upper bounds the value of the maximum stable set.
The $\vartheta$-function is the value of the SDP
\begin{align}\label{eq:lovasz}
    \max \quad \{{e}{e}^T \bullet X \; : \;
    X_{ij} = 0, \; ij \in E, \;
    \tr X = 1, \;
    X \succeq 0,\quad
    X \in \mathbb S^n(\RR)
    \}
\end{align}
where $e = (1,1,\dots,1) \in \RR^n$ is the all ones vector.
It was shown in \cite{grotschel1984polynomial} that the $\vartheta$-function agrees exactly with the stable set number for perfect graphs.

\begin{table}[!tbh]
\captionsetup{font=scriptsize}
\begin{centering}
\begin{tabular}{>{\centering}p{4.5cm}|>{\centering}p{1.5cm}>{\centering}p{1.5cm}>{\centering}p{1.5cm}>{\centering}p{1.7cm}}
\multicolumn{1}{c}{\textbf{\scriptsize{Problem Instance}}} & \multicolumn{4}{c}{\textbf{\scriptsize{}Runtime (seconds)}} \tabularnewline
\toprule
\scriptsize{Graph($n$; $|E|$)} & {\scriptsize{}\ourmethod} & {\scriptsize{}T-CGAL}  & {\scriptsize{10-Sketchy}} & {\scriptsize{100-Sketchy}} \tabularnewline
\midrule

\scriptsize{$H_{13,2}$(8,192; 53,248)} &{\scriptsize{\textbf{5.04}}} & {\scriptsize{*/.23e00}}& {{\scriptsize{1603.80}}} & {{\scriptsize{882.03}}}    \tabularnewline

\scriptsize{$H_{14,2}$(16,384; 114,688)} & {\scriptsize{\textbf{9.09}}} &  {\scriptsize{*/.45e00}} & {{\scriptsize{6058.60}}} & {{\scriptsize{6712.20}}}  \tabularnewline

\scriptsize{$H_{15,2}$(32,768; 245,760)} & {\scriptsize{\textbf{65.22}}} & {\scriptsize{*/.19e00}} & {\scriptsize{*/.19e-01}} & {\scriptsize{*/.14e-01}}   \tabularnewline

\scriptsize{$H_{16,2}$(65,536; 524,288)} & {\scriptsize{\textbf{104.71}}} & {\scriptsize{*/.11e-01}} & {\scriptsize{*/.24e00}} & {\scriptsize{*/.11e-01}}  \tabularnewline

\hline

\scriptsize{$H_{17,2}$(131,072; 1,114,112)} & {\scriptsize{\textbf{69.63}}} & {\scriptsize{*/.34e-01}}  & {\scriptsize{*/.72e00}} & {\scriptsize{*/.32e-01}}  \tabularnewline

\scriptsize{$H_{18,2}$(262,144; 2,359,296)} & {\scriptsize{\textbf{244.90}}} &  {\scriptsize{*/.99e-02}} & {\scriptsize{*/.88e-02}} & {\scriptsize{*/.31e00}}   \tabularnewline

\scriptsize{$H_{19,2}$(524,288; 4,980,736)} & {\scriptsize{\textbf{786.73}}} & {\scriptsize{*/.42e00}}  & {\scriptsize{*/.35e00}} & {\scriptsize{*/.24e00}}  \tabularnewline

\scriptsize{$H_{20,2}$(1,048,576; 10,485,760)} &{\scriptsize{\textbf{1157.96}}} & {\scriptsize{*/.47e00}}   & {\scriptsize{*/.31e-02}} & {\scriptsize{*/.31e-02}} \tabularnewline

\bottomrule
\end{tabular}
\par\end{centering}
\caption{Runtimes (in seconds) for the Maximum Stable Set problem. A relative tolerance of $\epsilon=10^{-5}$ is set and a time limit of 14400 seconds (4 hours) is given.  An entry marked with $*/N$
(resp., $**$) means that the corresponding method finds an approximate solution (resp., crashed) with relative accuracy 
strictly larger than the desired accuracy in which case $N$
expresses the maximum of the three relative accuracies
in \eqref{eq:termination}.}
\label{tab:StableSetLarge}
\end{table}

\begin{table}[!tbh]
\captionsetup{font=scriptsize}
\begin{centering}
\begin{tabular}{>{\centering}p{4.5cm}|>{\centering}p{2.5cm}}
\multicolumn{1}{c}{\textbf{\scriptsize{Problem Instance}}} & \multicolumn{1}{c}{\textbf{\scriptsize{}Runtime (seconds)}} \tabularnewline
\toprule
\scriptsize{Graph($n$; $|E|$)} & {\scriptsize{}\ourmethod} \tabularnewline
\midrule

\scriptsize{$H_{21,2}$(2,097,152; 22,020,096)} &  {\scriptsize{{2934.33}}} \tabularnewline

\scriptsize{$H_{22,2}$(4,194,304; 46,137,344)} & {\scriptsize{6264.50}} \tabularnewline

\scriptsize{$H_{23,2}$(8,388,608; 96,468,992)} &{\scriptsize{{14188.23}}} \tabularnewline

\scriptsize{$H_{24,2}$(16,777,216; 201,326,592)} &{\scriptsize{{46677.82}}} \tabularnewline

\bottomrule
\end{tabular}
\par\end{centering}
\caption{Runtimes (in seconds) for the Maximum Stable Set problem. A relative tolerance of $\epsilon=10^{-5}$ is set.}
\label{tab:StableSetHuge}
\end{table}

\begin{table}[!tbh]
\captionsetup{font=scriptsize}
\begin{centering}
\begin{tabular}{>{\centering}p{2.6cm}>{\centering}p{2.6cm} >{\centering}p{2.2cm}|>{\centering}p{2.5cm}}
\multicolumn{3}{c}{\textbf{\scriptsize{Problem Instance}}} & \multicolumn{1}{c}{\textbf{\scriptsize{}Runtime (seconds)}} \tabularnewline
\toprule
\scriptsize{Problem Size $(n; m)$} & \scriptsize{Graph} &\scriptsize{Dataset} & {\scriptsize{\ourmethod}} \tabularnewline
\midrule

\scriptsize{10,937; 75,488} & \scriptsize{wing\_nodal} & \scriptsize{DIMACS10} & {\scriptsize{1918.48}}  \tabularnewline

\scriptsize{16,384; 49,122} & \scriptsize{delaunay\_n14} & \scriptsize{DIMACS10} & {\scriptsize{1355.01}}  \tabularnewline

\scriptsize{16,386; 49,152} & \scriptsize{fe-sphere} & \scriptsize{DIMACS10} & {\scriptsize{147.93}}  \tabularnewline

\scriptsize{22,499; 43,858} & \scriptsize{cs4} & \scriptsize{DIMACS10} & {\scriptsize{747.66}}  \tabularnewline

\scriptsize{25,016; 62,063} & \scriptsize{hi2010} & \scriptsize{DIMACS10} & \scriptsize{3438.06}  \tabularnewline

\scriptsize{25,181; 62,875} & \scriptsize{ri2010} & \scriptsize{DIMACS10} & \scriptsize{2077.97}  \tabularnewline

\scriptsize{32,580; 77,799} & \scriptsize{vt2010} & \scriptsize{DIMACS10} & \scriptsize{2802.37}  \tabularnewline

\scriptsize{48,837; 117,275} & \scriptsize{nh2010} & \scriptsize{DIMACS10} & \scriptsize{8530.38}  \tabularnewline

\hline

\scriptsize{24,300; 34,992} & \scriptsize{aug3d} & \scriptsize{GHS\_indef} & {\scriptsize{8.56}}  \tabularnewline

\scriptsize{32,430; 54,397} & \scriptsize{ia-email-EU} & \scriptsize{Network Repo} & {\scriptsize{530.21}}  \tabularnewline

\hline

\scriptsize{11,806; 32,730} & \scriptsize{Oregon-2} & \scriptsize{SNAP} & {\scriptsize{2787.19}}  \tabularnewline

\scriptsize{11,380; 39,206} & \scriptsize{wiki-RFA\_negative} & \scriptsize{SNAP} & {\scriptsize{1151.31}}  \tabularnewline

\scriptsize{21,363; 91,286} & \scriptsize{ca-CondMat} & \scriptsize{SNAP} & {\scriptsize{7354.75}}  \tabularnewline

\scriptsize{31,379; 65,910} & \scriptsize{as-caida\_G\_001} & \scriptsize{SNAP} & {\scriptsize{3237.93}}  \tabularnewline

\scriptsize{26,518; 65,369} & \scriptsize{p2p-Gnutella24} & \scriptsize{SNAP} & {\scriptsize{344.83}}  \tabularnewline

\scriptsize{22,687; 54,705} & \scriptsize{p2p-Gnutella25} & \scriptsize{SNAP} & {\scriptsize{235.03}}  \tabularnewline

\scriptsize{36,682; 88,328} & \scriptsize{p2p-Gnutella30} & \scriptsize{SNAP} & {\scriptsize{542.07}}  \tabularnewline

\scriptsize{62,586; 147,892} & \scriptsize{p2p-Gnutella31} & \scriptsize{SNAP} & {\scriptsize{1918.30}}  \tabularnewline

\hline 
\scriptsize{49,152; 69,632} & \scriptsize{cca} & \scriptsize{AG-Monien} & {\scriptsize{47.24}}  \tabularnewline

\scriptsize{49,152; 73,728} & \scriptsize{ccc} & \scriptsize{AG-Monien} & {\scriptsize{12.14}}  \tabularnewline

\scriptsize{49,152; 98,304} & \scriptsize{bfly} & \scriptsize{AG-Monien} & {\scriptsize{13.15}}  
\tabularnewline

\scriptsize{16,384; 32,765} & \scriptsize{debr\_G\_12} & \scriptsize{AG-Monien} & {\scriptsize{818.61}}  
\tabularnewline

\scriptsize{32,768; 65,533} & \scriptsize{debr\_G\_13} & \scriptsize{AG-Monien} & {\scriptsize{504.29}}  \tabularnewline

\scriptsize{65,536; 131,069} & \scriptsize{debr\_G\_14} & \scriptsize{AG-Monien} & {\scriptsize{466.67}}  \tabularnewline

\scriptsize{131,072; 262,141} & \scriptsize{debr\_G\_15} & \scriptsize{AG-Monien} & {\scriptsize{488.07}}  \tabularnewline

\scriptsize{262,144; 524,285} & \scriptsize{debr\_G\_16} & \scriptsize{AG-Monien} & {\scriptsize{1266.71}}  \tabularnewline

\scriptsize{524,288; 1,048,573} & \scriptsize{debr\_G\_17} & \scriptsize{AG-Monien} & {\scriptsize{5793.57}}  \tabularnewline

\scriptsize{1,048,576; 2,097,149} & \scriptsize{debr\_G\_18} & \scriptsize{AG-Monien} & {\scriptsize{13679.12}}  \tabularnewline

\bottomrule
\end{tabular}
\par\end{centering}
\caption{Runtimes (in seconds) for the Maximum stable set problem. A relative tolerance of $\epsilon=10^{-5}$ is set.}\label{StableSetReal/NetworkData}
\end{table}

\begin{table}[!tbh]
\captionsetup{font=scriptsize}
\begin{centering}
\begin{tabular}{>{\centering}p{3.4cm}|>{\centering}p{1.5cm}>{\centering}p{2.5cm}}
\multicolumn{1}{c}{\textbf{\scriptsize{Problem Instance}}} & \multicolumn{2}{c}{\textbf{\scriptsize{}Runtime (seconds)}} \tabularnewline
\toprule
\scriptsize{Graph($n$; $|E|$)} & {\scriptsize{\ourmethod}} & {\scriptsize{SDPLR}} \tabularnewline
\midrule

\scriptsize{$H_{10,2}$(1024; 5120)} &\scriptsize{2.90} & \scriptsize{\textbf{1.28}} \tabularnewline

\scriptsize{$H_{11,2}$(2048; 11264)} & {\scriptsize{\textbf{3.03}}} & \scriptsize{10.14} \tabularnewline

\scriptsize{$H_{12,2}$(4096; 24576)} & {\scriptsize{\textbf{3.49}}} & {\scriptsize{56.60/.12e-03}}\tabularnewline

\scriptsize{$H_{13,2}$(8192; 53248)} &{\scriptsize{\textbf{5.04}}} & {\scriptsize{399.89/.38e-03}} \tabularnewline

\scriptsize{$H_{14,2}$(16384; 114688)} & {\scriptsize{\textbf{9.09}}} & {\scriptsize{2469.11/.16e-02}}\tabularnewline

\scriptsize{$H_{15,2}$(32768; 245760)} & {\scriptsize{\textbf{65.22}}} & {\scriptsize{*/.11e-01/.46e00}} \tabularnewline

\bottomrule
\end{tabular}
\par\end{centering}
\caption{Runtimes (in seconds) for the maximum stable set problem. A relative tolerance of $\epsilon=10^{-5}$ is set and a time limit of 14400 seconds (4 hours) is given. An entry marked with */N1/N2 means that the corresponding method finds an approximate solution that satisfies the first relation in \eqref{eq:termination} with relative accuracy 
strictly larger than the desired accuracy of $10^{-5}$ in which case $N1$ (resp., $N2$)
expresses the final accuracy that the method satisfies the first (resp., second) relation in \eqref{eq:termination} with.}\label{tab:SDPLR}
\end{table}

Tables \ref{tab:StableSetSmall}, \ref{tab:StableSetLarge},  \ref{tab:StableSetHuge}, \ref{StableSetReal/NetworkData}, and \ref{tab:SDPLR} present the results of the computational experiments performed on the maximum stable set~SDP. Table~\ref{tab:StableSetSmall} compares \ourmethod against all the methods listed in Subsection~\ref{Competing Methods} on smaller graph instances, i.e., with number of vertices not exceeding 70,000. All graph instances considered, except the last instance, are taken from the GSET data set, a curated collection of randomly generated graphs that can be found in \cite{gset}. 
The larger GSET graphs (GSET 66-81) are all toroidal graphs where every vertex has degree 4. The last graph instance presented in Table~\ref{tab:StableSetSmall} is a large toroidal graph with approximately 70,000 vertices that we generated ourselves.
A time limit of 10000 seconds (approximately 3 hours) is given. {Table~\ref{tab:StableSetLarge} compares \ourmethod against T-CGAL, $10$-Sketchy, and 100-Sketchy, on large graph instances with up to 1 million vertices and 10 million edges. CSDP
was not included in
Table~\ref{tab:StableSetLarge}
since it crashed on all but one of the instances included in it.} All graph instances considered in Table~\ref{tab:StableSetLarge} are Hamming $H(d,2)$ graphs, a special class of graphs that has $2^d$ number of vertices and $d\, 2^{d-1}$ number of edges. The vertex set of such graphs can be seen as corresponding to binary words of length~$d$,
and the edges correspond to binary words that differ in one bit.
A time limit of 14400 seconds (4 hours) is now given.

Tables \ref{tab:StableSetHuge} and \ref{StableSetReal/NetworkData} solely present the performance of \ourmethod on extremely large-sized Hamming instances (i.e., with number of vertices exceeding 2 millon) and hard graph instances from real-world datasets, respectively. The graph instances considered in Table~\ref{StableSetReal/NetworkData} are taken from the DIMACS10, Stanford SNAP, AG-Monien, GHS\_indef, and Network Repositories \cite{snapnets, nr, FloridaMatrix, DIMACS10}.

Table~\ref{tab:SDPLR} displays a special comparison between \ourmethod and SDPLR on 6 different Hamming graphs. Recall that SDPLR terminates only based off the first condition in \eqref{eq:termination} and hence often finds a solution that does not satisfy the second condition in \eqref{eq:termination} with the desired accuracy of $\epsilon=10^{-5}$. An entry marked with time/N in Table~\ref{tab:SDPLR}
means that the corresponding method finds an approximate solution (within the time limit) that satisfies the first relation in \eqref{eq:termination} with $\epsilon=10^{-5}$ but does not satisfy the second relation in \eqref{eq:termination} with the desired accuracy in which case $N$
expresses the final accuracy that the method satisfies the second relation in \eqref{eq:termination} with. An entry marked with */N1/N2 means that SDPLR finds an approximate solution that satisfies the first relation in \eqref{eq:termination} with relative accuracy 
strictly larger than the desired accuracy of $10^{-5}$ in which case $N1$ (resp., $N2$)
expresses the final accuracy that SDPLR satisfies the first (resp., second) relation in \eqref{eq:termination} with.

Remarks about the results presented in Tables \ref{tab:StableSetSmall}, \ref{tab:StableSetLarge}, \ref{tab:StableSetHuge}, \ref{StableSetReal/NetworkData}, and \ref{tab:SDPLR} are now given. As seen from Table~\ref{tab:StableSetSmall}, CSDP and \ourmethod are the two best performing methods on these smaller graph instances. HALLaR, however, is the only method that can solve each of the instances to the desired accuracy of $10^{-5}$ within the time limit of approximately 3 hours. On graph instances where the number of vertices exceeds 14,000 (resp., 10,000), CSDP (resp., COSMO) cannot perform a single iteration within 3 hours or crashes. SDPNAL+ crashed with a lack of memory error on the last graph instance with 69,192 vertices. Although T-CGAL, 10-Sketchy, and 100-Sketchy do not perform especially well on the smaller graph instances considered in Table~\ref{tab:StableSetSmall}, they are included for comparison on the larger graph instances considered in Table~\ref{tab:StableSetLarge} since they require considerably less memory than CSDP, COSMO, and SDPNAL+. The results presented in Table~\ref{tab:StableSetLarge} demonstrate that \ourmethod performs especially well for larger instances as it is the only method that can solve all instances within the time limit of 4 hours. T-CGAL, 10-Sketchy, and 100-Sketchy cannot find a solution with the desired accuracy of $10^{-5}$ on most of the instances considered, often finding solutions with accuracies on the range of $10^{0}$ to $10^{-2}$. CSDP was tested on the problems considered in Table~\ref{tab:StableSetLarge} but not included for comparison since it crashed on every instance except one. COSMO and SDPNAL+ are not included for comparison due to their high memory requirements.

Tables~\ref{tab:StableSetHuge} and \ref{StableSetReal/NetworkData} show that \ourmethod can solve extremely large Hamming instances and hard real-world instances, respectively, within a couple of hours. As seen from Table~\ref{tab:StableSetHuge}, \ourmethod can solve a Hamming instance with 4 million vertices and 40 million edges (resp. 16 million vertices and 200 million edges) in under 2 hours (resp., 13 hours). Table~\ref{StableSetReal/NetworkData} shows that \ourmethod can solve a huge Debruijin graph instance (which arises in the context of genome assembly) in just a few hours.

The results presented in Table \ref{tab:SDPLR} display the superior performance of \ourmethod compared to SDPLR on six different Hamming graphs.  \ourmethod not only finds more accurate solutions than SDPLR within the time limit of 4 hours but is also at least 80 times faster than SDPLR on the three largest instances.

\subsection{Phase retrieval}\label{Phase Retrieval SDP}

Given $m$ pairs $\{(a_i,b_i)\}_{i=1}^m \subseteq \CC^n \times \RR_+$, consider the problem of finding a vector $x \in \mathbb C^n$ such that
\begin{align*}
    |\langle a_i, x \rangle|^2 = b_i,
    \quad i=1,\dots,m.
\end{align*}
In other words, the goal is to retrieve $x$ from the magnitude of $m$ linear measurements.
By creating the complex Hermitian matrix $X = x x^H$,
this problem can be approached by solving the complex-valued SDP relaxation
\[
\min_X \quad \left\{\tr(X)\quad
:\quad \langle a_i a_i^H , X\rangle  = b_i,\quad
X \succeq 0,\quad
X \in \mathbb S^n(\CC)
\right\}.
\]
The motivation of the trace objective function is that it promotes obtaining a low rank solution.
It was shown in \cite{candes2013phaselift}
that the relaxation is tight (i.e., the vector $x$ can be retrieved from the SDP solution $X$)
when the vectors $a_i$ are sampled independently and uniformly on the unit sphere.
Notice that this class of SDPs does not have a trace bound.
However, since the objective function is precisely the trace,
any bound on the optimal value can be used as the trace bound.
In particular, the squared norm of the vector $x$ is a valid trace bound.
Even though $x$ is unknown, bounds on its norm are known (see for example \cite{yurtsever2015scalable}).

Computational experiments are performed on the synthetic data set from \cite{yurtsever2021scalable}
that is based on the coded diffraction pattern model from \cite{candes2015phase}.
Given $n$, the hidden solution vector ${x} \in \CC^n$ is generated from the complex standard normal distribution.
The are $m = 12 n$ measurements that are indexed by pairs $(j,l) \in [12] \times [n]$.
Consider vectors $y_j \in \CC^n$ for $j \in [12]$,
where the entries of $y_j$ are products of of two independent random variables:
the first is the uniform distribution on $\{1, i, -1, -i\}$, 
and the second chooses from $\{ \sqrt{2}/2, \sqrt{3}\}$ with probabilities $4/5$ and $1/5$.
The linear measurements correspond to modulating the vector ${x}$ with each of the $y_j$'s and then taking a discrete Fourier transform:
\begin{align*}
    \inner{a_{j,k}}{x} := \mathrm{DFT}(y_j \circ x)_l \text{ for } j \in [12],\ l\in [n]
\end{align*}
where $\circ$ denotes the Hadamard product,
and $\mathrm{DFT}(\cdot)_l$ denotes the $l$-th entry of the discrete Fourier transform.
The vector $b$ is obtained by applying the measurements to ${x}$.
The trace bound is set as $\tau = 3 n$, similarly as in \cite{yurtsever2021scalable}.

Tables~\ref{tab:PhaseRetMed} and \ref{PhaseRetLarge} present the results of the computational experiments performed on the phase retrieval~SDP. As mentioned in the above paragraph, all instances considered are taken from a synthetic dataset that can be found in \cite{yurtsever2021scalable}. Table~\ref{tab:PhaseRetMed} compares \ourmethod against T-CGAL, $10$-Sketchy, and 100-Sketchy on medium sized phase retrieval instances, i.e., the dimension $n$ is either 10000 or 31623. The ranks of the outputted solutions of \ourmethod and T-CGAL are now also reported. For entries corresponding to \ourmethod and T-CGAL, the number reported after the last forward slash 
indicates the rank of that corresponding method's outputted solution.
A time limit of 14400 seconds (4 hours) is given. Table~\ref{PhaseRetLarge} solely presents the performance of \ourmethod on larger sized phase retrieval instances, i.e., with dimension $n$ greater than or equal to 100,000. The rank of the outputted solution of \ourmethod is again reported.



\begin{table}[!tbh]
\captionsetup{font=scriptsize}
\begin{centering}
\begin{tabular}{>{\centering}p{3.4cm}|>{\centering}p{1.7cm}>{\centering}p{2.3cm}>{\centering}p{1.7cm}>{\centering}p{1.7cm}}
\multicolumn{1}{c}{\textbf{\scriptsize{Problem Instance}}} & \multicolumn{4}{c}{\textbf{\scriptsize{}Runtime (seconds)}} \tabularnewline
\toprule
\scriptsize{Problem Size $(n; m)$} & {\scriptsize{}\ourmethod} & 
{\scriptsize{}T-CGAL} &
{\scriptsize{}10-Sketchy} & {\scriptsize{}100-Sketchy}  \tabularnewline
\midrule

\scriptsize{10,000; 120,000} &\scriptsize{\textbf{69.11/2}} & \scriptsize{*/.18e-01/561} & \scriptsize{*/.13e-01}  & \scriptsize{*/.25e-01}\tabularnewline

\scriptsize{10,000; 120,000} & {\scriptsize{\textbf{66.14/2}}} & \scriptsize{*/.54e-01/521} & \scriptsize{11112.00}  & \scriptsize{*/.80e-01}\tabularnewline

\scriptsize{10,000; 120,000} & {\scriptsize{\textbf{64.42/2}}} & {\scriptsize{*/.12e00/224}}& {\scriptsize{*/.31e-01}} &  {\scriptsize{*/.13e00}} \tabularnewline

\scriptsize{10,000; 120,000} &{\scriptsize{\textbf{99.98/2}}} & {\scriptsize{*/.28e-01/201}} & {\scriptsize{*/.13e00}} &  {\scriptsize{*/.26e-01}}  \tabularnewline

\hline 

\scriptsize{31,623; 379,476} & {\scriptsize{\textbf{620.82/3}}} & {\scriptsize{*/.29e00/1432}}& {\scriptsize{*/.77e-01}} &  {\scriptsize{*/.23e00}} \tabularnewline

\scriptsize{31,623; 379,476} & {\scriptsize{\textbf{982.34/2}}} & {\scriptsize{*/.23e00/729}} & {\scriptsize{*/.63e-01}} &  {\scriptsize{*/.93e00}}\tabularnewline


\scriptsize{31,623; 379,476} & {\scriptsize{\textbf{870.25/2}}} & {\scriptsize{*/.66e00/794}} & {\scriptsize{*/.65e-02}} &  {\scriptsize{*/.78e-01}}  \tabularnewline

\scriptsize{31,623; 379,476} & {\scriptsize{\textbf{712.09/2}}} & {\scriptsize{*/.10e+01/1280}} & {\scriptsize{*/.10e+01}} &  {\scriptsize{*/.82e00}}  \tabularnewline

\bottomrule
\end{tabular}
\par\end{centering}
\caption{Runtimes (in seconds) for the Phase Retrieval problem. A relative tolerance of $\epsilon=10^{-5}$ is set and a time limit of 14400 seconds (4 hours) is given.  An entry marked with $*/N$ means that the corresponding method finds an approximate solution with relative accuracy 
strictly larger than the desired accuracy in which case $N$
expresses the maximum of the three relative accuracies
in \eqref{eq:termination}. For entries corresponding to \ourmethod and T-CGAL, the number reported after the last forward slash indicates that the rank of that corresponding method's outputted solution.}\label{tab:PhaseRetMed}
\end{table}

\begin{table}[!tbh]
\captionsetup{font=scriptsize}
\begin{centering}
\begin{tabular}{>{\centering}p{3.4cm}|>{\centering}p{2.5cm}}
\multicolumn{1}{c}{\textbf{\scriptsize{Problem Instance}}} & \multicolumn{1}{c}{\textbf{\scriptsize{}Runtime (seconds)}} \tabularnewline
\toprule
\scriptsize{Problem Size $(n; m)$} & {\scriptsize{}\ourmethod} \tabularnewline
\midrule

\scriptsize{100,000; 1,200,000} & {\scriptsize{1042.92/4}}  \tabularnewline

\scriptsize{100,000; 1,200,000} & {\scriptsize{1147.46/3}} \tabularnewline

\scriptsize{100,000; 1,200,000} & {\scriptsize{929.67/5}} \tabularnewline

\scriptsize{100,000; 1,200,000} & {\scriptsize{939.23/5}} \tabularnewline

\hline

\scriptsize{316,228; 3,794,736} &\scriptsize{8426.94/5}  \tabularnewline

\scriptsize{316,228; 3,794,736} & {\scriptsize{2684.83/1}} \tabularnewline

\scriptsize{316,228; 3,794,736} & {\scriptsize{7117.31/6}} \tabularnewline

\scriptsize{316,228; 3,794,736} &{\scriptsize{7489.42/7}}  \tabularnewline

\hline

\scriptsize{3,162,278; 37,947,336} & {\scriptsize{40569.10/1}}  \tabularnewline

\bottomrule
\end{tabular}
\par\end{centering}
\caption{Runtimes (in seconds) for the Phase Retrieval problem. The number after the forward slash indicates the rank of \ourmethod's outputted solution. A relative tolerance of $\epsilon=10^{-5}$ is set.} \label{PhaseRetLarge}
\end{table}

Table~\ref{tab:PhaseRetMed} only compares \ourmethod against T-CGAL and Sketchy-CGAL since these are the only methods that take advantage of the fact that the linear maps $\mathcal{A}$ and $\mathcal{A}^*$ in the phase retrieval SDP can be evaluated efficiently  using the fast Fourier transform (FFT).
As seen from Table~\ref{tab:PhaseRetMed}, \ourmethod is the best performing method and the only method that can solve each instance to a relative accuracy of $10^{-5}$ within the time limit of 4 hours.
T-CGAL and Sketchy-CGAL were unable to solve most instances to the desired accuracy, often finding solutions with accuracies on the range of $10^{0}$ to $10^{-2}$ in $4$ hours. Sketchy-CGAL was also over 150 times slower than \ourmethod on the single instance that it was able to find a $10^{-5}$ accurate solution.

Since T-CGAL and Sketchy-CGAL did not perform well on the medium sized phase retrieval instances considered in Table~\ref{tab:PhaseRetMed}, computational results
for large sized phase retrieval instances are only presented 
for \ourmethod
in Table~\ref{PhaseRetLarge}.
The results presented in Table~\ref{PhaseRetLarge} show that \ourmethod solves a phase retrieval SDP instance with dimension pair $(n,m)\approx (10^5,10^6)$ in approximately $15$ minutes and also one with dimension pair $(n,m) \approx (10^{6},10^{7})$ in just $11$ hours.

\subsection{Matrix completion}\label{Matrix Completion}

Consider the problem of retrieving a low rank matrix $M \in \RR^{n_1 \times n_2}$, where $n_1 \leq n_2$,
by observing a subset of its entries: $M_{ij}$, $ij\in \Omega$.
A standard approach to tackle this problem is by considering the nuclear norm relaxation:
\[
\min_Y \quad \left\{\|Y\|_*\quad
:\quad Y_{ij} = M_{ij}, \ \forall\, ij \in \Omega,\quad
Y \in \RR^{n_1\times n_2}
\right \}
\]
The above problem can be rephrased as the following SDP:
\begin{align}
\label{eq:matcomp-sdp}
\min_{X} \quad
\Bigl\{ \,\frac12 \tr(X) \quad : \quad
X= \begin{pmatrix} W_1 & Y \\ Y^T & W_2 \end{pmatrix} \succeq 0, \quad
Y_{i,j} = M_{i,j} \; \forall\, ij \in \Omega,\quad
X \in \mathbb S^{n_1+n_2}(\RR)
\Bigr\}
\end{align}

\begin{table}[!tbh]
\captionsetup{font=scriptsize}
\begin{centering}
\begin{tabular}{>{\centering}p{3.0cm} >{\centering}p{1.0cm}|>{\centering}p{1.5cm}>{\centering}p{2.5cm}}
\multicolumn{2}{c}{\textbf{\scriptsize{Problem Instance}}} & \multicolumn{2}{c}{\textbf{\scriptsize{}Runtime (seconds)}} \tabularnewline
\toprule
\scriptsize{Problem Size $(n; m)$} & \scriptsize{$r$} & {\scriptsize{}\ourmethod} & {\scriptsize{}10-Sketchy}\tabularnewline
\midrule

\scriptsize{10,000; 828,931}  & \scriptsize{3} &\scriptsize{\textbf{321.81}} & \scriptsize{*/.81e00/.89e-02}\tabularnewline

\scriptsize{10,000; 828,931} & \scriptsize{3} & {\scriptsize{\textbf{332.54}}} & \scriptsize{*/.80e00/.82e-02}\tabularnewline

\scriptsize{10,000; 2,302,586} & \scriptsize{5} & {\scriptsize{\textbf{1117.60}}} & {\scriptsize{*/.92e00/.28e00}}\tabularnewline

\scriptsize{10,000; 2,302,586} & \scriptsize{5} & {\scriptsize{\textbf{1067.15}}} & {\scriptsize{*/.11e+01/ .41e00}} \tabularnewline

\hline 

\scriptsize{31,623; 2,948,996} &  \scriptsize{3} &{\scriptsize{\textbf{1681.03}}} & {\scriptsize{*/.81e00/.69e-02}}\tabularnewline

\scriptsize{31,623; 2,948,996} &  \scriptsize{3} & {\scriptsize{\textbf{1362.22}}} & {\scriptsize{*/.81e00/.82e-02}} \tabularnewline


\scriptsize{31,623; 8,191,654} &  \scriptsize{5} & {\scriptsize{\textbf{4740.48}}} & {\scriptsize{*/.90e00/.43e-01}} \tabularnewline

\scriptsize{31,623; 8,191,654} &  \scriptsize{5} &{\scriptsize{\textbf{5238.57}}} & {\scriptsize{*/.90e00/.84e-01}} \tabularnewline

\bottomrule
\end{tabular}
\par\end{centering}
\caption{Runtimes (in seconds) for the Matrix Completion problem. A relative tolerance of $\epsilon=10^{-5}$ is set and a time limit of 14400 seconds (4 hours) is given.  An entry marked with $*/N_1/N_2$ means that the implicit solution corresponding to 10-Sketchy had relative accuracy 
strictly larger than the desired accuracy in which case $N_1$ (resp. $N_2$)
expresses the maximum of the three relative accuracies
in \eqref{eq:termination} of its computed (resp. implicit) solution.}\label{tab:MatrixCompMed}
\end{table}

\begin{table}[!tbh]
\captionsetup{font=scriptsize}
\begin{centering}
\begin{tabular}{>{\centering}p{3.5cm} >{\centering}p{1.0cm}|>{\centering}p{2.5cm}}
\multicolumn{2}{c}{\textbf{\scriptsize{Problem Instance}}} & \multicolumn{1}{c}{\textbf{\scriptsize{}Runtime (seconds)}} \tabularnewline
\toprule
\scriptsize{Problem Size $(n; m)$} & \scriptsize{$r$}  & {\scriptsize{}\ourmethod} \tabularnewline
\midrule

\scriptsize{75,000; 3,367,574} &\scriptsize{2} & {\scriptsize{3279.85}}  \tabularnewline

\scriptsize{75,000; 7,577,040} &\scriptsize{3} & {\scriptsize{5083.68}} \tabularnewline

\hline 

\scriptsize{100,000; 4,605,171} &\scriptsize{2} & {\scriptsize{2872.44}} \tabularnewline

\scriptsize{100,000; 10,361,633} &\scriptsize{3} & {\scriptsize{6048.63}} \tabularnewline

\hline

\scriptsize{150,000; 7,151,035} &\scriptsize{2} & {\scriptsize{10967.74}} \tabularnewline

\scriptsize{150,000; 16,089,828} &\scriptsize{3} & {\scriptsize{14908.08}} \tabularnewline

\hline 

\scriptsize{200,000; 9,764,859} &\scriptsize{2} & {\scriptsize{13454.12}} \tabularnewline

\scriptsize{200,000; 21,970,931} &\scriptsize{3} & {\scriptsize{28021.56}} \tabularnewline

\bottomrule
\end{tabular}
\par\end{centering}
\caption{Runtimes (in seconds) for the Matrix Completion problem. A relative tolerance of $\epsilon=10^{-5}$ is set.}\label{MatrixComplLarge}
\end{table}

The nuclear norm relaxation was introduced in \cite{fazel2002matrix}.
It was shown in \cite{candes2012exact} it provably completes the matrix when $m = |\Omega|$ is sufficiently large
and the indices of the observations are independent and uniform.

Similar to the SDP formulation of phase retrieval in subsection~\ref{Phase Retrieval SDP}, the SDP formulation of matrix completion does not include a trace bound,
but the objective function is a multiple of the trace.
Hence, any bound on the optimal value leads to a trace bound.
In particular, a valid trace bound is $2 \|Y_0\|_*$,
where $Y_0 \in \RR^{n_1\times n_2}$ is the trivial completion, which agrees with $M_{ij}$ in the observed entries and has zeros everywhere else.
However, computing the nuclear norm of $Y_0$ is expensive, as it requires an SVD decomposition. 
In the experiments the inexpensive, though weaker, bound 
$\tau = 2 \sqrt{n_1} \|Y_0\|_F$ is used instead.

The matrix completion instances are generated randomly, using the following procedure.
Given $r \leq n_1 \leq n_2$,
the hidden solution matrix $M$ is the product $U V^T$,
where the matrices $U \in \RR^{n_1 \times r}$ and $V \in \RR^{n_2 \times r}$ have independent standard Gaussian random variables as entries.
Afterwards, $m$ independent and uniformly random observations from $M$ are taken.
The number of observations is
$m = \lceil \gamma\ r (n_1 + n_2 - r) \rceil$
where
$\gamma=r\log(n_1+n_2)$
is the oversampling ratio.

Tables~\ref{tab:MatrixCompMed} and~\ref{MatrixComplLarge} present the results of the computational experiments performed on the matrix completion~SDP. All instances are generated randomly using the procedure described in the previous paragraph. Table~\ref{tab:MatrixCompMed} compares \ourmethod against $10$-Sketchy on medium sized matrix completion instances, i.e., the dimension $n=n_1+n_2$ is either 10000 or 31623. A time limit of 14400 seconds (4 hours) is given. On instances where $10$-Sketchy did not terminate within the time limit, the relative accuracy of both of its computed and implicit solutions are now reported. An entry marked with $*/N_1/N_2$ means that, within 4 hours, the implicit solution corresponding to 10-Sketchy had relative accuracy 
strictly larger than the desired accuracy in which case $N_1$ (resp. $N_2$)
expresses the maximum of the three relative accuracies
in \eqref{eq:termination} of its computed (resp. implicit) solution. Table~\ref{MatrixComplLarge} solely presents the performance of \ourmethod on larger sized matrix completion instances, i.e., with dimension $n$ greater than or equal to 75000.

Table~\ref{tab:MatrixCompMed} only compares \ourmethod against 10-Sketchy due to 10-Sketchy's low memory requirements and its superior/comparable performance to 100-Sketchy and T-CGAL on previous problem classes. As seen from Table~\ref{tab:MatrixCompMed}, \ourmethod is the best performing method and the only method that can solve each instance to a relative accuracy of $10^{-5}$ within the time limit of 4 hours. 
10-Sketchy is unable to solve a single instance to the desired accuracy, often finding solutions with accuracies on the range of $10^{0}$ to $10^{-2}$ in 4 hours.

Since 10-Sketchy did not perform well on the medium sized matrix completion instances considered in Table~\ref{tab:MatrixCompMed}, computational results
for large sized matrix completion instances are only presented 
for \ourmethod
in Table~\ref{MatrixComplLarge}.
The results presented in Table~\ref{MatrixComplLarge} show that \ourmethod solves a matrix completion instance with dimension pair $(n,m)\approx (10^5,10^6)$ in approximately $48$ minutes and also one with dimension pair $(n,m) \approx (10^{5},10^{7})$ in just $1.7$ hours.

\appendix

\section{Technical Results}

The following section states some useful facts about the spectraplex $\Delta^{n}$ defined in \eqref{Delta Definition}. The first result characterizes the optimal solution a given linear form over the set $\Delta^{n}$. The second result characterizes its $\epsilon$-normal cone.

\subsection{Characterization of Optimal Solution of Linear Form over Spectraplex}

Consider the problem
\begin{equation}\label{subproblem appendix}
\min_U \{G \bullet U: U \in \Delta^{n}\},
\end{equation}
where $\Delta^{n}$ is as in \eqref{Delta Definition}.
The optimality condition for \eqref{subproblem appendix} implies that $Z$ is an optimal solution of
\eqref{subproblem appendix} if and only if there exists
$\theta \in \R$ such that
the pair $(Z,\theta)$ satisfies
\begin{equation}\label{optimality fw sub App}
G+ \theta I \succeq 0 ,
\quad (G + \theta I) \bullet Z = 0, \quad \theta \ge 0, \quad \theta (\tr(Z)-1)=0.
\end{equation}
The following proposition explicitly characterizes solutions of the above problem using the special structure of $\Delta^{n}$.

\begin{prop}\label{Pair Optimal FW subproblem}
Let $(\lambda_{\min},v_{\min})$ be a minimum eigenpair of $G$ and define
\begin{equation}\label{theta X}
\theta^{F} = \max\{-\lam_{\min}(G) , 0\}, \quad Z^F= \begin{cases}
     v_{\min}v_{\min}^{T} & \text{ if $\theta^{F}>0$,}\\
        0 & \text{ otherwise}.
\end{cases}
\end{equation}
Then the pair $(Z,\theta)=(Z^F,\theta^F)$ satisfies  \eqref{optimality fw sub App}.
As a consequence,
$Z^F$ is an optimal solution of \eqref{subproblem appendix}.
\end{prop}

\begin{proof}
Consider the pair $(Z^F,\theta^F)$ as defined in \eqref{theta X}. It is immediate from the definition of $\theta^{F}$ that $\theta^{F} \geq 0$. Consider now two cases. For the first case, suppose that $\theta^{F}=0$ and hence $Z^F=0$. It then follows immediately that the pair $(Z^F,\theta^F)$ satisfies the optimality conditions in \eqref{optimality fw sub App}. 

For the second case, suppose that $\theta^{F}>0$. Thus $\theta^{F}=-\lam_{\min}(G)$ and $Z^{F}=v_{\min}v_{\min}^{T}$. Clearly, then $G+\theta^{F}I\succeq 0$ and $\tr(Z^{F})=1$ and hence the pair $(Z^{F},\theta^{F})$ satisfies the first and fourth relations in \eqref{optimality fw sub App}. The fact that $(-\theta^{F}, v_{\min})$ is an eigenpair implies that $G\bullet Z^{F}=-\theta^F$ and hence the pair $(Z^{F},\theta^{F})$ also satisfies the second relation in \eqref{optimality fw sub App}.

\end{proof}
\subsection{Characterization of $\epsilon$-Normal Cone of Spectraplex}

\begin{prop}\label{Normal Cone Characterization Appendix}
Let $Z \in \Delta^{n}$ and $\theta=\theta^{F}$ where $\theta^{F}$ is as in \eqref{theta X}. Then:
\begin{itemize}
\item[a)] there hold
   \begin{equation}\label{PSD dual Appendix}
 \theta \ge 0, \quad
       G+\theta I \succeq 0;
    \end{equation}
\item[b)] for any $\epsilon>0$,
the inclusion holds
\begin{equation}\label{Normal Cone inclusion}
0 \in G + N^{\epsilon}_{\Delta^n}(Z)
\end{equation}
if and only if
\begin{equation}\label{Epsilon Normal Cone}
   G\bullet Z +\theta \leq \epsilon.
   \end{equation}
\end{itemize}
\end{prop}
\begin{proof}
(a) It follows immediately from the definition of $\theta^{F}$ in \eqref{theta X} and the fact that $\theta=\theta^{F}$ that the two relations in \eqref{PSD dual Appendix} hold.

(b) Proposition~\ref{Pair Optimal FW subproblem} implies that the pair $(Z^{F},\theta^{F})$ satisfies the relation $G\bullet Z^{F}=-\theta^{F}$ where $(Z^{F},\theta^{F})$ is as in \eqref{theta X}. It then follows from this relation, the fact that $\theta=\theta^{F}$, and the inclusions $-G \in N^\epsilon_{\Delta^{n}}(Z)$ and $Z^{F} \in \Delta^{n}$ that 
    \begin{align*}
        \epsilon \geq -G \bullet (Z^F - Z)
        = G\bullet Z+\theta.
    \end{align*}
For the other direction, suppose the pair $(Z, \theta)$ satisfies \eqref{Epsilon Normal Cone} and let $U \in \Delta^{n}$. It then follows that
\begin{align*}
        -G \bullet (U - Z)
        &= G \bullet Z-(G+\theta I)\bullet U+\theta \tr(U)\\
        &\leq G\bullet Z+\theta-(G+\theta I)\bullet U\\
        & \leq \epsilon -0=\epsilon.
    \end{align*}
Hence, $-G \in N^{\epsilon}_{\Delta^n}(Z)$, proving the result.
\end{proof}

\section{ADAP-FISTA Method}\label{ADAP-FISTA}
Let $\mathbb E$ denote a finite-dimensional inner product real vector space with inner product and induced norm denoted by $\inner{\cdot}{\cdot}$ and $\|\cdot\|$, respectively. Also, let $\psi_n : \mathbb{E} \to (-\infty, \infty]$ be a proper closed convex function whose domain $\dom \psi_{n}:=\mathcal N \subseteq \mathbb E$, has finite diameter $D_{\psi_{n}}$.

ADAP-FISTA considers the following problem
\begin{equation}\label{OptProb1}
\min\{\psi(u):= \psi_s(u) + \psi_n(u) : u \in \mathbb E\}
\end{equation}
where $\psi_s$ is assumed to satisfy the following assumption:

\begin{itemize}
\item[\textbf{(B1)}]
$\psi_s$ is a real-valued function that is differentiable on $\mathbb E$ and
there exists $\bar L \geq 0$ such that
\begin{equation}\label{ineq:uppercurvature1}
\|\nabla \psi_s(u') -  \nabla \psi_s(u)\|\le \bar L \|u'-u\| \quad \forall u,u' \in \tilde B,
\end{equation}
where 
\begin{equation}\label{Closed Ball}
\tilde B:= \mathcal N+\bar B_{D_{\psi_{n}}}
\end{equation}
and
$\bar B_{l}:=\{u: \|u\|\leq l\}$ is the closed unit ball centered at $0$ with radius $l$.
\end{itemize}




We now describe the type of approximate solution that ADAP-FISTA aims to find. 

\vgap

\noindent{\bf Problem:}
Given $\psi$ satisfying the above assumptions,
a point $x_0 \in \mathcal N$, a parameter $\sigma\in(0,\infty)$,
the problem is to find a pair $(y,r) \in \mathcal N \times \mathbb E$ such that
\begin{gather}
    \|r\|\leq \sigma \|y-x_0\|, \quad r \in \nabla \psi_s(y)+\partial \psi_n(y).\label{acg problem}
\end{gather}
We are now ready to present the ADAP-FISTA algorithm below. 

\noindent\begin{minipage}[t]{1\columnwidth}%
\rule[0.5ex]{1\columnwidth}{1pt}

\noindent \textbf{ADAP-FISTA Method}

\noindent \rule[0.5ex]{1\columnwidth}{1pt}%
\end{minipage}
\noindent \textbf{Universal Parameters}: $\sigma>0$ and $\chi \in (0,1)$.

\noindent \textbf{Input}: initial point $x_0 \in \mathcal N$,  scalars $\mu >0$, $L_0>\mu$, and function pair $(\psi_s,\psi_n)$.
\begin{itemize}
\item[{\bf 0.}] set $y_0=x_0 $, $ A_0=0 $, $\tau_0=1$, and $i=0$;
\item[{\bf 1.}] Set $L_{i+1}=L_i$;
\item[{\bf 2.}]	Compute
\begin{equation}\label{def:ak-sfista1}
a_i=\frac{\tau_{i}+\sqrt{\tau_{i}^2+4\tau_{i} A_{i}(L_{i+1}-\mu)}}{2(L_{i+1}-\mu)}, \quad \tx_{i}=\frac{A_{i}y_{i}+a_i x_i}{A_i+a_i},
\end{equation}
\begin{equation}
y_{i+1}:=\underset{u\in \mathcal N}\argmin\left\lbrace q_i (u;\tx_{i},L_{i+1}) 
:= \ell_{\psi_s}(u;\tilde x_{i}) + \psi_n(u) + \frac{L_{i+1}}{2}\|u-\tx_{i}\|^2\right\rbrace,
\label{eq:ynext-sfista1}
\end{equation}
If the inequality
\begin{equation}\label{ineq check}
\ell_{\psi_s}(y_{i+1};\tilde x_{i})+\frac{(1-\chi) L_{i+1}}{4}\|y_{i+1}-\tilde x_{i}\|^2\geq \psi_s(y_{i+1})
\end{equation}
holds go to step~3;
else set $L_{i+1} \leftarrow 2L_{i+1} $ and repeat step~2;
\item[{\bf 3.}] Compute
\begin{align}
A_{i+1}&=A_i+a_i, \quad \tau_{i+1}= \tau_i + a_i\mu,  \label{eq:taunext-sfista1} \\
s_{i+1}&=(L_{i+1}-\mu)(\tilde x_i-y_{i+1}),\label{eq:sk}\\
\quad x_{i+1}&= \frac{1}{\tau_{i+1}} \left[\mu a_i y_{i+1} + \tau_i x_i-a_is_{i+1} \right];\label{eq:xnext-sfista1}
\end{align}
	
\item[{\bf 4.}]  If the inequality
\begin{align}
&\|y_{i+1}-x_{0}\|^{2} \geq \chi A_{i+1}L_{i+1} \|y_{i+1}-\tilde x_{i}\|^2, \label{ineq: ineq 1}
\end{align}
holds, then  go to step 5; otherwise, stop with {\bf failure};

\item[{\bf 5.}] Compute 
\begin{equation}\label{def:uk}
v_{i+1}=\nabla \psi_s(y_{i+1})-\nabla \psi_s(\tilde x_{i})+L_{i+1}(\tilde x_i-y_{i+1}).   \end{equation}
If the inequality
\begin{equation}\label{u sigma criteria}
\|v_{i+1}\| \leq \sigma \|y_{i+1}-x_0\|
\end{equation}
holds then stop with {\bf success} and output $(y,v,L):=(y_{i+1},v_{i+1}, L_{i+1})$; otherwise,
 $ i \leftarrow i+1 $ and go to step~1.
\end{itemize}
\noindent \rule[0.5ex]{1\columnwidth}{1pt}

The ADAP-FISTA method was first developed in \cite{SujananiMonteiro}. The method assumes that the gradient of $\psi_s$ is Lipschitz continuous on all of $\mathbb E$ since it requires Lipchitz continuity of the gradient at the sequences of points $\{\tilde x_i\}$ and $\{y_i\}$. This assumption can be relaxed to as in (B3) by showing that the sequence $\{\tilde x_i\}$ lies in the set $\tilde B$ defined in \eqref{Closed Ball}. The following lemma establishes this result inductively by using several key lemmas which can be found in \cite{SujananiMonteiro}. 
\begin{lem}\label{Key proposition}
    Let $m \geq 1$ and suppose ADAP-FISTA generates sequence $\{\tilde x_{i}\}^{m}_{i=0} \subseteq \tilde B$. Then, the following statements hold:
    \begin{itemize}
        \item[(a)] $L_0\leq L_{i+1}\leq \max\{L_0, 4\bar L/(1-\chi)\}$ for any $i\in\{0,\ldots m\}$;
        \item[(b)] for any $x \in \mathcal N$, the relation
        \begin{align}\label{key convergence inequality}
        A_i[\psi(y_{m+1}) - \psi(x) ] + \frac{\tau_{m+1}}2 \|x-x_{m+1}\|^2 \le
        \frac12 \|x-x_0\|^2 - \frac{\chi}{2} \sum_{i=0}^{m} A_{i+1}L_{i+1} \| y_{i+1} - \tx_i \|^2
        \end{align}
        holds;
        \item[(c)] $\tilde x_{m+1} \in \tilde B$.
    \end{itemize}
As a consequence, the entire sequence $\{\tilde x_{i}\} \subseteq \tilde B$.
\end{lem}
\begin{proof}
(a) Let $i \in \{0, \ldots m\}.$ Clearly $y_{i+1} \in \tilde B$ since the definition of $\tilde B$ in \eqref{Closed Ball} implies that $\tilde B\supseteq \mathcal N$. Now, using the facts that $\tilde x_{i}\in \tilde B$ and $y_{i+1} \in \tilde B$, assumption (B1) implies that $\nabla \psi_{s}$ is Lipschitz continuous at these points. The proof of (a) is then identical to the one of Lemma A.3(b) in \cite{SujananiMonteiro}. 

(b) Clearly, since ADAP-FISTA generates sequence $\{\tilde x_i\}_{i=0}^{m}$, its loop in step 2 always terminates during its first $m$ iterations. Hence, ADAP-FISTA also generates sequences $\{y_i\}_{i=0}^{m+1}$ and $\{x_i\}_{i=0}^{m+1}$. The proof of relation \eqref{key convergence inequality} then follows from this observation, (a), and by using similar arguments to the ones developed in Lemmas~A.6-A.10 of \cite{SujananiMonteiro}.

(c) It follows from the fact that $\tau_{m+1} \geq 1$, relation \eqref{key convergence inequality} with $x=y_{m+1}$, and the definition of $D_{\psi_{n}}$ that $\|y_{m+1}-x_{m+1}\|\leq \|y_{m+1}-x_0\| \le D_{\psi_{n}}$. This relation, the fact that $y_{m+1} \in \mathcal N$, and the definition of $\tilde B$ in \eqref{Closed Ball} then imply that 
\begin{align*}
x_{m+1} \subseteq \mathcal N + \bar B_{D_{\psi_{n}}}=\tilde B.
\end{align*}
The result then immediately follows from the fact that $\tilde x_{m+1}$ is a convex combination of $x_{m+1}$ and $y_{m+1}$ and that $\tilde B$ is a convex set.




The last statement in Proposition~\ref{Key proposition} follows immediately from (c) and a simple induction argument.
\end{proof}

We now present the main convergence results of ADAP-FISTA, whose proofs can be found in \cite{SujananiMonteiro}. 
Proposition~\ref{prop:nest_complex1} below gives an
iteration complexity bound regardless if ADAP-FISTA terminates with success or failure and shows that if ADAP-FISTA successfully stops, then it obtains a stationary solution of \eqref{OptProb1} with respect to a relative error criterion. It also shows that ADAP-FISTA always stops successfully
whenever $\psi_s$ is $\mu$-strongly convex.

\begin{proposition}\label{prop:nest_complex1} 
The following statements about ADAP-FISTA hold:
\begin{itemize}
\item[(a)] if
$L_0 = {\cal O}(\bar L)$, it
always stops (with either success or failure) in at most 
\[\mathcal O_1\left(\sqrt{\frac{\bar L}{\mu}}\log^+_1 (\bar L) \right)\] 
iterations/resolvent evaluations;
\item[(b)] if it stops successfully, it terminates with a triple
$(y,v,L) \in \mathcal N \times \mathbb E \times \RR_{++}$ satisfying 
\begin{align}
&v \in \nabla \psi_s(y)+\partial \psi_n(y), \quad \|v\|\leq \sigma\|y-x_0\|, \quad L \leq \max\{L_0, 4\bar L/(1-\chi)\} \label{u inclusion 1};
\end{align}
\item[(c)] if $\psi_s$ is $\mu$-convex on $\mathcal N$, then ADAP-FISTA always terminates with success and
its output
$(y,v, L)$, in addition
to satisfying
\eqref{u inclusion 1}
also  satisfies
the inclusion $v \in \partial (\psi_s+\psi_n)(y)$.
\end{itemize}
\end{proposition}

\section{Proof of Proposition~\ref{New Main Result ADAP-AIPP}}\label{ADAP-AIPP Appendix Proof}
This section provides the proof of Proposition~\ref{New Main Result ADAP-AIPP} stated in Subsection~\ref{ADAP-AIPP Method}.

Let $\bar B_r:=\{U:\|U\|_{F}\leq r\}$.
The following lemma establishes that the function $\tilde g(\cdot)$ in \eqref{factorized problem} has Lipschitz continuous gradient over $\bar B_3$.

\begin{lemma}\label{Lipschitz gradient}
The function $\tilde g(U)$ defined in \eqref{factorized problem} is $L_{\tilde g}$-smooth on $\bar B_3$ where $L_{\tilde g}$ is as in \eqref{phi* alpha0 quantities}.
\end{lemma}
\begin{proof}
Let $U_{1}, U_2 \in \bar B_3$ be given.
Adding and subtracting $U_1U^{T}_2$ and using the triangle inequality, we have
\begin{equation}\label{MVT}
    \|U_{1}U_{1}^{T}-U_{2}U_{2}^{T}\|_{F}\leq \|U_1\|\,\|U_1-U_2\|_{F}+\|U_2\|\,\|U_1-U_2\|_{F}\leq 6\|U_1-U_2\|_{F}.
\end{equation}
This relation, the chain rule, the triangle inequality, the facts that $U_1, U_2 \in \bar B_{3}$ and $g$ is $L_g$-smooth
on $\bar B_3$,
and the definition of $\bar G$ in \eqref{phi* alpha0 quantities},  then imply that 
\begin{align*}
 & \|\nabla \tilde g(U_1) -\nabla \tilde g(U_2)\|_{F} =\|2\nabla g(U_1U^{T}_1)U_1-2\nabla g(U_2U^{T}_{2})U_2\|_{F} \\
& \leq \|2\nabla g(U_1U^{T}_1)U_1-2\nabla g(U_1U^{T}_1)U_2\|_{F} + 
\|2\nabla g(U_1U^{T}_1)U_2-2\nabla g(U_2U^{T}_2)U_2\|_{F} \\
& \leq 2\|\nabla g(U_1U^{T}_1)\|_{F}\,\|U_1-U_2\|_{F} + 
2\|U_2\|\,\|\nabla g(U_1U^{T}_1)-\nabla g(U_2U^{T}_2)\|_{F}  \\
&\overset{\eqref{phi* alpha0 quantities}}{\leq} 2\bar G\|U_1-U_2\|_{F} + 6\|\nabla g(U_1U^{T}_1)-\nabla g(U_2U^{T}_2)\|_{F}   \\
&\overset{\eqref{upper curvature g}}{\le} 2\bar G\|U_1-U_2\|_{F}+6L_g\|U_1U^{T}_1-U_2U^{T}_2\|_{F}\\
&\overset{\eqref{MVT}}{\le} 2\bar G\|U_1-U_2\|_{F}+36L_g\|U_1-U_2\|_{F}=(2\bar G+36L_g)\|U_1-U_2\|_{F}.
 \end{align*}
The conclusion of the lemma now follows from the above inequality and the definition of $L_{\tilde g}$ in \eqref{phi* alpha0 quantities}.
\end{proof}

The following lemma establishes key properties about each ADAP-FISTA call made in step 1 of ADAP-AIPP. It is a translation of the results in Proposition~\ref{prop:nest_complex1}.


\begin{lemma}\label{ADAP FISTA translated}
Let $(\psi_s,\psi_n)$ be as in \eqref{eq:psiS-psimu}. The following statements about each ADAP-FISTA call made in the $j$-th iteration of ADAP-AIPP hold:
\begin{itemize}
\item[(a)] if
$\underbar M_j = {\cal O}(1+\lambda_0L_{\tilde g})$, it
always stops (with either success or failure) in at most 
\[\mathcal O_1\left(\sqrt{2\left(1+\lambda_0L_{\tilde g}\right)}\log^+_1 (1+\lambda_0L_{\tilde g}) \right)\] 
iterations/resolvent evaluations where $\lambda_0$ is the initial prox stepsize and $L_{\tilde g}$ is as in \eqref{phi* alpha0 quantities};
\item[(b)] if ADAP-FISTA stops successfully, it terminates with a triple
$(W,V,L) \in \bar B_1 \times \mathbb S^{n} \times \RR_{++}$ satisfying 
\begin{equation}
V \in \lambda \left[\nabla \tilde g(W)+\partial \delta_{\bar B_1}(W)\right]+(W-W_{j-1})
\end{equation}
\begin{equation}
\|V\|_{F}\leq \sigma\|W-W_{j-1}\|_{F}, \quad L \leq \max\{\underbar M_j, \omega (1+\lambda_0L_{\tilde g})\}\label{inclusion}
\end{equation}
where $\omega=4/(1-\chi)$;
\item[(c)] if $\psi_s$ is $1/2$-convex on $\bar B_1$, then ADAP-FISTA always terminates with success and
its output
$(W,V,L)$
always satisfies relation \eqref{subdiff ineq check}. 
\end{itemize}
\end{lemma}
\begin{proof}
(a)  The result follows directly from Proposition~\ref{prop:nest_complex1} in Appendix~\ref{ADAP-FISTA} and hence the proof relies on verifying its assumptions. First it is easy to see that $\dom \psi_{n}+\bar B_2=\bar B_3$, where $\psi_{n}$ is as in \eqref{eq:psiS-psimu}. It also follows immediately from the fact that $\lambda\leq \lambda_0$ and from Lemma~\ref{Lipschitz gradient} that  $\psi_{s}$ is $(1+\lambda_0 L_{\tilde g})$-smooth on $\bar B_3$ in view of its definition in \eqref{eq:psiS-psimu}. These two observations imply that $\bar L=1+\lambda_0L_{\tilde g}$ satisfies \eqref{ineq:uppercurvature1}. Hence, it follows from this conclusion, the fact that each ADAP-FISTA call is made with $(\mu,L_0)=(1/2,\underbar M_j)$, and from Proposition~\ref{prop:nest_complex1}(a) that statement (a) holds.


(b) In view of the definition of $\psi_s$ in \eqref{eq:psiS-psimu}, it is easy that see that $\nabla \psi_s(W)=\lambda \tilde g(W)+(W-W_{j-1})$. Statement (b) then immediately follows from this observation, the fact that each ADAP-FISTA call is made with inputs $x_0=W_{j-1}$, $L_0=\underbar M_j$, and $(\psi_s,\psi_n)$ as in \eqref{eq:psiS-psimu}, and from Proposition~\ref{prop:nest_complex1}(b) with $\bar L=1+\lambda_0L_{\tilde g}$.


(c) It follows immediately from Proposition~\ref{prop:nest_complex1}(c) and the fact that each ADAP-FISTA call is made with inputs $(\psi_s,\psi_n)$ as in \eqref{eq:psiS-psimu} and $\mu=1/2$ that the first conclusion of statement (c) holds and that output $(W,V,L)$ satisfies inclusion
\begin{equation}\label{Inclusion ADAP}
V \in \partial \left(\lambda \left(\tilde g+\delta_{\bar B_1}\right)(\cdot)+\frac{1}{2}\|\cdot-W_{j-1}\|^2_{F}\right)\left(W\right).
\end{equation}
Inclusion \eqref{Inclusion ADAP} and the definition of subdifferential in \eqref{def:epsSubdiff} then immediately imply that output $(W,V,L)$ satisfies relation \eqref{subdiff ineq check}.
\end{proof}

The lemma below shows that,
in every iteration of ADAP-AIPP,
the loop within steps 1 and 2 always stops and shows key properties of its output.
Its proof (included here for completeness) closely follows the one of Proposition 3.1
of \cite{SujananiMonteiro}.

\begin{lemma}\label{ACG facts}
The following statements about ADAP-AIPP hold for every $j\ge 1$:
\begin{itemize}
\item[(a)] the function $\psi_s$ in \eqref{eq:psiS-psimu} has
$(1+\lambda_0L_{\tilde g})$-Lipschitz continuous gradient on $\bar B_3$;
\item[(b)]
the loop within steps 1 and 2 of its
$j$-th iteration always ends and the output  $(W_j,V_j,R_j,\lam_j, \bar M_j)$ obtained
at the end of step 2 satisfies
\begin{align}
&R_j\in\nabla \tilde g(W_j)+\partial \delta_{\bar B_1}(W_j); \label{w inclusion}\\
&\left(\frac{1-\sigma}{\sigma}\right)\|V_{j}\|_{F}\leq \|\lambda_j R_j\|_{F}\leq (1+\sigma)\|W_j-W_{j-1}\|_{F}\label{w inequalities};\\
&\lambda_{j} \tilde g(W_{j-1}) - \left [\lambda \tilde g(W_{j}) + \frac{1}{2} \|W_{j}-W_{j-1}\|_{F}^2 \right] \ge V_j\bullet (W_{j-1}-W_{j});\label{subdiff zk} \\
&\bar M_{j} \le \max\{\underbar M_{j}, \omega \mathcal (1+\lambda_0L_{\tilde g})\}\label{upper L bound};\\
&\lambda_0\geq \lambda_j \geq \underline \lambda,\label{lower lambda bound}
\end{align}
where $\omega=4/(1-\chi)$, $\lambda_0$ is the initial prox stepsize, and $L_{\tilde g}$ and $\underline \lambda$ are as in \eqref{phi* alpha0 quantities} and \eqref{phi* alpha0 quantities-2}, respectively;
moreover, every prox stepsize~$\lam$ generated within the aforementioned loop is in $[\underline \lam, \lambda_0]$.
\end{itemize}
\end{lemma}

\begin{proof}
(a) It follows that $\psi_{s}$ has $(1+\lambda L_{\tilde g})$-Lipschitz continuous gradient on $\bar B_3$ in view of its definition in \eqref{eq:psiS-psimu} and Lemma~\ref{Lipschitz gradient}. The result then follows immediately from the fact that $\lambda \leq \lambda_{0}$. 

(b) We first claim that if the loop consisting of steps 1 and 2 of the $j$-th iteration of ADAP-AIPP
stops, then relations \eqref{w inclusion}, \eqref{w inequalities}, \eqref{subdiff zk}, and \eqref{upper L bound} hold.
Indeed, assume that the loop consisting of steps 1 and 2 of the $j$-th iteration of ADAP-AIPP
stops. 
It then follows from the logic within step 1 and 2 of ADAP-AIPP that the last ADAP-FISTA call within the loop stops successfully and outputs triple $(W_j,V_j,\bar M_j)$ satisfying \eqref{subdiff ineq check}, which immediately implies that \eqref{subdiff zk} holds. Since (a) implies that
$\bar L=1+\lambda_0L_{\tilde g}$ satisfies
relation \eqref{ineq:uppercurvature1}, it follows Proposition~\ref{prop:nest_complex1}(b) with $(\psi_s,\psi_n)$ as in \eqref{eq:psiS-psimu}, $x_0=W_{j-1}$, and $L_0=\underbar M_j$
that the triple $(W_j,V_j,\bar M_j)=(y,v,L)$ satisfies inequality \eqref{upper L bound} and the following two relations
\begin{align}
&V_{j} \in \lambda_j[\nabla \tilde g(W_j)+\partial \delta_{\bar B_1}(W_j)]+W_{j}-W_{j-1} \label{R tilde inclusion}\\
&\|V_{j}\|_{F}\leq \sigma\|W_{j}-W_{j-1}\|_{F} \label{Nesterov Wj}.
\end{align}
Now, using the definition of
$R_j$ in \eqref{Rj def},
it is easy to see that
the inclusion \eqref{R tilde inclusion} is equivalent
to \eqref{w inclusion} and that the inequality in
\eqref{Nesterov Wj} together with the triangle inequality for norms imply the two inequalities in \eqref{w inequalities}.


We now claim that
if 
step 1 is performed
with a prox stepsize
$\lambda \leq 1/(2L_{\tilde g})$
in the $j$-th iteration,
then for every $l >j$, we have that
$\lam_{l-1}=\lam$
and the $l$-th iteration performs step 1 only once.
To show the claim,
assume that $\lambda\leq 1/(2L_{\tilde g})$. Using this assumption and the fact that Lemma~\ref{Lipschitz gradient} implies that $\tilde g$ is $L_{\tilde g}$ weakly convex on $\bar B_3$, it is easy to see that the function $\psi_s$ in \eqref{eq:psiS-psimu} is strongly convex on $\bar B_1\subseteq \bar B_3$ with modulus $1-\lam L_{\tilde g} \geq 1/2$. Since each ADAP-FISTA call is performed in step~1 of ADAP-AIPP with $\mu=1/2$, it follows immediately from Proposition~\ref{prop:nest_complex1}(c) with $(\psi_s,\psi_n)$ as in \eqref{eq:psiS-psimu} that ADAP-FISTA terminates successfully and outputs
a pair $(W,V)$ satisfying $V \in \partial (\psi_s+\psi_n)(W)$. This inclusion, the definitions of
$(\psi_s,\psi_n)$, and 
the definition of subdifferential in \eqref{def:epsSubdiff},
then imply that
\eqref{subdiff ineq check} holds.
Hence, in view of the termination criteria of step~2 of ADAP-AIPP, it follows that $\lambda_j=\lambda$. It is then easy to see, by the way $\lambda$ is updated in step~2 of ADAP-AIPP, that $\lambda$ is not halved in the $(j+1)$-th iteration or any subsequent iteration, hence proving the claim.

It is now straightforward to
see that the above two claims,
the fact that the initial value of the prox stepsize is equal to~$\lambda_0$, and
the way $\lam_j$ is updated in
ADAP-AIPP, imply
that the lemma holds.
\end{proof}

\begin{lemma}\label{induction lemma}
For any $j \geq 1$,
the quantity $\underbar M_j$ satisfies
\begin{equation}\label{important recursion}
\underbar M_{j} \leq \omega (1+\lambda_0L_{\tilde g})
\end{equation}
where $\omega=4/(1-\chi)$, $\lambda_0$ is the initial prox stepsize, and $L_{\tilde g}$ is as in \eqref{phi* alpha0 quantities}.
\end{lemma}
\begin{proof}
The result follows from a simple induction argument. The inequality with $j=1$ is immediate due to the facts that $\underbar M_{1}=1$ and $\omega>1$. Now suppose inequality \eqref{important recursion} holds for $j-1$. It then follows from relation \eqref{upper L bound} and the fact that $\underbar M_j\leq \bar M_{j-1}$ that
\[\underbar M_{j} \leq \bar M_{j-1} \overset{\eqref{upper L bound}}{\le} \max\{\underbar M_{j-1}, \omega (1+\lambda_0L_{\tilde g})\}=\omega (1+\lambda_0L_{\tilde g}),\]
where the equality is due to the assumption that \eqref{important recursion} holds for $j-1$. Hence, Lemma~\ref{induction lemma} holds.
\end{proof}

\begin{rem}
It follows from Lemma~\ref{induction lemma}  that $\underbar M_j = {\cal O}(1+\lambda_0L_{\tilde g})$ and hence Lemma~\ref{ADAP FISTA translated}(a) implies
that each ADAP-FISTA call made in step 1 of ADAP-AIPP performs at most
\[\mathcal O_1\left(\sqrt{2\left(1+\lambda_0L_{\tilde g}\right)}\log^+_1 (1+\lambda_0L_{\tilde g}) \right)\]
iterations/resolvent evaluations where $\lambda_0$ is the initial prox stepsize and $L_{\tilde g}$ is as in \eqref{phi* alpha0 quantities}.
\end{rem}
The following lemma shows that ADAP-AIPP is a descent method.
\begin{lemma}
If $j$ is an iteration index for ADAP-AIPP, then
\begin{align}
\frac {\underline \lambda}{C_{\sigma}}\|R_j\|_{F}^2
&\le \tilde g(W_{j-1})-\tilde g(W_j)\label{auxineq:declemma3}
\end{align}
where $C_\sigma$ and $\underline \lambda$ are as in \eqref{phi* alpha0 quantities-2}.
\end{lemma}
\begin{proof}
It follows immediately from the first inequality in \eqref{w inequalities}, relation \eqref{subdiff zk}, and the definitions of $R_{j}$ and $C_{\sigma}$ in \eqref{Rj def} and \eqref{phi* alpha0 quantities-2}, respectively that:
\begin{align}\label{string of inequalities}
    \lambda_j\tilde g(W_{j-1}) -  \lambda_j\tilde g(W_j)&\overset{\eqref{subdiff zk}}{\ge} \frac{1}{2} \|W_{j}-W_{j-1}\|_{F}^2+ V_{j} \bullet (W_{j-1}-W_j) \nonumber\\
    &=\frac{1}{2}\|W_{j-1}-W_{j}+V_{j}\|_{F}^2-\frac{1}{2}\|V_j\|_{F}^2 \nonumber\\
    &\overset{\eqref{Rj def}}{=}\frac{1}{2}\|\lambda_j R_j\|_{F}^2-\frac{1}{2}\|V_j\|_{F}^2 \nonumber\\
    &\overset{\eqref{w inequalities}}{\geq}\frac{1}{2}\|\lambda_jR_j\|_{F}^2-\frac{\sigma^2}{2(1-\sigma)^2}\|\lambda_jR_j\|_{F}^2 \nonumber\\
    &=\frac{1-2\sigma}{2(1-\sigma)^2}\|\lambda_jR_j\|_{F}^2 \overset{\eqref{phi* alpha0 quantities-2}}{=}\frac{\|\lambda_jR_j\|_{F}^2}{C_\sigma}.
\end{align}
Dividing inequality \eqref{string of inequalities} by $\lambda_j$ and using relation \eqref{lower lambda bound} then imply
\begin{align*}
\tilde g(W_{j-1})-\tilde g(W_j)\overset{\eqref{string of inequalities}}{\geq}  \frac{\lambda_{j}}{C_\sigma}\|R_j\|_{F}^2 \overset{\eqref{lower lambda bound}}{\geq} \frac{\underline \lambda}{C_\sigma}\|R_j\|_{F}^2
\end{align*}
from which the result of the lemma immediately follows.
\end{proof}


We are now ready to give the proof of Proposition~\ref{New Main Result ADAP-AIPP}.
\begin{proof}[Proof of Proposition~\ref{New Main Result ADAP-AIPP}]

(a) The first statement of (a) follows immediately from the fact that relation \eqref{w inclusion} and the termination criterion of ADAP-AIPP in its step 3 imply that the pair $(\overline W,\overline R)$ satisfies \eqref{stationary solution}. Assume now by contradiction that \eqref{iteration bound} does not hold.
This implies that there exists an iteration index $l$ such that
\begin{equation}\label{iteration count}
l > 1+\frac{C_{\sigma}}{\underline \lambda \bar \rho^2} \left[\tilde g(\underline W)-\tilde g(W_l)\right].
\end{equation}
As ADAP-AIPP generates $l$ as an iteration index, it does not terminate at the $(l-1)$-th iteration. In view of step 3 of ADAP-AIPP, this implies that
$\|R_j\|_F > \bar \rho$ for every $j=1,\ldots,l-1$.
Using this conclusion and the fact that $W_0=\underline W$, and summing inequality \eqref{auxineq:declemma3} from $1$ to $l$, we conclude that
\begin{align*}
(l-1) \bar \rho^2 < \sum_{j=1}^{l-1} \|R_j\|_{F}^2 \leq \sum_{j=1}^{l} \|R_j\|_{F}^2 \overset{\eqref{auxineq:declemma3}}{\leq} \frac{C_{\sigma}}{\underline \lambda}\sum_{j=1}^{l}\
	\tilde g(W_{j-1})-\tilde g(W_j)
=\frac{C_{\sigma}}{\underline \lambda} \left[\tilde g(\underline W)-\tilde g(W_{l})\right]
\end{align*}
which can be easily seen to contradict \eqref{iteration count}. 

(b) The result follows immediately from (a) and the fact that
the number of times $\lam$ is divided by $2$
in step 2 of ADAP-AIPP is at most $\lceil \log_0^+(\lambda_0/{\underline \lam})/\log 2 \rceil$.
%
\end{proof}

\section{Relaxed Frank-Wolfe Method}\label{FW Appendix}
Let $\mathbb E$ denote a finite-dimensional inner product real vector space with inner product and induced norm denoted by $\inner{\cdot}{\cdot}$ and $\|\cdot\|$, respectively. Let  $\Omega \subseteq \mathbb E$ be a nonempty compact convex set with diameter $D_{\Omega}$. 
Consider the problem
\begin{equation}\label{Problem of Interest}
 (P)   \quad g_{*}:=\min_{U}\{g(U): U\in \Omega\}
\end{equation}
where $g:\mathbb E \rightarrow \mathbb R$ is a convex function that satisfies the following assumption:
\begin{itemize}
\item[\textbf{(A1)}]
there exists $L_g > 0$ such that
\begin{equation}\label{upper curvature}
g(U') -\ell_g(U';U) \leq \frac{L_g}2 \|U'-U \|^2 \quad \forall U,U' \in \Omega.
\end{equation}
\end{itemize}

The formal description of the Relaxed FW (RFW) method and its main complexity result for finding a near-optimal solution of \eqref{Problem of Interest} are presented below. The proof of the main result is given in the next subsection.

\noindent\begin{minipage}[t]{1\columnwidth}%
\rule[0.5ex]{1\columnwidth}{1pt}

\noindent \textbf{RFW Method}

\noindent \rule[0.5ex]{1\columnwidth}{1pt}%
\end{minipage}

\noindent \textbf{Input}: tolerance $\bar \epsilon>0$ and initial point $\tilde Z_0 \in \Omega$.

\noindent \textbf{Output}: a point $\bar Z$.
\begin{itemize}
\item[{\bf 0.}] set  $k=1$;

\item[{\bf 1.}]  find a point $Z_k \in \Omega$ such that 
\begin{equation}\label{descent condition}
   g(Z_{k})\leq g(\tilde Z_{k-1});
\end{equation}
\item[{\bf 2.}]
compute
\begin{equation}\label{FW algo definitions}
Z^F_k \in \argmin \{ \ell_g(U;Z_k) : U \in \Omega\}, \quad D_k:=Z_k-Z^{F}_k, \quad \epsilon_k:=\inner{\nabla g(Z_k)}{D_k};
\end{equation}
\item[\bf 3.] 
if $\epsilon_k \leq \bar \epsilon$,
then {\bf stop} and output the point $\bar Z= Z_k$; else compute \begin{equation}\label{FW stepsize}
    \alpha_k=\argmin_{\alpha}\left\{g(Z_k-\alpha D_k): \alpha \in [0,1] \right\}\}
\end{equation}
and set \begin{equation}\label{Tilde X Update}
  \tilde Z_{k}=Z_k-\alpha_kD_k;  
\end{equation}
\item[{\bf 4.}] 
set $k\gets k+1$ and go to step~1. 
\end{itemize}
\noindent \rule[0.5ex]{1\columnwidth}{1pt}

\begin{thm}\label{FW Theorem}
    For a given tolerance $\bar \epsilon>0$, the RFW method finds a point $\bar Z \in \Omega$ such that \begin{equation}\label{Approximate Solution Type}
    0 \in \nabla g(\bar Z)+\partial_{\bar \epsilon}\delta_{\Omega}(\bar Z)
    \end{equation}
in at most 
    \begin{equation}\label{Hybrid FW Method Exact Complexity}
    \left\lceil1+\frac{4\max\left\{g(\tilde Z_0)-g_{*},\sqrt{\left(g(\tilde Z_0)-g_{*}\right)L_gD^2_{\Omega}}, L_{g}D^2_{\Omega}\right\}}{\bar \epsilon} \right\rceil
    \end{equation}
    iterations where $L_g$ is as in \eqref{upper curvature} and $D_{\Omega}$ is the diameter of $\Omega$.
\end{thm}


\subsection{Proof of Theorem~\ref{FW Theorem}}
This subsection is dedicated to proving Theorem~\ref{FW Theorem}. 

The following lemma establishes important properties of the iterates $Z_k$ and $\epsilon_k$.
\begin{lem}\label{Epsilon K Facts Proposition}
For every $k \geq 1$, the following relations hold:
\begin{equation}\label{epsilon nonnegative}
\epsilon_{k} \geq g(Z_k)-g_{*},
 \end{equation}
\begin{equation}\label{epsilon inclusion-appendix}
    0 \in \nabla g(Z_k)+\partial_{\epsilon_k}\delta_{\Omega}(Z_k),
\end{equation}
where $\epsilon_{k}$ is defined in \eqref{FW algo definitions} and $g_{*}$ is the optimal value of \eqref{Problem of Interest}.
\end{lem}
\begin{proof}
Suppose $Z' \in \Omega$. It follows from the fact $g$ is convex, the definitions of $D_k$ and $\epsilon_{k}$ in \eqref{FW algo definitions}, and the way $Z^{F}_{k}$ is computed in \eqref{FW algo definitions} that
    \begin{align}\label{key relation New}
        g(Z_k)-\epsilon_{k}&\overset{\eqref{FW algo definitions}}{=}g(Z_k)+\inner{\nabla g(Z_k)}{Z^F_k-Z_k} \nonumber\\
        &\overset{\eqref{FW algo definitions}}{\leq} g(Z_k)+\inner{\nabla g(Z_k)}{Z'-Z_k} \leq g(Z').
    \end{align}
Since \eqref{key relation New} holds for any $Z'$ in $\Omega$, it must hold for the minimizer $Z^{*}$ of \eqref{Problem of Interest}, and hence
$g(Z_k)-\epsilon_{k} \leq g_{*}$,
which immediately shows relation \eqref{epsilon nonnegative}.

It follows from the fact that $Z^{F}_k \in \argmin \{ \ell_g(U;Z_k) : U \in \Omega\}$ that \[0 \in \nabla g(Z_k)+\partial \delta_{\Omega}(Z^{F}_k).\] It then follows from the above relation, the definition of $\epsilon$-subdifferential in \eqref{def:epsSubdiff}, and the definition of $\epsilon_{k}$ in \eqref{FW algo definitions} that inclusion \eqref{epsilon inclusion-appendix} holds.
\end{proof}

The following lemma establishes that the RFW method is a descent method.

\begin{lem}\label{Important FW Convergence Lemma}
Define
\begin{equation}\label{Definition of Hat Alpha}
\hat \alpha_k:=\min\left\{ 1, \frac{\epsilon_k}
{L_gD^2_{\Omega}}\right\} \quad \forall k\geq 1.
\end{equation}
Then the following statements hold for every $k \geq 1$:
\begin{equation}\label{key convergence relation}
\epsilon_{k}\leq \frac{2}{\hat \alpha_k}\left(g(Z_k)-g(\tilde Z_{k}) \right),
\end{equation}
\begin{equation}\label{descent tildes}
    g(Z_{k+1}) \leq g(\tilde Z_{k})\leq g(Z_{k}).
\end{equation}

\end{lem}
\begin{proof}
It follows from the definitions of $\epsilon_k$, $\alpha_k$, $\tilde Z_k$, and $\hat \alpha_k$ in \eqref{FW algo definitions}, \eqref{FW stepsize}, \eqref{Tilde X Update}, \eqref{Definition of Hat Alpha}, respectively, the fact that $\|D_k\|^2\leq D^2_{\Omega}$, and from applying inequality \eqref{upper curvature} with $U'=Z_k-\hat \alpha_kD_k$ and $U=Z_{k}$ that
\begin{align*}
  g(\tilde Z_k)&\overset{\eqref{FW stepsize}, \eqref{Tilde X Update}}{\leq} g(Z_k-\hat \alpha_kD_k)\overset{\eqref{upper curvature}}{\leq} g(Z_k)-\hat \alpha_k \inner{\nabla g(Z_k)}{D_k}+\frac{\hat \alpha_k^2L_{g}}{2}D^2_{\Omega}\\
 &\overset{\eqref{FW algo definitions}}{=} g(Z_k)-\hat \alpha_k \epsilon_k+\frac{\hat \alpha_k^2L_{g}}{2}D^2_{\Omega} \overset{\eqref{Definition of Hat Alpha}}{\leq} g(Z_k)-\hat \alpha_{k}\epsilon_{k}+\frac{\hat \alpha_k\epsilon_k}{2}=g(Z_k)-\frac{\hat \alpha_k\epsilon_k}{2}
\end{align*}
which immediately implies relation \eqref{key convergence relation}. 

The first inequality in \eqref{descent tildes} follows immediately from \eqref{descent condition}. The second inequality in \eqref{descent tildes} follows immediately from relations \eqref{key convergence relation} and \eqref{epsilon nonnegative}.
\end{proof}

The next proposition establishes the convergence rate of the RFW method.

\begin{prop}\label{Hybrid FW Convergence Rate}
For every $k \geq 2$, the following relations hold:
\begin{align}\label{Function Value Convergence}
        g(Z_k)-g_{*} &\leq \frac{2}{k-1} \max\left\{g(\tilde Z_0)-g_{*},L_gD^2_{\Omega}\right\},\\
        \min_{k \leq j < 2k} \epsilon_j &\leq  \frac{4}{k-1} \max\left\{g(\tilde Z_0)-g_{*}, \sqrt{\left(g(\tilde Z_0)-g_{*}\right)L_gD^2_{\Omega}},L_gD^2_{\Omega}\right\} \label{epsilon Convergence bound}.
    \end{align}
\end{prop}
\begin{proof}
Define $\tilde \gamma_0:=g(\tilde Z_0)-g_{*}$ and let $\gamma_{j}:=g(Z_{j})-g_{*}$ for any iteration index $j$. It then follows from relations \eqref{epsilon nonnegative} and \eqref{key convergence relation} and relation \eqref{descent tildes} with $k=j$ that the following two relations hold
\begin{align}
&\frac{\hat \alpha_{j}}{2}\gamma_{j}\overset{\eqref{epsilon nonnegative}}{\leq} \frac{\hat \alpha_{j}}{2}\epsilon_{j}\overset{\eqref{key convergence relation}}{\leq} g(Z_j)-g(\tilde Z_{j})\overset{\eqref{descent tildes}}{\leq} g(Z_j)-g(Z_{j+1})=\gamma_{j}-\gamma_{j+1}\label{Recurrence Relation-1}\\
&\gamma_{j+1} \overset{\eqref{descent tildes}}{\leq}\gamma_{j} \overset{\eqref{epsilon nonnegative}}{\leq} \epsilon_{j}.\label{Recurrence Relation-2}
\end{align}
Hence, using relations \eqref{Recurrence Relation-1} and \eqref{Recurrence Relation-2}, relation \eqref{descent condition} with $k=1$, and the expression for $\hat \alpha_{j}$ in \eqref{Definition of Hat Alpha}, it follows that
\begin{align}\label{Telescoping FW}
    \frac{1}{\gamma_{j+1}}-
    \frac{1}{\gamma_{j}}
    &= \frac{\gamma_{j} -\gamma_{j+1} }{\gamma_{j+1}\gamma_{j}} \overset{\eqref{Recurrence Relation-1}}{\ge} \frac{\hat \alpha_j \gamma_j}{2\gamma_{j+1}\gamma_{j}} \overset{\eqref{Definition of Hat Alpha}}{=}\frac{1}{2\gamma_{j+1}} \min \left\{1, \frac{\epsilon_j}{L_gD^2_{\Omega}} \right\} \\
   &\overset{\eqref{Recurrence Relation-2}}{\geq} \frac{1}{2}\min \left\{\frac{1}{{\gamma_{1}}}, \frac{1}  {L_gD^2_{\Omega}} \right\} \overset{\eqref{descent condition}}{\geq} \frac{1}{2}\min \left\{\frac{1}{\tilde \gamma_0}, \frac{1}  {L_gD^2_{\Omega}} \right\}.
\end{align}
It follows from summing the inequality in \eqref{Telescoping FW} from $j=1$ to $k-1$ that
\begin{equation}\label{intermediary function value}
\frac{1}{\gamma_{k}}
\ge \frac{1}{\gamma_{k}}-\frac{1}{\gamma_1}
\ge \frac{k-1}2 \min \left\{ \frac{1}{\tilde \gamma_{0}}, \frac{1} 
{L_gD^2_{\Omega}}
\right\},
\end{equation}
which, together, with the definition of $\gamma_{k}$ implies relation \eqref{Function Value Convergence}.

It follows from summing the relation in \eqref{key convergence relation} from $j=k$ to $j=2k+1$ and relations \eqref{descent tildes} and \eqref{intermediary function value} that 
\begin{align}
\frac{2}{k-1} \max\{\tilde \gamma_0,L_gD^2_{\Omega}\}&\overset{\eqref{intermediary function value}}{\geq} \gamma_k \geq
g(Z_k) - g(Z_{2k+1}) \nonumber\\
&=\sum_{j=k}^{2k} g(Z_j) - g(Z_{j+1})
\overset{\eqref{descent tildes}}{\ge} \sum_{j=k}^{2k} g(Z_j) - g(\tilde{Z}_{j})
\overset{\eqref{key convergence relation}}{\ge} \sum_{j=k}^{2k} \frac{\hat \alpha_j}2 \epsilon_j \label{Final FW inequality}.
\end{align}
It now follows from relation \eqref{Final FW inequality} and the definition of $\hat \alpha_{j}$ in \eqref{Definition of Hat Alpha} that
\begin{align}
\frac{4}{(k-1)^2}\max\left\{\tilde \gamma_{0},L_gD^2_{\Omega}\right\} \overset{\eqref{Final FW inequality}}{\geq} \min_{k \leq j \leq 2k} \hat \alpha_{j}\epsilon_{j} \overset{\eqref{Definition of Hat Alpha}}{\geq} \min_{k \leq j \leq 2k} \left\{1, \frac{\epsilon_j}
{L_gD^2_{\Omega}}\right\} \min_{k \leq j \leq 2k} \epsilon_{j},
\end{align}
which implies relation \eqref{epsilon Convergence bound} in view of the definition of $\tilde \gamma_0$.
\end{proof}
We are now ready to prove Theorem~\ref{FW Theorem}.
\begin{proof}[Proof of Theorem~\ref{FW Theorem}]
The stopping criterion in step 3 of the RFW method and relation \eqref{epsilon inclusion-appendix} immediately imply that output $\bar Z$ satisfies relation \eqref{Approximate Solution Type}.

In view of the stopping criterion in step 3 of the RFW method, the iteration complexity result in \eqref{Hybrid FW Method Exact Complexity} follows immediately from relation \eqref{epsilon Convergence bound}.
\end{proof}

\section{AL method for linearly-constrained convex optimization}
\label{s:ALgeneral}

This section is dedicated to analyzing the convergence of the augmented Lagrangian framework for solving linearly-constrained convex optimization problems.

Let $\mathbb E$ denote an Euclidean space,
$\mathcal{A}:\mathbb E \to \RR^m$ be a linear operator,  
$b\in\RR^m$, $f:\mathbb E \to \RR$ be a differentiable convex function, and $h:\mathbb E \to (-\infty,\infty]$ be a closed proper convex function.
Consider the linearly-constrained convex optimization problem
	\begin{equation} \label{eq:general_probl}
    \min \{ \phi(X) := f(X) + h(X) : \mathcal{A} X  =b\},
	\end{equation}
where the domain of $h$ has finite diameter $D_h$.
The following assumption is also made.

\begin{assumption}\label{ass:attained-general}
There exists $(X_*,p_*)$ such that
\begin{align}
  0\in \nabla f(X_*) + \partial h(X_{*})+\mathcal{A}^*p_*, \quad  \mathcal{A}X_*-b = 0
\end{align}
\end{assumption}
Given a previous dual iterate $p_{t-1}$, the AL framework finds the next primal iterate $X_{t}$ by
\begin{equation}\label{minimization problem}
    X_{t} \approx \argmin_{X} \mathcal L_{\beta}(X;p_{t-1})
\end{equation}
where 
\begin{equation}\label{augmented Lagrangian}
{\cal L}_\beta(X;p)
  := f(X) + h(X)+\left\langle p,\mathcal{A}X-b\right\rangle+\frac{\beta}{2}\|\mathcal{A} X-b\|^2 
\end{equation}
is the augmented Lagrangian function and $\beta>0$ is a fixed penalty parameter.
We assume the existence of a blackbox that inexactly solves such minimization problems as in \eqref{minimization problem}.

\begin{blackbox}{AL}\label{ass:blackbox-general}
Given a pair $(\hat\epsilon_{\mathrm{c}},\hat\epsilon_{\mathrm{d}}) \in \R^2_+$ and convex functions $g: \mathbb{E} \to \RR$ and $h: \mathbb{E} \to \RR$,
the blackbox returns a pair $(\hat X, \hat R)$ satisfying
\begin{equation}\label{residual}
    \hat X \in \dom h,\quad
    \hat R \in
    \nabla g(\hat X) + \partial_{\hat\epsilon_{\mathrm{c}}} h(\hat X),
    \qquad \|\hat R\| \leq \hat\epsilon_{\mathrm{d}}.
\end{equation}
\end{blackbox}

The AL framework is now presented formally below.

\noindent\begin{minipage}[t]{1\columnwidth}%
\rule[0.5ex]{1\columnwidth}{1pt}

\noindent \textbf{AL Framework}

\noindent \rule[0.5ex]{1\columnwidth}{1pt}%
\end{minipage}

\noindent \textbf{Input}: $p_0 \in \RR^{m}$,
tolerances $\epsilon_{\mathrm{p}}$ > 0, $\epsilon_{\mathrm{d}} \geq 0, \epsilon_{\mathrm{c}} \geq 0$,
and penalty parameter $\beta>0$.

\noindent \textbf{Output}: 
triple $(\bar X, \bar p, \bar R)$.

\begin{itemize}
\item[{\bf 0.}]
Set $t=1$ and
\begin{equation}\label{epsilon hat}
\hat \epsilon_{\mathrm{c}} = 
\min\{\epsilon_{\mathrm{c}},\;
{\beta \epsilon_{\mathrm{p}}^2}/{6}\}, \quad \hat \epsilon_{\mathrm{d}} =  
\min\{\epsilon_{\mathrm{d}},\;
{\beta \epsilon_{\mathrm{p}}^2}/(6 D_h)\};
\end{equation} 
\item[{\bf 1.}]
Call the Blackbox~AL with
tolerance pair $(\hat \epsilon_{\mathrm{c}}, \hat\epsilon_{\mathrm{d}})$ and functions $h=h$ and $g(\cdot)=\mathcal L_{\beta}(\cdot;p_{t-1})$
and 
let
$(X_t,R_t)$ be its output;

\item[{\bf 2.}]
Set
\begin{equation}\label{dual update}
 p_t=p_{t-1}+\beta (\mathcal AX_t-b);
\end{equation}
 
\item[{\bf 3.}]
If $\|\mathcal{A}X_t - b\| \leq \epsilon_{\mathrm{p}}$,
then set $T=t$ and \textbf{return} $(X_{T}, p_{T}, R_{T})$;
 
\item[{\bf 4.}]
 Set $t \leftarrow t+1$ and \textbf{go to} step \textbf{1.}
\end{itemize}
\noindent \rule[0.5ex]{1\columnwidth}{1pt}


The following result states the main iteration complexity of the AL framework and establishes the boundedness of its sequence of Lagrange multipliers. The proof of the result is given in the next subsection.

\begin{thm}\label{thm:OuterComplexity}
Under Assumption~\ref{ass:attained-general},
  the following statements about the AL framework hold:
  \begin{itemize}
\item[(a)] the AL framework terminates with
an iterate
$(X_{T}, p_{T}, R_{T}) \in \dom h \times \R^m \times \mathbb E$ such that
    \begin{align}\label{eq:stationary-general}
     R_{T}\in \nabla f(X_{T}) + \partial_{\epsilon_{\mathrm{c}}} h(X_{T})+\mathcal{A}^*p_{T} ,\quad
      \|R_{T}\|\leq  \epsilon_{\mathrm{d}},\quad  \|\mathcal{A}X_{T}-b\|\leq  \epsilon_{\mathrm{p}}
    \end{align}
and
  \begin{equation}\label{Complexity Result}
        T \le 
 \left\lceil\frac{3\|p_{*}-p_0\|^2}{\beta^2\epsilon_{\mathrm{p}}^2}\right\rceil;
    \end{equation}
\item[(b)] 
there hold
\begin{equation}\label{bound on Lagrange multipliers}
   \max_{t\in\{0,\ldots T\}}\|p_{t}\|\leq \|p_{*}\|+\sqrt{3\|p_{*}-p_0\|^2+2\beta(D_h\hat \epsilon_{\mathrm{d}}+\hat \epsilon_{\mathrm{c}})},
\end{equation}
\begin{equation}\label{bound on feasibility}
\beta^2\sum_{l=1}^{T}\|\mathcal{A}X_l-b\|^2\leq 3\|p_{*}-p_0\|^2+2\beta(D_h\hat \epsilon_{\mathrm{d}}+\hat \epsilon_{\mathrm{c}}),
\end{equation}
\end{itemize}
where $p_{*}$ is an optimal Lagrange multiplier and $\hat \epsilon_{\mathrm{c}}$ and $\hat \epsilon_{\mathrm{d}}$ are as in \eqref{epsilon hat}.
\end{thm}

\subsection{Proof of Theorem~\ref{thm:OuterComplexity}}
This subsection is dedicated to proving Theorem~\ref{thm:OuterComplexity}.
The proof relies on the following two preliminary lemmas.

\begin{lem}
For any $t \geq 1$, the following relation holds
\begin{equation}\label{inner product bound}
  \inner{\mathcal{A}X_t-b}{p_{*}-p_t}\geq \inner{R_t}{X_{*}-X_t}-\hat\epsilon_{\mathrm{c}},
\end{equation}
where $(X_{*},p_{*})$ is an optimal primal-dual pair of \eqref{eq:general_probl}. 
\end{lem}
\begin{proof}
Since $(X_{*},p_{*})$ is an optimal primal-dual pair, it follows that 
\begin{equation}\label{optimality inclusion}
  0 \in \partial \phi(X_{*})+\mathcal{A}^{*}p_{*}, \quad \mathcal{A}X_{*}=b,
\end{equation}
where $\phi$ is as in \eqref{eq:general_probl}.
Relation \eqref{residual} implies that 
\begin{equation}\label{inclusion R}
R_t \in 
\nabla f(X_t) + \partial_{\hat\epsilon_{\mathrm{c}}}h(X_t)+\mathcal{A}^{*}p_t \subset
\partial_{\hat\epsilon_{\mathrm{c}}}\phi(X_t)+\mathcal{A}^{*}p_t.
\end{equation}
It follows from relations \eqref{optimality inclusion} and \eqref{inclusion R} and the definition of $\hat\epsilon_{\mathrm{c}}$-subdifferential that
\begin{align*}
  \phi(X_t) - \phi(X_*) &\overset{\eqref{optimality inclusion}}{\geq} \inner{-\mathcal{A}^*p_*}{X_t - X_*}\\
\phi(X_*) - \phi(X_t) &\overset{\eqref{inclusion R}}{\geq} \inner{R_t -\mathcal{A}^*p_t}{X_* - X_t} - \hat\epsilon_{\mathrm{c}}.
\end{align*}
Adding the two above relations implies that
\begin{equation}\label{monotonicity}
  \inner{\mathcal{A}^{*}(p_{*}-p_t)}{X_t-X_{*}}\geq \inner{R_t}{X_{*}-X_t}-\hat\epsilon_{\mathrm{c}}.
\end{equation}
Relations \eqref{optimality inclusion}
and \eqref{monotonicity} then imply that
\[\inner{\mathcal{A}X_t-b}{p_{*}-p_t}\overset{\eqref{optimality inclusion}}{=}\inner{\mathcal{A}(X_t-X_{*})}{p_{*}-p_t}=\inner{X_t-X_{*}}{\mathcal{A}^{*}(p_{*}-p_t)}\overset{\eqref{monotonicity}}{\geq} \inner{R_t}{X_{*}-X_t}-\hat\epsilon_{\mathrm{c}},\]
from which the result immediately follows.
\end{proof}

\begin{lem}\label{lem:ALtelescopic}
For any iteration index $t$ of the AL framework, there holds:
\begin{equation}
\label{final inequality}
  \beta^2\sum_{l=1}^{t}\|\mathcal{A}X_l-b\|^2\leq \|p_{*}-p_0\|^2-\|p_{*}-p_t\|^2+2\beta t (D_h \hat\epsilon_{\mathrm{d}} + \hat\epsilon_{\mathrm{c}}).
\end{equation}
\end{lem}
\begin{proof}
Let $t$ be an iteration index of the AL framework and suppose $l\leq t$. By completing the square and using relation \eqref{dual update}, it follows that
\begin{align}\label{telescoping dual}
  \|p_*-p_{l-1}\|^2-\|p_*-p_l\|^2
  &=\|p_{l-1}-p_l\|^2 + 2 \inner{p_l-p_{l-1}}{p_*-p_l} \nonumber\\
& \overset{\eqref{dual update}}{=} \beta^2 \|\mathcal{A}X_l-b\|^2 + 
2 \beta \inner{\mathcal{A}X_l-b}{p_*-p_l}.
\end{align}
Moreover, relation \eqref{inner product bound}, the definition of $D_h$, the Cauchy-Schwarz inequality, and the fact that the Blackbox is called in step 1 with tolerance $\hat \epsilon_{d}$, imply that
    \begin{align}\label{final inequality needed}
      2\beta \inner{\mathcal{A}X_l-b}{p_{*}-p_l}
      \overset{\eqref{inner product bound}} \geq
      2\beta \inner{R_l}{X_{*}-X_l}-2\beta \hat\epsilon_{\mathrm{c}}
      \geq  -2\beta D_h \hat\epsilon_{\mathrm{d}} -2\beta \hat\epsilon_{\mathrm{c}}.
    \end{align}
Combining relations \eqref{telescoping dual} and \eqref{final inequality needed}, we then conclude  that
    \begin{align}\label{dual recursive}
      \|p_{*}-p_{l-1}\|^2-\|p_{*}-p_l\|^2
      \geq \beta^2\|\mathcal{A}X_l-b\|^2-2\beta D_h\hat\epsilon_{\mathrm{d}} -2\beta \hat\epsilon_{\mathrm{c}}.
    \end{align}
The conclusion of the lemma now follows by summing relation \eqref{dual recursive} from $l=1$ to~$t$.
\end{proof}
We are now ready to prove Theorem~\ref{thm:OuterComplexity}
\begin{proof}[Proof of Theorem~\ref{thm:OuterComplexity}]
(a) Let $t$ be an iteration index of the AL framework. The fact that the Blackbox AL is called in step 1 with inputs $g$ and $(\hat \epsilon_{\mathrm{c}}, \hat \epsilon_{\mathrm{d}})$ implies that its output $(X_t,R_t)$ satisfies that $\|R_t\|\leq \hat\epsilon_{\mathrm{d}}$ and also
  \begin{align*}
      R_t \in \nabla g(X_t) + \partial_{\hat\epsilon_{\mathrm{c}}} h(X_t)
      &= 
      \nabla f(X_t) +
      \mathcal{A}^*(p_{t-1} + \beta (\mathcal{A}X_t - b))
      + \partial_{\hat\epsilon_{\mathrm{c}}} h(X_t)
      \\
      &= \nabla f(X_t)
      + \partial_{\hat\epsilon_{\mathrm{c}}} h(X_t)
      + \mathcal{A}^*p_t.
  \end{align*}
  Since $\hat\epsilon_{\mathrm{c}} \leq \epsilon_{\mathrm{c}}$ and $\hat\epsilon_{\mathrm{d}} \leq \epsilon_{\mathrm{d}}$, it follows that
  \begin{equation*}
      R_t \in 
      \nabla f(X_t)
      + \partial_{\epsilon_{\mathrm{c}}} h(X_t) + \mathcal{A}^*(p_t),
      \qquad \|R_t\| \leq \epsilon_{\mathrm{d}}.
  \end{equation*}
Since the above relations hold for any iteration index $t$, the output $(X_{T}, p_{T}, R_{T})$ of the AL framework satisfies the first two relations in \eqref{eq:stationary-general}.
It remains to show that the AL framework terminates and that its last iteration index $T$ satisfies \eqref{Complexity Result}. Suppose by contradiction that the AL framework generates an iteration index $\hat t$ satisfying
\begin{equation}\label{AL contradiction}
\hat t> \left\lceil\frac{3\|p_{*}-p_0\|^2}{\beta^2\epsilon_{\mathrm{p}}^2}\right\rceil.
\end{equation}
In view of the stopping criterion of step 3 of the AL framework, this implies that $\|\mathcal AX_t-b\|>\epsilon_{\mathrm{p}}$ for every $t=1,\ldots \hat t-1$. Using this conclusion, relation \eqref{final inequality} with $t=\hat t-1$, and the definitions of $\hat \epsilon_{\mathrm c}$ and $\hat \epsilon_{\mathrm d}$ in \eqref{epsilon hat}, it follows that
\begin{align}\label{Al contradiction 2}
(\hat t-1)\epsilon^2_p&<\sum_{l=1}^{\hat t-1}\|\mathcal{A}X_l-b\|^2\overset{\eqref{final inequality}}{\leq} \frac{\|p_{*}-p_0\|^2}{\beta^2}+(\hat t-1) \frac{2\beta(D_h \hat\epsilon_{\mathrm{d}} + \hat\epsilon_{\mathrm{c}})}{\beta^2} \nonumber\\
&\overset{\eqref{epsilon hat}}{\leq} \frac{\|p_{*}-p_0\|^2}{\beta^2}+(\hat t-1)\frac{2\epsilon^2_p}{3}
\end{align}
which clearly contradicts the bound on $\hat t$ in \eqref{AL contradiction}.
Hence, in view of this conclusion and the termination criterion in step 3, the AL framework must terminate with final iteration index $T$ satisfying \eqref{Complexity Result} and output $(X_T,p_T,R_T)$ satisfying the third relation in \eqref{eq:stationary-general}.

(b) Let $t \leq T$ where $T$ is the final iteration index of the AL framework. It then follows from taking square root of relation \eqref{final inequality} and triangle inequality that
\begin{equation}\label{pt first bound}
\|p_t\|\overset{\eqref{final inequality}}{\leq} \|p_{*}\|+ \sqrt{\|p_{*}-p_0\|^2+2\beta t (D_h \hat\epsilon_{\mathrm{d}} + \hat\epsilon_{\mathrm{c}})}\leq \|p_{*}\|+\sqrt{\|p_{*}-p_0\|^2+2\beta T(D_h \hat\epsilon_{\mathrm{d}} + \hat\epsilon_{\mathrm{c}})}.
\end{equation}
The fact that $T$ satisfies relation \eqref{Complexity Result} and the definitions of $\hat \epsilon_{\mathrm{c}}$ and $\hat \epsilon_{\mathrm{d}}$ in \eqref{epsilon hat} then imply that
\begin{equation}\label{pt second bound}
2\beta T (D_h \hat\epsilon_{\mathrm{d}} + \hat\epsilon_{\mathrm{c}}) \overset{\eqref{Complexity Result}}{\leq} \frac{6(D_h\hat \epsilon_{\mathrm{d}}+\hat \epsilon_{\mathrm{c}})}{\beta \epsilon^2_{p}}\|p_{*}-p_0\|^2 +2\beta(D_h\hat \epsilon_{\mathrm{d}}+\hat \epsilon_{\mathrm{c}}) \overset{\eqref{epsilon hat}}{\leq} 2\|p_{*}-p_0\|^2+2\beta(D_h\hat \epsilon_{\mathrm{d}}+\hat \epsilon_{\mathrm{c}}).
\end{equation}
Relation \eqref{bound on Lagrange multipliers} then immediately follows from combining relations \eqref{pt first bound} and \eqref{pt second bound}.

It follows from relation \eqref{final inequality} with $t=T$ that
\begin{equation}\label{feasibility bound-1}
\beta^2\sum_{l=1}^{T}\|\mathcal{A}X_l-b\|^2 \overset{\eqref{final inequality}}{\leq}\|p_{*}-p_{0}\|^2+ 2\beta T(D_h \hat\epsilon_{\mathrm{d}} + \hat\epsilon_{\mathrm{c}}). 
\end{equation}
Combining relations \eqref{pt second bound} and \eqref{feasibility bound-1} then immediately implies inequality \eqref{bound on feasibility}.
\end{proof}

{\small 
\bibliographystyle{plain}
\bibliography{Proxacc_ref,refs}

\end{document}

\bibliographystyle{plain}
\bibliography{Proxacc_ref,refs}

\end{document}